\documentclass[11pt,a4paper,fleqn]{article}

\oddsidemargin=.\paperwidth
\topmargin=.\paperheight

\makeatletter
\newcommand\footnotetext@\relax
\let\footnotetext@\@footnotetext
\newcommand{\MSC}[2][2020]{%
 \unskip\protected@xdef\@thefnmark{}%
 \protect\footnotetext@{\kern-1.8em{\itshape MSC#1\spacefactor3000:}\/ #2}}
\newcommand{\keywords}[1]{%
 \unskip\protected@xdef\@thefnmark{}%
 \protect\footnotetext@{\kern-1.8em{\itshape Keywords\spacefactor3000:}\/ #1}}
\newcommand{\address}[1]{%
 \unskip\footnotemark
 \protected@xdef\@thanks{\@thanks\protect\footnotetext[\the\c@footnote]{{\itshape Address\spacefactor3000:}\/ #1}}}
\newcommand{\email}[1]{%
 \unskip\protected@xdef\@thanks{\@thanks\protect\footnotetext[0]{{\itshape Email\spacefactor3000:}\/ \texttt{#1}}}}
\renewcommand{\thanks}[1]{%
 \unskip\protected@xdef\@thanks{\@thanks\protect\footnotetext[0]{#1}}}
\gdef\@date{}
\renewcommand\l@section{\@dottedtocline{1}{0em}{1.5em}}
\renewcommand\l@subsection{\@dottedtocline{2}{1.5em}{2.3em}}
\makeatother

\usepackage[english]{babel}
\usepackage[utf8]{inputenc}
\usepackage{amsmath}%
 \allowdisplaybreaks[4]
\usepackage{amssymb,amsthm}%
 \theoremstyle{plain}
  \newtheorem{thm}{Theorem}[section]
  \newtheorem*{thm*}{Theorem}
  \newtheorem{cor}[thm]{Corollary}
  \newtheorem{lem}[thm]{Lemma}
  \newtheorem*{lem*}{Lemma}
  \newtheorem{prop}[thm]{Proposition}
 \theoremstyle{definition}
  \newtheorem{defn}[thm]{Definition}
  \newtheorem{exmp}[thm]{Example}
  \newtheorem*{exmps*}{Examples}
  \newtheorem{prob}[thm]{Problem}
 \theoremstyle{remark}
  \newtheorem{rem}[thm]{Remark}
  \newtheorem{rems}[thm]{Remarks}
  \newtheorem*{rem*}{Remark}
\usepackage{mathrsfs}
\usepackage[all]{xy}%
 \entrymodifiers={!!<.em,.6ex>+}
 \SelectTips{cm}{11}
\usepackage[pdftitle={},pdfauthor={}]{hyperref}

\DeclareMathOperator{\Ad}{Ad}
\newcommand{\an}{\sharp}
\DeclareMathOperator{\Aut}{Aut}
\newcommand{\blank}{\hphantom{b}}
\newcommand{\Der}{\mathit{D}}
\newcommand{\der}{\mathit{d}}
\DeclareMathOperator{\End}{End}
\newcommand{\freepi}{\pi^\mathrm{free}}
\newcommand{\ftimes}[2]{\mathbin{_{#1}\times_{#2}}}
\DeclareMathOperator{\GL}{GL}
\newcommand{\id}{\mathinner{\mathrm{id}}}
\usepackage{accents}
\newcommand{\idcomp}[1]{\accentset{\circ}{#1}}
\def\integral#1\der{\int#1\mathclose{\thinspace}\der}
\def\iintegral#1\der#2\der{\iint#1\mathclose{\thinspace}\der#2\thinspace\der}
\newcommand{\justify}[1]{\mathrel{\phantom=}#1\mathopen{\mkern\medmuskip}}
\newcommand{\Lie}{\mathfrak{L}}
\newcommand{\longto}{\longrightarrow}
\newcommand{\N}{\mathbb{N}}
\newcommand{\onto}{\twoheadrightarrow}
\newcommand{\pr}{\mathinner{\mathrm{pr}}}
\newcommand{\R}{\mathbb{R}}
\newcommand{\simto}{\overset\sim\to}
\DeclareMathOperator{\supp}{supp}
\newcommand{\T}{\mathit{T}}
\newcommand{\tto}{\rightrightarrows}
\newcommand{\xto}{\xrightarrow}
\newcommand{\Z}{\mathbb{Z}}

\begin{document}

\title{Proper Cartan Groupoids: Reduction to the Regular Case%
  \MSC[2010]{Primary 58H05, Secondary 53C05, 53C10, 58C15}
  \keywords{Proper Lie groupoid, multiplicative connection, regular groupoid, obstruction, recursive averaging, fast convergence}
}%
\author{Giorgio Trentinaglia%
  \address{Centro de Análise Matemática, Geometria e Sistemas Dinâmicos, %
      Ins\-ti\-tu\-to Su\-pe\-ri\-or Téc\-ni\-co, University of Lisbon, %
      Av.~Ro\-vis\-co Pais, 1049-001 Lisbon, Portugal}
  \email{gtrentin@math.tecnico.ulisboa.pt}
  \thanks{Some of the results in this article were obtained %
      while the author was a guest of the Max Planck Institute for Mathematics in Bonn, Germany. %
      The author acknowledges the support of the Portuguese Foundation for Science and Technology %
      through grants SFRH\slash BPD\slash 81810\slash 2011, UID\slash MAT\slash 04459\slash 2013, and %
      UID\slash MAT\slash 04459\slash 2020.}
}%
\maketitle

\begin{abstract}
We discuss a method for constructing multiplicative connections on proper Lie groupoids or, more exactly, for reducing the task of constructing such connections to a number of in principle simpler tasks involving only Lie groupoids that are both proper and regular.
\end{abstract}

\tableofcontents

\section*{Introduction}

The notion of multiplicative connection in the theory of Lie groupoids arguably traces back to the 1950s, being already implicit in the work of C.~Ehresmann on Cartan geometries (“es\-paces à {connexion} de Cartan”) \cite{Ehr50}. It is precisely because of the historical roots of this notion in É.~Cartan's views about differential geometry \cite{Sharpe} that multiplicative connections are still nowadays referred to as \emph{Cartan connections} in spite of the fact that they are currently looked at in a much broader context than Cartan or Ehresmann originally envisioned. Back in Ehresmann's times and until not too long ago much emphasis was given to the study of Cartan parallelisms on principal bundles. From the perspective of Lie groupoid theory the latter entities correspond, more or less exactly depending on one's point of view, to multiplicative connections on transitive Lie groupoids; a concise description of essential features of this correspondence is given in the introduction of \cite{Tre8}. In more recent years the transitive case has gradually lost its central status and the focus has shifted towards the general study of \emph{Cartan groupoids}---a name by which Lie groupoids equipped with multiplicative connections are known among specialists---and of even more general multiplicative structures (distributions etc.)\ on Lie groupoids and Lie algebroids \cite{Bl06,Bl12,Bla16,CSS,CrampS,JO,Tang,Yudil}.

It is a classical problem to determine the obstructions to the existence of Cartan parallelisms on principal bundles. This problem was already well present to Ehresmann, who was probably the first to observe that it could be reduced to a simpler question, namely, that of determining the obstructions to the existence of a soldering (“sou\-dure”) \cite{Ehr50}. The latter problem can in turn be formulated as one concerning the existence of $G$-structures, where $G$ is the structure group of the principal bundles which the parallelisms are supposed to be defined on \cite{Chern,Koba}. Although there are a number of techniques in the theory of fiber bundles that may be used to find necessary and sufficient conditions for a manifold to admit $G$-structures (obstruction theory, characteristic classes, cohomology operations, etc.)\ \cite{Huse,MilnS,Steen}, the existence problem remains a difficult one in general: this is already true when $G = \{1\}$, in which case the task becomes that of deciding whether or not a given manifold is parallelizable.

Once Lie groupoids are allowed into the picture it is virtually impossible not to think about the natural generalization of the above classical problem\spacefactor3000: Determine the obstructions to the existence of multiplicative connections on Lie groupoids. Of course when the Lie groupoids in question are transitive there are a number of criteria for existence that can be deduced at once from the known solutions to the classical problem. On the other hand when the groupoids are not transitive very little seems to be known in terms not only of practical results but also of theoretical methods. Given this state of affairs as well as the difficulty of the classical problem already, one may look at the general problem with pessimism. While there do indeed seem to be grounds for pessimism as far as \emph{arbitrary} Lie groupoids go, it can nevertheless be argued that a workable theory might exist provided only \emph{proper} Lie groupoids are taken into consideration. In the present article we are going to substantiate this view and, more concretely, to lay down what we consider to be the foundation stone for the theory we have in prospect: the method of \emph{reduction to the regular case.}

Our method introduces simplifications into the general problem which are analogous in many ways to the simplifications that Ehresmann's observation introduces into the classical problem. Such analogy is corroborated by the fact that our simplified general problem, much like Ehresmann's simplified classical problem, can be made to fit into the standard conceptual framework of (topological, equivariant) obstruction theory \cite{Tre8}. The analogy should not be pushed too far, however. Our analysis works for quite different (and much deeper) reasons than Ehresmann's, as witnessed by the fact that in ours the hypothesis of properness plays a crucial role which has no counterpart in the classical transitive case; examples can be given of Lie groupoids that are regular but neither proper nor transitive for which our general problem cannot possibly be reformulated in the way indicated in \cite{Tre8} as one amenable to obstruction-theoretic treatment.

Our method of reduction to the regular case is underpinned by a number of original results whose statements, together with a description of the method itself, may be found in section \ref{sec:results} of this paper. With the exception of section \ref{sec:preliminaries}, which recollects some indispensable background material and is intended mainly for the non-specialist, the rest of the paper is devoted to the demonstration of these results. Our arguments involve a combination of fairly standard ideas whose relevance for the study of proper Lie groupoids was first pointed out by A.~Weinstein in the context of his pioneering work on the linearization problem \cite{We00,Wein} and which directly or indirectly influenced later work on the same problem by other authors \cite{CS,Zung}. Among these ideas, the use of groupoid cohomology \cite{Crainic03,PS,Ren} and of (recursive) averaging \cite{delaHK,GKR,Wein00} plays a prominent role. Our practical implementation of the latter techniques takes the form of an averaging operator for groupoid connections (section \ref{sec:averaging}) which is a key ingredient of our theory and which, we believe, is also of independent interest: in retrospect, had Weinstein been aware of our averaging operator, he could well have used it to prove the linearization theorem he had conjectured in \cite{We00,Wein}---for, by a classical argument \cite{Bredon}, one can easily deduce such theorem from the local existence of multiplicative connections around fixed points---in a way that would perhaps have been closer in spirit to his own original approach than N.~T.~Zung's proof \cite{Zung}.

\section{Preliminaries}\label{sec:preliminaries}

The present section may be skipped, and referred back to only when needed, by those readers who are already familiar with the notion of multiplicative connection.

All manifolds in this paper are of class $C^\infty$, Hausdorff, and second countable. We allow them to be disconnected, but we do not allow their components to have different dimensions. We refer to such manifolds briefly as \emph{smooth manifolds.} If $f: P \to M$ is any $C^\infty$ map between smooth manifolds, we write $\der f: \T P \to f^*\T M$ for the morphism of smooth vector bundles over $P$ that corresponds to the tangent map $\T f: \T P \to \T M$; this is consistent with viewing $\der f$ as a $f^*\T M$~valued one-form on $P$. If $f$ is submersive, and if $b: S \to M$ is any other $C^\infty$ map of a smooth manifold into $M$, we write $S \ftimes{b}{f} P$ for the closed $C^\infty$ submanifold $\{(y,z) \in S \times P: b(y) = f(z)\}$ of the cartesian product $S \times P$.

We assume familiarity with the basic theory of Lie groupoids at the level say of Chapters 5 and 6 of \cite{MM}. For each arrow $g \in \varGamma$ in a groupoid $\varGamma \tto M$ we write $sg$,~$tg \in M$ for the source, respectively, the target, of $g$, and $g^{-1} \in \varGamma$ for its inverse. For any two arrows $g_1$,\ $g_2$ forming a composable pair in the sense that $sg_1 = tg_2$, we write $g_1g_2 = g_1 \cdot g_2 \in \varGamma$ for the composition, also called the product, of $g_1$ and $g_2$. We further write $1x \in \varGamma$ for the identity, also called the unit, at a base point $x \in M$. Whenever $\varGamma$ and $M$ are each given the structure of a smooth manifold in such a way that the source map $s: \varGamma \to M$ is both $C^\infty$ and submersive and the unit map $1: M \to \varGamma$, as well as the division map $\varGamma \ftimes{s}{s} \varGamma \to \varGamma$, $(g,h) \mapsto gh^{-1}$, is $C^\infty$, we call $\varGamma \tto M$ a \emph{Lie groupoid;} this notion is a bit more restrictive than the homonymous notion defined in \cite{MM} but is better suited for the needs of our study.

Let $E$ be a smooth vector bundle over the base $M$ of a Lie groupoid $\varGamma \tto M$. By a \emph{pseudo-representation} of $\varGamma \tto M$ on $E$ we mean an arbitrary morphism of smooth vector bundles over $\varGamma$ from $s^*E$ to $t^*E$ i.e.~an arbitrary global $C^\infty$ cross-section of the ``internal hom'' bundle $L(s^*E,t^*E) \to \varGamma$. To each arrow $g \in \varGamma$ a pseudo-representation $\lambda: s^*E \to t^*E$ assigns a linear map $\lambda_g: E_{sg} \to E_{tg}$, where $E_x$ denotes the fiber of $E$ at $x$. If $\lambda_{1x} = \id$ for all $x$, we say that $\lambda$ is \emph{unital.} If in addition $\lambda_{g_1g_2} = \lambda_{g_1}\lambda_{g_2}$ for all pairs of arrows $g_1$,~$g_2$ with $sg_1 = tg_2$, we say that $\lambda$ is a \emph{representation.}

By a \emph{connection} on the Lie groupoid $\varGamma \tto M$ we mean an arbitrary right splitting $\eta$ of the following short exact sequence of smooth vector bundles over $\varGamma$.
\begin{equation}
\xymatrix{%
 0 \ar[r]
 &	\ker\der s \ar[r]
	&	\T\varGamma \ar[r]^-{\der s}
		&	\ar@/^1pc/[l]^-{\eta} s^*\T M \ar[r]
			&	0}
\label{eqn:12B.8.1}
\end{equation}
Any such splitting is fully determined by its image, $H$, which is a smooth vector distribution on $\varGamma$, i.e.~a smooth vector subbundle of $\T\varGamma$, such that $\der s$ restricts to an isomorphism of $H$ onto $s^*\T M$. We may thus refer to the splitting $\eta$ or distribution $H$ indifferently as a connection on $\varGamma \tto M$; for expository purposes we choose to regard $H$ as fundamental and we refer to $\eta = (\der s \mathbin| H)^{-1}$ as the associated \emph{horizontal lift} $\eta^H$. For each arrow $g \in \varGamma$ the connection $H$ determines a linear map $\eta^H_g: \T_{sg}M \to \T_g\varGamma$ satisfying $\T_gs \circ \eta^H_g = \id$. When $\eta^H_{1x} = \T_x1$ for all $x$ in $M$, we say $H$ is \emph{unital.} Upon composing $\eta^H: s^*\T M \to \T\varGamma$ with $\der t: \T\varGamma \to t^*\T M$, we obtain a pseudo-representation of $\varGamma \tto M$ on $\T M$ which we call the \emph{effect} of $H$ and denote $\lambda^H$; our terminology is consistent with the standard terminology in use for {étale} groupoids \cite[pp.~136--137]{MM}. The relationship between connections $H$ and pseudo-representations $\lambda^H$ will be central to our theory. Clearly $\lambda^H$ must be unital if so is $H$, but not vice versa. From the obvious fact that $1(M)$ is a closed $C^\infty$ submanifold of $\varGamma$ and the existence of partitions of unity on $\varGamma$, it follows immediately that $\varGamma \tto M$ admits unital connections.

Generically speaking, the \emph{tangent groupoid} of $\varGamma \tto M$ is the Lie groupoid $\T\varGamma \tto \T M$ whose structure maps $\T s$,\ $\T t$,\ etc.~are obtained by differentiating those of $\varGamma \tto M$; more specifically, the tangent multiplication law owes its definability to the fact that the natural map of $\T(\varGamma \ftimes{s}{t} \varGamma)$ into $\T\varGamma \ftimes{\T s}{\T t} \T\varGamma$ is a diffeomorphism. A connection on $\varGamma \tto M$ is said to be \emph{multiplicative} whenever its vector distribution $H$ is a subgroupoid $H \tto \T M$ of the tangent groupoid $\T\varGamma \tto \T M$. The multiplicativity of $H$ may be expressed more conveniently in terms of the horizontal lift $\eta^H$ of $H$: a connection $H$ on $\varGamma \tto M$ is multiplicative if and only if it is unital and, for every composable pair of arrows $g_1$,~$g_2$ and tangent vector $v \in \T_{sg_2}M$, it satisfies the following identity, the product in the right-hand side of which is multiplication within $\T\varGamma \tto \T M$.
\begin{equation}
	\eta^H_{g_1g_2}v = \eta^H_{g_1}\lambda^H_{g_2}v \cdot \eta^H_{g_2}v
\label{prop:12B.9.4}
\end{equation}
The proof is straightforward. Upon applying the tangent map $\T_{g_1g_2}t$ to both sides of this identity one deduces at once that the effect $\lambda^H$ of a multiplicative connection $H$ is a representation. Although the condition \eqref{prop:12B.9.4} does not by itself imply multiplicativity, it comes quite close to it: indeed, dividing both members of the identity $\eta^H_{1x}v = \eta^H_{1x}\lambda^H_{1x}v \cdot \eta^H_{1x}v$ in $\T\varGamma \tto \T M$ from the right by $\eta^H_{1x}v$ yields $(\T_x1)\lambda^H_{1x}v = \eta^H_{1x}\lambda^H_{1x}v$, whence provided $\lambda^H_{1x}$ is \emph{onto,} $\T_x1 = \eta^H_{1x}$. It will be convenient to let $\eta^H_{g_1}\lambda^H_{g_2} \cdot \eta^H_{g_2}$ denote the linear map $\T_{sg_2}M \to \T_{g_1g_2}\varGamma$ defined by the right-hand side of \eqref{prop:12B.9.4}.

While unital connections do commonly exist in profusion on any Lie groupoid, multiplicative connections are mathematical objects of a much more elusive type.

\begin{exmps*} \textit{(a)\spacefactor3000} Let the Lie group $G$ act on the smooth manifold $M$ from the left in a $C^\infty$ fashion. Let $g$ denote both an element of $G$ and the transformation of $M$ given by $x \mapsto gx$. By definition, the \emph{action groupoid} $G \times M \tto M$ has source map $(g,x) \mapsto x$, target map $(g,x) \mapsto gx$, and multiplication law $(g',gx) \cdot (g,x) = (g'g,x)$. The vector distribution $H$ on $G \times M$ tangent to the fibers of the projection $G \times M \to G$ is a connection on $G \times M \tto M$ with effect $\lambda^H_{g,x} = \T_xg$. One recognizes without difficulty that $H$ is multiplicative.

\textit{(b)\spacefactor3000} Any Lie groupoid $\varGamma \tto M$ whose source map and target map coincide may be thought of as a family of Lie groups, the \emph{isotropy groups} of $\varGamma \tto M$, parameterized ``smoothly'' by the points of $M$. We call any such $\varGamma \tto M$ a \emph{Lie group bundle.} We contend that if every isotropy group is both abelian and compact there can be at most one multiplicative connection on $\varGamma \tto M$; in particular, the action groupoid $G \times M \tto M$ associated with any compact abelian $G$ operating trivially on $M$ admits exactly one multiplicative connection (the one described in the previous example). To prove it, we start by noting that for every multiplicative connection $H$ the $H$-horizontal lift of any vector field on $M$ is a vector field $\xi$ on $\varGamma$ which is \emph{multiplicative} in the sense that $\xi(g_1g_2) = \xi(g_1) \cdot \xi(g_2)$ for all $g_1$,\ $g_2$ in the same isotropy group. For any two such connections $H_0$ and $H_1$ the difference between the $H_0$-horizontal lift $\xi_0$ of a vector field on $M$ and its $H_1$-horizontal lift $\xi_1$ is a multiplicative vector field $\xi_0 - \xi_1$ with values in $\ker\der s$. It will then suffice to show that under the stipulated hypotheses any such vector field is necessarily zero. Now, its restriction to each isotropy group $G$ may be viewed as a multiplicative vector field on $G$. The restriction to $\mathfrak{g} = \T_1G$ of the multiplication law of the tangent group $\T G$ is given by $X \cdot Y = X + Y$.%
\footnote{%
\label{npar:12B.9.3}%
 This identity is valid in $\T\varGamma \tto \T M$ for any Lie groupoid $\varGamma \tto M$ for all $X \in \ker\T_{1x}s$, $Y \in \ker\T_{1x}t$: indeed, by the linearity of tangent multiplication, $X \cdot Y = (X + 0_{1x}) \cdot (0_{1x} + Y) = X \cdot 0_{1x} + 0_{1x} \cdot Y = X + Y$ because $0_{1x} = (\T_x1)0_x$ is the unit tangent arrow at $0_x \in \T_xM$.
} %
Upon setting $X(g) = \xi(g) \cdot 0_{g^{-1}}$ we obtain a linear one-to-one correspondence between vector fields $\xi$ on $G$ and $\mathfrak{g}$-valued $C^\infty$ functions $X$ on $G$. It follows at once from the commutativity of ($G$ and hence) $\T G$ that $\xi$ is multiplicative iff the corresponding $X$ is \emph{additive:} $X(g_1g_2) = X(g_1) + X(g_2)$. But then, since $G$ is compact, the sequence $n \mapsto nX(g) = X(g^n)$ has to be bounded, which cannot happen unless $X(g) = 0$.

\textit{(c)\spacefactor3000} Any Lie groupoid $\varGamma \tto M$ whose \emph{anchor map} $(t,s): \varGamma \to M \times M$ is both surjective and submersive%
\footnote{%
 Actually, the submersiveness of the anchor map follows from its surjectivity; one does not need to postulate it. See the appendix of \cite{Tre6} for a simple proof of this fact.
} %
is said to be \emph{transitive.} These Lie groupoids lie figuratively speaking at the very antipodes of the Lie group bundles mentioned in the previous example. The simplest transitive Lie groupoid over a given $M$ is the \emph{pair groupoid:} $M \times M \tto M$ with anchor map the identity. The multiplicative connections on $M \times M \tto M$ are easily recognized to correspond one-to-one to the global trivializations of the tangent bundle $\T M$, so that $M \times M \tto M$ admits multiplicative connections iff $M$ is parallelizable. The pair groupoid over the two-sphere $S^2$ constitutes the simplest example of a Lie groupoid admitting no multiplicative connections. \end{exmps*}

\subsection{The extension problem}\label{sub:problem}

Given any $C^\infty$ map $b: S \to M$ of a smooth manifold into the base of a Lie groupoid $\varGamma \tto M$ we may ask whether the pullback of $\varGamma \tto M$ along $b$, i.e., the groupoid $S \ftimes{b}{t} \varGamma \ftimes{s}{b} S \tto S$, is a Lie groupoid in its own right. Such will evidently be the case whenever the map of $\varGamma \ftimes{s}{b} S$ into $M$ that sends $(g,y)$ to the target of $g$ is a submersion. When this happens, we say $b$ is \emph{transversal.} If $\varGamma \ftimes{s}{b} S \to M$ is also onto, we say $b$ is \emph{completely} transversal. Any transversal inclusion of a submanifold is called a \emph{slice.} The pullback of $\varGamma \tto M$ along a slice $S \subset M$ may be identified with the restriction of $\varGamma \tto M$ over $S$, i.e., the subgroupoid $\varGamma \mathbin| S = s^{-1}(S) \cap t^{-1}(S) \tto S$ of $\varGamma \tto M$ consisting of all arrows whose source and target lie in $S$. We shall need the notion of transversality in section \ref{sec:results} in conjunction with the following criterion, which, for future reference, we dub the ``Morita trick\spacefactor3000:'' \em Let\/ $Z$ be a subset of\/ $M$ which is\/ \emph{invariant} in the sense that it coincides with its own\/ \emph{saturation} $\varGamma Z = t\bigl(s^{-1}(Z)\bigr)$. Let\/ $b: S \to M$ be a completely transversal map. Then, $Z$ is a\/ $C^\infty$ submanifold of\/ $M$ iff its inverse image under\/ $b$ is a\/ $C^\infty$ submanifold, $b^{-1}(Z)$, of\/ $S$. \em The proof is easy and runs as follows: $Z$ is a submanifold of $M$ iff its preimage $t^{-1}(Z) \ftimes{s}{b} S$ under the surjective submersion $\varGamma \ftimes{s}{b} S \to M$ is a submanifold of $\varGamma \ftimes{s}{b} S$; by the invariance of $Z$, this preimage coincides with $\varGamma \ftimes{s}{b} b^{-1}(Z)$, the preimage of $b^{-1}(Z)$ under the projection $\varGamma \ftimes{s}{b} S \to S$; the latter map is the pullback of $s$ along $b$, and hence is itself a surjective submersion.

Apart from slices, which play a marginal role in this paper, there are other submanifolds of $M$ over which $\varGamma \tto M$ restricts ``nicely'' which are contextually far more important for our theory of connections. Let $Z$ be a submanifold of $M$ of class $C^\infty$ which is \emph{locally invariant} in the sense that it can be covered with open subsets $U$ of $M$ satisfying $U \cap \varGamma Z \subset Z$. Of course the union of all such $U$ is itself one such subset so there is actually some $U$ for which $U \cap \varGamma Z = Z$ in other words $\varGamma \mathbin| Z = s^{-1}(Z) \cap t^{-1}(U)$. $\varGamma \mathbin| Z$ is therefore open inside the $C^\infty$ submanifold $s^{-1}(Z)$ of $\varGamma$ and we have that $\varGamma \mathbin| Z \tto Z$ is a \emph{Lie} subgroupoid of $\varGamma \tto M$.%
\footnote{%
\label{ftn:Z}%
 Strictly speaking according to our conventions $\varGamma \mathbin| Z \tto Z$ may fail to be a Lie groupoid in that different invariant components of $Z$ may have different dimensions. We shall nevertheless refer to it improperly as a Lie groupoid as a convenient way of collectively referring to the Lie groupoids $\varGamma \mathbin| S \tto S$ that one obtains by letting $S$ vary over all invariant components of $Z$; in this respect one should note that by the existence of local bisections of $\varGamma \tto M$ \cite[p.~115]{MM} and by the local invariance of $Z$ one has $\dim\T_{sg}Z = \dim\T_{tg}Z$ for all $g \in \varGamma \mathbin| Z$, so that every such $S$ is indeed a submanifold of definite, overall constant, dimension.
} %
Any connection $H$ on $\varGamma \tto M$ restricts to a connection on $\varGamma \mathbin| Z \tto Z$ which we designate $H \mathbin| Z$ and refer to as the connection \emph{induced} by $H$ \emph{along} $Z$, whose horizontal lift at $g \in \varGamma \mathbin| Z$ is
\begin{equation}
	\eta^{H|Z}_g = \eta^H_g \mathbin| \T_{sg}Z: \T_{sg}Z \longto (\T_gs)^{-1}(\T_{sg}Z) = \T_g(\varGamma \mathbin| Z).
\label{eqn:H|Z}
\end{equation}
(In particular $H$ will induce a connection in this way along any open subset or invariant submanifold of $M$.) The effect of $H \mathbin| Z$ is simply the restriction of the effect of $H$ to $\T Z$.
\begin{equation}
	\lambda^{H|Z}_g = \lambda^H_g \mathbin| \T_{sg}Z: \T_{sg}Z \longto \T_{tg}Z
\label{eqn:H|Z*}
\end{equation}
It is immediate that $\lambda^{H|Z}$ must be a representation whenever so is $\lambda^H$ and that $H \mathbin| Z$ must be multiplicative whenever so is $H$.

Although our focus in this paper will be on multiplicative connections, we shall also need to consider connections that are not multiplicative. In order to make the distinction easier to track, we shall now start employing the variables $\varPhi$ and $\varPsi$ for \emph{multiplicative} connections while retaining our use of the variable $H$ for those connections whose multiplicativity is \emph{not} postulated or granted.

\begin{prob}\label{prob:ext} Let $\varGamma \tto M$ be a Lie groupoid. Let $U$ and $V$ be open subsets of $M$ such that $\overline{U}$, the closure of $U$ in $M$, is contained in $V$. Let $\varPhi$ be a multiplicative connection on $\varGamma \mathbin| V \tto V$. What are the obstructions to extending $\varPhi \mathbin| U$ to a multiplicative connection defined on all of $\varGamma \tto M$? \end{prob}

The role of the larger open set $V \supset \overline{U}$ of course is simply to prevent the problem from failing to be solvable for trivial reasons: since the class of unital connections is closed under taking locally finite, convex, linear combinations, a straightforward argument involving partitions of unity yields the existence of a unital connection $H$ on $\varGamma \tto M$ such that $H \mathbin| U = \varPhi \mathbin| U$; the question, then, is whether among all such $H$ we can find one which is also multiplicative. It is instructive to inquire about the most general circumstances under which a partition of unity argument can be used to glue a family of partially defined multiplicative extensions of $\varPhi \mathbin| U$ into one global such extension. Namely, let us suppose that $M$ can be covered with open \emph{invariant} sets $V_i = \varGamma V_i$ such that for each $i$ there is a multiplicative connection $\varPhi_i$ on $\varGamma \mathbin| V_i \tto V_i$ which agrees with $\varPhi$ along $U \cap V_i$, and that the \emph{effects} of $\varPhi_i$ and $\varPhi_{i'}$ \emph{coincide} along $V_i \cap V_{i'}$ for all $i$,\ $i'$. Let us further suppose that there exist real-valued functions $c_j \in C^\infty(M)$ with $\supp c_j \subset V_{i(j)}$ which are \emph{invariant} in the sense that $c_j \circ s = c_j \circ t$ and which satisfy $\sum c_j = 1$ (locally finite sum). Then $\sum{}(c_j \circ s)\eta^{\varPhi_{i(j)}}$ is evidently the horizontal lift for a new connection on $\varGamma \tto M$ which extends $\varPhi \mathbin| U$ and which, by the indicated assumptions and by the linearity of tangent multiplication, satisfies \eqref{prop:12B.9.4} and thus is \emph{multiplicative.}

The special case of Problem \ref{prob:ext} where one takes $U = V = \emptyset$ is the question which originally motivated the investigations whose results form the subject of this paper:

\begin{prob}\label{prob:ext*} What are the obstructions to the existence of multiplicative connections on $\varGamma \tto M$? \end{prob}

There are however at least two good reasons for us to consider the more general relative version of Problem \ref{prob:ext*} formulated as Problem \ref{prob:ext}. The first reason is that, as we shall explain in the next section, the latter problem happens to play an essential role in our method for analyzing the obstructions to the existence of multiplicative connections on proper Lie groupoids (\emph{reduction to the regular case}). The second reason is that Problem \ref{prob:ext} includes, as special cases, other questions which we wish to answer about multiplicative connections which are interesting in their own right. Among such questions, the following one will prove to be particularly relevant from the viewpoint of our analysis:

\begin{prob}\label{prob:ext**} Let $\varPhi_0$,\ $\varPhi_1$ be multiplicative connections on $\varGamma \tto M$. What are the obstructions to deforming $\varPhi_0$ into $\varPhi_1$ smoothly through multiplicative connections? \end{prob}

Problem \ref{prob:ext**} is only one of the several instances of Problem \ref{prob:ext} which result from replacing $\varGamma \tto M$ with $\varGamma \times S \tto M \times S$, the product of $\varGamma \tto M$ by the unit groupoid $S \tto S$ over an arbitrary smooth manifold $S$, and from taking each one of $U$,\ $V$ to be the product of $M$ by an open subset of $S$. Any connection $H$ on $\varGamma \times S \tto M \times S$ whose horizontal lift is of the form $\eta^H_{g,y} = \eta^{H_y}_g \times \id: \T_{sg}M \times \T_yS \to \T_g\varGamma \times \T_yS$ ($g \in \varGamma$, $y \in S$), where $H_y$ denotes the connection on $\varGamma \tto M$ induced by $H$ along the invariant submanifold $M \simeq M \times \{y\} \subset M \times S$, may be viewed as a family of connections $H_y$ on $\varGamma \tto M$ parameterized smoothly by the points $y$ of $S$. Clearly $H$ will be multiplicative iff so is every $H_y$. More generally, whether or not of the indicated form, $H$ will give rise to one such parametric family, in which every $H_y$ will be multiplicative when so is $H$. With this understood, Problem \ref{prob:ext**} corresponds to taking $S = \R$, $H_y = \varPhi_0$ for $y < 0$, and $H_y = \varPhi_1$ for $y > 1$.

\section{Statement of main results. Applications}\label{sec:results}

Our strategy for the analysis of Problems \ref{prob:ext*} and \ref{prob:ext**} works on condition that the Lie groupoid $\varGamma \tto M$ be \emph{proper} in other words all compact sets in $M \times M$ have compact inverse image under the anchor map $(t,s): \varGamma \to M \times M$: we shall therefore require $\varGamma \tto M$ to be proper throughout the rest of the paper.

Following \cite{Ren,Tu} we shall make use of the auxiliary notations $\varGamma_x = s^{-1}(x)$, $\varGamma^x = t^{-1}(x)$, and $\varGamma_x^x = \varGamma_x \cap \varGamma^x$ for the source fiber, the target fiber, and the isotropy group of $\varGamma \tto M$ at a point $x \in M$. Following \cite{MM} we shall write $\varGamma x = t(\varGamma_x)$ for the orbit of $x$. As before, $\varGamma S = t\bigl(s^{-1}(S)\bigr)$ denotes the invariant saturation of a subset $S \subset M$, and $\varGamma \mathbin| S = s^{-1}(S) \cap t^{-1}(S)$. We begin by reviewing a couple of basic facts about the orbit structure of $\varGamma \tto M$.

It is standard knowledge that, because of properness, the orbit $\varGamma x$ through any base point $x$ is a closed submanifold of $M$ of class $C^\infty$; a proof can be found e.g.~in \cite{Wein}. This submanifold is also of definite (i.e.~overall constant) dimension, essentially because $M$ is.%
\footnote{%
 This might no longer be true if $M$ were allowed to contain components of different dimensions. Think of the pair groupoid $M \times M \tto M$ as an obvious counterexample.
} %
Then, for each integer $r \geq 0$, the set
\begin{equation}
	M_r = \{x \in M: \dim\varGamma x = r\}
\label{eqn:12B.19.2}
\end{equation}
of all base points lying on $r$-dimensional orbits is invariant. The union $\bigcup_{q=0}^r M_q$ is a closed subset of $M$, its complement being the set of all those base points where the rank of the vector bundle morphism $1^*\ker\der s \xto{1^*\der t} 1^*t^*\T M \simeq \T M$ is at least $r + 1$. Because of properness, every $M_r$ must be an invariant submanifold of $M$ of class $C^\infty$: for $r = 0$, this is a direct consequence of the linearization theorem \cite{CS,FdH,Zung}, since for a linear action of a Lie group the union of all zero-dimensional orbits is the linear subspace consisting of all those vectors that under the infinitesimal counterpart of the action are annihilated by every Lie algebra element; the case $r > 0$ follows by cutting through each point of $M_r$ a slice of dimension complementary to $r$ contained in the (open, invariant) complement of $\bigcup_{q=0}^{r-1} M_q$ and then using the ``Morita trick'' mentioned at the beginning of subsection \ref{sub:problem}. The reader should bear in mind that different invariant components of $M_r$ may have different dimensions.%
\footnote{%
 See footnote \ref{ftn:Z} on page \pageref{ftn:Z}.}

We proceed to describe our main results, pointing out their relevance for the analysis of the problems discussed in subsection \ref{sub:problem}. We start with the one which has the simplest formulation:

\begin{thm}\label{prop:12B.11.7} Let\/ $H$ be a connection on the proper Lie groupoid\/ $\varGamma \tto M$ whose effect\/ $\lambda^H = \der t \circ \eta^H$ is a representation. There exists a multiplicative connection on\/ $\varGamma \tto M$ with the same effect as\/ $H$ that agrees with\/ $H$ along every invariant\/ $C^\infty$ submanifold of\/ $M$ along which\/ $H$ already induces a multiplicative connection. \end{thm}

The proof will be given in section \ref{sec:averaging}. As a first illustrative application of this result to the existence problem \ref{prob:ext*}, suppose $\varGamma \tto M$ is a proper Lie group bundle: in this case the target map equals the source map so for any choice of $H$ the condition on $H$ in our theorem is satisfied in the tautological form $\lambda^H = \id$ and we conclude that there is at least one multiplicative connection. We could of course have arrived at the same conclusion by combining the linearization theorem of \cite{Wein}, which implies the local existence of multiplicative connections, with the gluing argument described in the paragraph following Problem \ref{prob:ext}; however, we shall soon meet applications where using our theorem appears to be the most viable, if not the only, option. As to the deformation problem \ref{prob:ext**}, if we are given any two multiplicative connections $\varPhi_0$ and $\varPhi_1$ then the connections $H_\epsilon$ defined for all $\epsilon$ in $\R$ by $\eta^{H_\epsilon} = (1 - \epsilon)\eta^{\varPhi_0} + \epsilon\eta^{\varPhi_1}$ provide a $C^\infty$ one-parameter family deforming $\varPhi_0$ into $\varPhi_1$, and a moment's reflection shows that every $H_\epsilon$ satisfies \eqref{prop:12B.9.4} and hence is multiplicative. In general, on an \emph{arbitrary} Lie groupoid, the deformation problem can be solved in this way provided the effects of $\varPhi_0$ and $\varPhi_1$ coincide.

Our next result is a far-reaching generalization of the previous one. It has not only a broader spectrum of applications but also a deeper significance, as witnessed by the fact that its proof, which will occupy us for most of sections \ref{sec:estimates} and \ref{sec:proof}, requires a conspicuous amount of extra work in addition to what is needed to prove Theorem \ref{prop:12B.11.7}. Still, as its statement is somewhat technical, it seems a good idea to precede it with the simpler Theorem \ref{prop:12B.11.7}, also because the latter has its own importance, and we shall need to refer back to it more than once.

\begin{thm}\label{thm:main} Let\/ $H$ be a connection on the proper Lie groupoid\/ $\varGamma \tto M$. Let\/ $S$ be an invariant subset of\/ $M$ over which the effect\/ $\lambda^H$ of\/ $H$ is a representation in the sense that\/ $\lambda^H_{1x} = \id$ for all\/ $x$ in\/ $S$ and\/ $\lambda^H_{g_1g_2} = \lambda^H_{g_1}\lambda^H_{g_2}$ for all\/ $g_1$,~$g_2 \in \varGamma \mathbin| S$ for which\/ $sg_1 = tg_2$. Then, over some open neighborhood\/ $V$ of\/ $\overline{S}$, there exists for any choice of an open subset\/ $B$ of\/ $V$ such that\/ $\varGamma B \supset V$ a multiplicative connection\/ $\varPhi$ which satisfies\/ $\lambda^\varPhi_g = \lambda^H_g$ for all\/ $g$ in\/ $\varGamma \mathbin| S$ and which has the property that, for every locally invariant\/ $C^\infty$ submanifold\/ $Z$ of\/ $V$ such that\/ $B \cap \varGamma Z \subset Z$ and such that the connection\/ $H \mathbin| Z$ induced by\/ $H$ on\/ $\varGamma \mathbin| Z \tto Z$ is multiplicative, $\varPhi \mathbin| Z$ equals\/ $H \mathbin| Z$. \end{thm}

We shall see in subsection \ref{sub:deducing1} that, among the consequences of Theorem \ref{thm:main}, there is the following local extension principle for multiplicative connections, which is one of the theoretic pillars on which we shall base our approach to Problems \ref{prob:ext*} and \ref{prob:ext**}:

\begin{thm}\label{prop:14A.5.2} Let\/ $\varGamma \tto M$ be a proper Lie groupoid. Let\/ $C$ be a closed invariant subset of\/ $M$, let\/ $V$ be an open neighborhood of\/ $C$, and let\/ $Z$ be an invariant submanifold of\/ $M$ of class\/ $C^\infty$. Suppose that\/ $\varPhi$ is a multiplicative connection on\/ $\varGamma \mathbin| V \tto V$ such that the induced connection on\/ $\varGamma \mathbin| V \cap Z \tto V \cap Z$ can be extended to a multiplicative connection\/ $\varPsi$ on\/ $\varGamma \mathbin| Z \tto Z$. Then, over some open neighborhood\/ $V'$ of\/ $C \cup Z$, there exists a multiplicative connection which induces\/ $\varPsi$ along\/ $Z$ and agrees with\/ $\varPhi$ over some open neighborhood of\/ $C$ within\/ $V \cap V'$. \end{thm}

Notice that Problems \ref{prob:ext*} and \ref{prob:ext**} are both instances of Problem \ref{prob:ext} with $U$ \emph{invariant.} For such $U$, there is a method for the resolution of Problem \ref{prob:ext}, based on Theorem \ref{prop:14A.5.2}, which goes roughly as follows.

As our initial step, we try to apply Theorem \ref{prop:14A.5.2} to the situation where $C = \overline{U}$ and $Z = M_0$. It turns out that, perhaps at the expense of shrinking $V$ around $\overline{U}$ a little bit, it is always possible to find some multiplicative connection on $\varGamma \mathbin| M_0 \tto M_0$ which extends $\varPhi \mathbin| V \cap M_0$. Although a full justification of this assertion will become possible only in section \ref{sec:averaging} after we have discussed our averaging operator for groupoid connections and its basic properties (cf., to wit, our ``supplementary remarks'' at the end of subsection \ref{sub:proof2}), we can be slightly less mysterious if we are willing to make the simplifying assumption that $V$ contains some invariant open neighborhood of $\overline{U}$; this is often the case in practice and if not, it might still be possible to achieve this by shrinking $U$ itself a bit. If we assume this, our assertion descends from Theorem \ref{prop:12B.11.7} in a rather straightforward way:

\begin{cor}\label{thm:rank=0} Let\/ $\varGamma \tto M$ be a proper Lie groupoid. Suppose that it is regular of rank zero in other words that\/ $M_0 = M$. Let\/ $U$,\ $V$,\ and\/ $\varPhi$ be as in Problem\/ \textup{\ref{prob:ext}.} Assume that\/ $U$ is invariant. Then, there exists a multiplicative connection on\/ $\varGamma \tto M$ that extends\/ $\varPhi \mathbin| U$. Such an extension is unique up to smooth deformation through multiplicative connections that agree with\/ $\varPhi$ over\/ $U$. \end{cor}

\begin{proof} $M_0 = M$ implies $\ker\der s = \ker\der t$, hence for any $g$ in $\varGamma$ the linear map $\T_gt \circ \eta_g$ is the same for all splittings $\eta_g$ of $\T_gs$. Let us consider an arbitrary connection $H$ on $\varGamma \tto M$. For every pair of arrows $g_1$,~$g_2$ with $sg_1 = tg_2$, both $\eta^H_{g_1g_2}$ and $\eta^H_{g_1}\lambda^H_{g_2} \cdot \eta^H_{g_2}$ are splittings of $\T_{g_1g_2}s$, thus $\lambda^H$ must be a representation, the same for all $H$. The existence of partitions of unity on $\varGamma$ implies that we can find some $H$ for which $H \mathbin| U = \varPhi \mathbin| U$. Theorem \ref{prop:12B.11.7} then says that there must be some globally defined multiplicative extension of $\varPhi \mathbin| U$, whose uniqueness up to smooth deformations follows by an argument identical to the argument that we used in the special case of Lie group bundles. \end{proof}

As far as we can tell, already this simple corollary to Theorem \ref{prop:12B.11.7} does not follow from any known linearization theorem, at least, not in the same obvious way it does in the special case of Lie group bundles: for the same argument to work, it would seem necessary to have \emph{invariant} linearizability, a property that not all $\varGamma \tto M$ of interest enjoy.

Back to our method, from the preceding considerations and Theorem \ref{prop:14A.5.2} it follows that we can always extend $\varPhi \mathbin| U$ to a multiplicative connection defined over an open neighborhood of $\overline{U} \cup M_0$. We proceed by induction on $r \geq 1$. Suppose that for some open neighborhood $V$ of the closed invariant set $C = \overline{U} \cup \bigcup_{q=0}^{r-1} M_q$ we have been able to find some prolongation of $\varPhi \mathbin| U$ by a multiplicative connection $\varPhi$ defined on $\varGamma \mathbin| V \tto V$. We may then try to see whether, perhaps at the expense of shrinking $V$ around $C$ a bit, we can extend $\varPhi \mathbin| V \cap M_r$ to a multiplicative connection on $\varGamma \mathbin| M_r \tto M_r$. This amounts to trying to solve a problem of type \ref{prob:ext} for the proper \emph{regular} groupoid $\varGamma \mathbin| M_r \tto M_r$. If successful, we can invoke Theorem \ref{prop:14A.5.2} with $Z = M_r$ in order to prolong $\varPhi \mathbin| U$ further over some open neighborhood of $C \cup Z = \overline{U} \cup \bigcup_{q=0}^r M_q$, thus completing the inductive step.

In essence, this is the content of the idea that, for proper Lie groupoids, the resolution of Problem \ref{prob:ext} can be ``reduced to the regular case.'' The reader, at this point, may be wondering what the practical relevance of such ``reduction to the regular case'' is. Another natural question that may cross the reader's mind is, to what extent do the arbitrary choices involved at each step of the above inductive method influence the outcome, or even the feasibility, of the method itself? Of these two questions the first one is easier to answer and will be addressed presently; consideration of the second question will be postponed to subsection \ref{sub:algorithm}.

\subsection{The extension problem in the regular case}\label{sub:regular}

Within the scope of the present discussion we assume that our Lie groupoid $\varGamma \tto M$, besides being proper, is also \emph{regular:} thus $M = M_r$ for some $r$. We refer to $r$, the common dimension of the orbits of $\varGamma \tto M$, as the \emph{rank} of $\varGamma \tto M$. Because of regularity, the vector bundle morphism $1^*\ker\der s \xto{1^*\der t} 1^*t^*\T M \simeq \T M$ has constant rank $r$. Its image is therefore a smooth subbundle $L$ of $\T M$ of rank $r$, which we refer to as the \emph{longitudinal bundle} of $\varGamma \tto M$. It is the vector distribution on $M$ spanned by all directions tangent to orbits.

From our discussion of the connections induced along invariant submanifolds as applied to orbits, it readily follows that the effect $\lambda^H = \der t \circ \eta^H$ of any groupoid connection $H$ on $\varGamma \tto M$ must carry the longitudinal distribution $L \subset \T M$ into itself. Upon restriction, it must therefore give rise to a pseudo-representation of $\varGamma \tto M$ on $L$, which we call the \emph{longitudinal effect} of $H$. The longitudinal effect of any multiplicative connection is an example of a \emph{longitudinal representation,} by which we mean a representation of $\varGamma \tto M$ on its own longitudinal bundle $L$. The relevance of the latter notion is illustrated by our next result, whose proof will be given in section \ref{sec:averaging}.

\begin{thm}\label{thm:regular} Let the Lie groupoid\/ $\varGamma \tto M$ be both proper and regular. Let\/ $H$ be any connection on\/ $\varGamma \tto M$ that is multiplicative along the open and invariant subset\/ $U$ of\/ $M$. Then, $H \mathbin| U$ is prolongable globally as a multiplicative connection if, and only if, its longitudinal effect extends to the whole\/ $\varGamma \tto M$ as a longitudinal representation. \end{thm}

This result is substantially a corollary of Theorem \ref{prop:12B.11.7}. The conclusions of Corollary \ref{thm:rank=0} may be understood in the light of this result: when the rank is zero, the longitudinal bundle $L$ is the zero subbundle of $\T M$ and thus the extension problem for longitudinal representations can always be solved (trivially). Another interesting consequence of Theorem \ref{thm:regular} is the following. Suppose that the longitudinal bundle of $\varGamma \tto M$ is trivializable, i.e., there is an isomorphism $\tau: M \times \R^r \simto L$ of smooth vector bundles over $M$. Since $g \mapsto \tau_{tg} \circ \tau_{sg}^{-1}$ is a longitudinal representation of $\varGamma \tto M$, the theorem enables us to conclude that $\varGamma \tto M$ admits multiplicative connections. Every $\varGamma \tto M$ that is transitive over a parallelizable $M$ falls in particular within the scope of such conclusion. In fact, it can be shown quite independently of Theorem \ref{thm:regular} that, whether or not a given \emph{transitive} Lie groupoid $\varGamma \tto M$ is proper, every tangent representation $s^*\T M \simto t^*\T M$ of $\varGamma \tto M$ is the effect of some multiplicative connection; a concise account of this result---which, essentially, was already known to Ehresmann \cite[p.~43]{Ehr50}---is contained in the introduction of \cite{Tre8}. Our theorem may be viewed as a generalization of the latter result to the \emph{intransitive} case, valid under an additional assumption of properness which, as we shall see, cannot be dispensed with.

As suggested by our preceding remarks, the practical utility of Theorem \ref{thm:regular} lies in that it reduces the extension problem for multiplicative connections to a corresponding problem for longitudinal representations whose resolution is in many cases much simpler. In general, the point of view elaborated in \cite{Tre8} enables a reformulation of the extension problem for longitudinal representations as a \emph{standard} problem in (topological, equivariant) \emph{obstruction theory.} The usual methods and techniques of that theory are thus available to us for the study of our extension problem. We shall review a couple of simple illustrative applications of these methods in the course of the present subsection and the next.

We warn the reader that the close relationship between multiplicative connections and longitudinal representations highlighted in Theorem \ref{thm:regular} is specific to the \emph{proper} context; the conclusions of the theorem are normally false for regular groupoids that are not proper. Our next example shows that the properness hypothesis cannot even be dropped in the rank-zero case.

\begin{exmp}\label{npar:14A.3.4} Let $\theta$ be any nonnegative real-valued $C^\infty$ function on a smooth manifold $M$ and let $U$ be the open set on which $\theta > 0$. The quotient $\varGamma$ of the trivial Lie group bundle over $M$ with fiber the additive group of the reals $\R = (\R,+)$ by the subgroupoid
\begin{equation*}
	M \times \{0\} \cup \left\{\bigl(u,2\pi n/\theta(u)\bigr): u \in U\text{, }n \in \Z\right\}
\end{equation*}
can be turned into a Lie group bundle over $M$ uniquely in such a manner that the quotient projection $M \times \R \to \varGamma$ becomes a submersive homomorphism; this follows at once from remarks contained in \cite[appendix A]{Tre6} but it can also be verified directly without difficulty. The quotient projection is actually a local diffeomorphism, along which any connection on $\varGamma$ may be pulled back to a connection on $M \times \R$. Now $\varGamma \mathbin| U$ is a (trivial) bundle of circle groups over $U$ and thus as we saw in section \ref{sec:preliminaries}, Example (b), it admits exactly one multiplicative connection, whose pullback coincides with the vector distribution on $U \times \R$ tangent to the embedded submanifolds $u \mapsto \bigl(u,\epsilon/\theta(u)\bigr)$, $\epsilon$ in $\R$. But unless $U$ coincides with its own closure in $M$, no connection on $M \times \R$ can be an extension of the latter vector distribution. We conclude that for nearly every choice of $\theta$ there exist no multiplicative connections on $\varGamma$. \end{exmp}

Corollary \ref{thm:rank=0} gives us complete information about the resolution of Problem \ref{prob:ext} in the rank-zero case, provided $U$ is invariant. In the rank-one case, Theorem \ref{thm:regular} gives us similar information if we only add one simple, mildly restrictive, hypothesis on $\varGamma \tto M$:

\begin{cor}\label{thm:rank=1} Let\/ $\varGamma \tto M$ be a proper Lie groupoid which is regular of rank one in the sense that\/ $M_1 = M$. Suppose that it is\/ \emph{source connected} in the sense that its source fibers\/ $\varGamma_x = s^{-1}(x)$ are connected. Then, provided\/ $U$ is invariant, Problem\/ \textup{\ref{prob:ext}} always admits a solution, and this solution is unique up to smooth deformation through multiplicative connections extending\/ $\varPhi \mathbin| U$. \end{cor}

\begin{proof} Because of the rank-one hypothesis and the source-connectedness, for any choice of a vector bundle metric on $L$ there is exactly one longitudinal representation that is orthogonal for the chosen metric.%
\footnote{%
 The existence of this representation is slightly less obvious than it may at first seem: although its definition makes sense even for non source-connected $\varGamma \tto M$ as long as all orbits $\varGamma x$ are connected, the “representation” thus obtained need not be $C^\infty$-differentiable (not even continuous) unless further restrictions such as source-connectedness are imposed on $\varGamma \tto M$.
} %
Since by a standard construction any vector bundle carrying a representation of a proper Lie groupoid admits invariant metrics (cf.~Lemma \ref{lem:metric} below), we can always find a metric on $L \mathbin| V$ which is invariant under the longitudinal effect of $\varPhi$ and hence a metric on $L$ which, over $U$, is invariant under the longitudinal effect of $\varPhi \mathbin| U$. Now obviously $\varPhi \mathbin| U = H \mathbin| U$ for some globally defined $H$ so Theorem \ref{thm:regular} yields the existence of a solution to Problem \ref{prob:ext}. As to the uniqueness up to smooth deformations of such a solution, let us fix any $C^\infty$ function $c: \R \to [0,1]$ with $c(\epsilon) = 0$ for $\epsilon < 0$ and $c(\epsilon) = 1$ for $\epsilon > 1$. Given the two solutions $\varPhi_0$ and $\varPhi_1$, let us consider the connections $H_\epsilon$ defined by $\eta^{H_\epsilon} = \bigl(1 - c(\epsilon)\bigr)\eta^{\varPhi_0} + c(\epsilon)\eta^{\varPhi_1}$. Clearly, these satisfy $H_\epsilon = \varPhi_0$ for $\epsilon < 0$, $H_\epsilon = \varPhi_1$ for $\epsilon > 1$, and $H_\epsilon \mathbin| U = \varPhi \mathbin| U$ for all $\epsilon$. We may view them as a single connection $H$ on the proper, source-connected, rank-one regular groupoid $\varGamma \times \R \tto M \times \R$ which is multiplicative along the invariant open set $M \times (-\infty,0) \cup M \times (1,+\infty) \cup U \times \R$. It is practically obvious that on the longitudinal bundle of this groupoid there is some metric which, over the indicated open set, is invariant under the effect of $H$. We finish by invoking Theorem \ref{thm:regular} one more time. \end{proof}

The obstructions to the solvability of Problem \ref{prob:ext} become significantly more involved when we move on to the next higher rank, $r = 2$. Even so, we can describe them completely for a reasonably large class of groupoids. In order to introduce their discussion, we remind the reader that at any base point $x$ of $\varGamma \tto M$ we have the principal $\varGamma_x^x$ bundle $t: \varGamma_x \to \varGamma x$, where the isotropy group $\varGamma_x^x$ acts on the source fiber $\varGamma_x$ via right composition. When $\varGamma \tto M$ is source connected, the initial segment of the long exact sequence of homotopy groups associated with the pointed fibration $(\varGamma_x^x,1x) \to (\varGamma_x,1x) \xto{t} (\varGamma x,x)$ reads
\begin{equation}
\xymatrix@C=1.33em{%
 0 \ar[r] & \pi_2(\varGamma_x) \ar[r] & \pi_2(\varGamma x) \ar[r]^-{\partial_2} & \pi_1(\varGamma_x^x) \ar[r] & \pi_1(\varGamma_x) \ar[r] & \pi_1(\varGamma x) \ar[r]^-{\partial_1} & \pi_0(\varGamma_x^x) \ar[r] & \{\ast\}}
\label{eqn:long-exact}
\end{equation}
in view of the result, as applied to $\varGamma_x^x$, that the second homotopy group of a [compact] Lie group is always zero---see \cite[Chapter V, Proposition 7.5]{BtomD}. The set $\pi_0(\varGamma_x^x)$ of path components of $\varGamma_x^x$ is in fact a group, and the boundary map $\partial_1$ is a homomorphism of groups. As any fundamental group of a topological group, $\pi_1(\varGamma_x^x)$ is abelian. Let $\freepi_1$ denote the functor, from Lie groups to abelian groups, which to every Lie group assigns the corresponding fundamental group modulo torsion.

\begin{thm*}[\cite{Tre8}, p.~38] Let\/ $\varGamma \tto M$ be a Lie groupoid which is regular of rank two in the sense that\/ $M_2 = M$ as well as\/ \emph{source proper} in the sense that all compact sets in\/ $M$ have compact inverse image under the map\/ $s: \varGamma \to M$. Suppose that it is source connected and that its source fibers have finite fundamental groups. Further, suppose that for every\/ $x \in M$ the four conditions below concerning the long exact sequence of homotopy groups\/ \eqref{eqn:long-exact} are satisfied.
\begin{enumerate}
\def\labelenumi{\upshape (\arabic{enumi})}
 \item $\pi_2(\varGamma_x) = 0$.
 \item For every\/ $a \in \freepi_1(\varGamma_x^x)$ such that the intersection\/ $F_a$ of\/ $\Z a$ with the image mod torsion of the boundary homomorphism\/ $\pi_2(\varGamma x) \xto{\partial_2} \pi_1(\varGamma_x^x)$ is not zero, $F_a$ contains\/ $2a$.
 \item The boundary map\/ $\pi_1(\varGamma x) \xto{\partial_1} \pi_0(\varGamma_x^x)$ is injective, hence an isomorphism of groups.
 \item $\freepi_1(c_g) = -\id \in \Aut\bigl(\freepi_1(\varGamma_x^x)\bigr)$ for every\/ $g \in \varGamma_x^x \smallsetminus \idcomp{\varGamma_x^x}$, where\/ $c_g \in \Aut(\varGamma_x^x)$ denotes the conjugation homomorphism\/ $h \mapsto ghg^{-1}$ and\/ $\idcomp{\varGamma_x^x}$ denotes the identity component of\/ $\varGamma_x^x$.
\end{enumerate}
Then, the same conclusions hold as in Corollary\/ \textup{\ref{thm:rank=1}}. On the other hand, as soon as there exists a base point\/ $x$ at which one or more of these four conditions fails to be satisfied---all of the other hypotheses withstanding---no multiplicative connection can exist on\/ $\varGamma \tto M$. \end{thm*}

It turns out that in this theorem either of the last two conditions is implied by the conjunction of the first two and, thus, is actually redundant; although this can be deduced without much effort from the theory of \cite{Tre8}, we failed to realize it at the time when \cite{Tre8} was written. On the other hand the first two conditions are independent in the sense that examples can be provided satisfying either condition but not the other; we refer the interested reader to \cite[\S 6.3]{Tre8}. Of course, whenever the source fibers are simply connected, the second condition will be satisfied automatically. The theorem quoted above is a first illustration of the kind of results that can be obtained by combining Theorem \ref{thm:regular} with the obstruction-theoretic analysis of longitudinal representations which we alluded to earlier in this subsection. Another such illustration will be given towards the end of subsection \ref{sub:algorithm}.

\subsection{A resolution algorithm for the extension problem}\label{sub:algorithm}

In this subsection we give a more refined and systematic discussion of our method of reduction to the regular case, in particular, we clarify in what sense the answer to Problem \ref{prob:ext} provided by our method is independent of the arbitrary choices involved in the stepwise prolongation process described at the beginning of the section. In order to be able to do this without getting distracted by unenlightening technicalities, we shall henceforth and until the end of the section assume that $\varGamma \tto M$ is not only \emph{proper} but also \emph{source proper,} i.e., that $s^{-1}(K)$ is compact for every compact $K \subset M$. While this simplifying assumption is only mildly restrictive, it will enable us to express our ideas in a particularly neat and suggestive form. In practice, it will manifest itself as the following principle (whose proof is but a trivial exercise in point-set topology), of which we shall take advantage in the proof of Proposition \ref{prop:step}\spacefactor3000: \em Any open neighborhood of an invariant subset\/ $S$ of\/ $M$ contains an\/ \emph{invariant} open neighborhood\/ $W = \varGamma W \supset S$.\em

Let an \emph{invariant} open subset $U$ of $M$ be given, and let $\varPhi \mathbin| U$ be a multiplicative connection on $\varGamma \mathbin| U \tto U$ that is prolongable over some open neighborhood of $\overline{U}$. For $r = -1$,~$0$,~$1$,~$2$,~$\dotsc$ let us set $C_r = \overline{U} \cup \bigcup_{q=0}^r M_q$, in particular, $C_{-1} = \overline{U}$. For any two multiplicative connections $\varPhi$ on $\varGamma \mathbin| V \tto V$ and $\varPhi'$ on $\varGamma \mathbin| V' \tto V'$ defined over two open neighborhoods $V$ and $V'$ of $C_r$ and inducing the assigned connection $\varPhi \mathbin| U = \varPhi' \mathbin| U$ on $\varGamma \mathbin| U \tto U$, let us write $\varPhi \sim_r \varPhi'$ to say that for some open set $W \supset C_r$ with $W \subset V \cap V'$ there exists on $\varGamma \mathbin| W \tto W$ some $C^\infty$ one-parameter family of multiplicative connections $\varPhi_\epsilon$, $\epsilon$ in $\R$, such that $\varPhi_0 = \varPhi \mathbin| W$, $\varPhi_1 = \varPhi' \mathbin| W$, and $\varPhi_\epsilon \mathbin| U = \varPhi \mathbin| U$ for all $\epsilon$. The binary relation $\sim_r$ is an equivalence%
\footnote{%
 Given any one-parameter family $\varPhi_\epsilon$ as in the definition of $\sim_r$, we can always smoothly reparameterize it so as to make it locally constant near $0$ or $1$. This enables us to compose it with another similarly reparameterized family.
} %
on the set of all multiplicative connections $\varPhi$ that extend the assigned connection from $\varGamma \mathbin| U \tto U$ to the vicinity of $C_r$. Let $\mathfrak{P}_r$ denote the set of all $\sim_r$~equivalence classes $[\varPhi]_r$ of such $\varPhi$. For every $r$ we have the ``restriction'' map $\mathfrak{P}_{r+1} \to \mathfrak{P}_r$, $[\varPhi]_{r+1} \mapsto [\varPhi]_r$, which need not be surjective or injective. For $r \geq r_{\max}$ this map ``stabilizes,'' becoming the identity. The ``stable'' set $\mathfrak{P}_{\max} = \mathfrak{P}_r$ for $r \geq r_{\max}$ may be identified with the set of path-connected components of the space of global prolongations of $\varPhi \mathbin| U$. Problem \ref{prob:ext} requires us to find an effective algorithm for deciding whether $\mathfrak{P}_{\max}$ is empty or not.

We are going to describe one such algorithm which works on condition that we can actually solve Problem \ref{prob:ext} for each one of the \emph{regular} groupoids $\varGamma \mathbin| M_r \tto M_r$. Our strategy is as follows. Proceeding by induction on $r$, and in a highly non-canonical way, we are going to construct a tower of sets $\mathscr{P}_r$ covering the tower of all $\mathfrak{P}_r$, as in the next diagram. The $r$-th set $\mathscr{P}_r$ will consist of multiplicative connections extending $\varPhi \mathbin| U$, each one defined over a corresponding open neighborhood of $C_r$. The surjection $\mathscr{P}_r \onto \mathfrak{P}_r$ will simply assign each $\varPhi \in \mathscr{P}_r$ the corresponding class $[\varPhi]_r$.
\begin{equation}
 \begin{split}
\xymatrix{%
 \mathscr{P}_{-1}
 \ar@{->>}[d]
 &	\ar[l] \mathscr{P}_0
	\ar@{->>}[d]
	&	\ar[l] \dotsb
		&	\ar[l] \mathscr{P}_{r-1}
			\ar@{->>}[d]
			&	\ar[l] \mathscr{P}_r
				\ar@{->>}[d]
				&	\ar[l] \dotsb
					&	\ar[l] \mathscr{P}_{\max}
						\ar@{->>}[d]
\\ \mathfrak{P}_{-1}
 &	\ar[l] \mathfrak{P}_0
	&	\ar[l] \dotsb
		&	\ar[l] \mathfrak{P}_{r-1}
			&	\ar[l] \mathfrak{P}_r
				&	\ar[l] \dotsb
					&	\ar[l] \mathfrak{P}_{\max}
}\end{split}
\label{eqn:tower}
\end{equation}
Each $\mathscr{P}_r$ will be constructed in such a way that any two $\varPhi$,~$\varPhi' \in \mathscr{P}_r$ will be the same whenever they satisfy the following condition for every $q = 0$,~$\dotsc$,~$r$: there exists some $C^\infty$ one-parameter family of multiplicative connections $\varPsi_\epsilon$ on $\varGamma \mathbin| M_q \tto M_q$ such that $\varPsi_0 = \varPhi \mathbin| M_q$, $\varPsi_1 = \varPhi' \mathbin| M_q$, and $\varPsi_\epsilon \mathbin| M_q \cap W = \varPhi \mathbin| W \cap M_q$ for all $\epsilon$ for some open $W \supset C_{q-1}$ contained in the domain of definition of $\varPhi$. The point is to keep $\mathscr{P}_r$ within the limits of a ``reasonably small,'' ``computable,'' (though possibly redundant) random selection of representatives of the $\sim_r$~classes. We want the tower of the $\mathscr{P}_r$ to come as close as possible to that of the $\mathfrak{P}_r$, which is our true object of interest but is harder to determine explicitly. The obvious remark, that $\mathfrak{P}_{\max}$ will be nonempty iff so is $\mathscr{P}_{\max}$, means precisely that the choices which we eventually make while executing our algorithm (the specific $\mathscr{P}_r$ which we end up constructing) are completely irrelevant insofar as deciding whether the extension problem is solvable is concerned. As a side bonus, since $\mathscr{P}_{\max}$ will contain at least one representative for each path-connected component of the space of global prolongations of $\varPhi \mathbin| U$, we shall obtain an estimate of the number of ``inequivalent'' such prolongations.

According to the properties that we demand of our sets $\mathscr{P}_r$, we are allowed to throw no more than one element into $\mathscr{P}_{-1}$. This causes no conflict with our demanding that the map of $\mathscr{P}_{-1}$ to $\mathfrak{P}_{-1}$ be a surjection, thanks to Theorem \ref{thm:main}:

\begin{prop}\label{prop:base} $\mathfrak{P}_{-1} = \{\ast\}$ is a singleton, i.e., it contains exactly one element. \end{prop}

\begin{proof} By our hypotheses on $\varPhi \mathbin| U$, the set $\mathfrak{P}_{-1}$ is nonempty. Let $\varPhi$ on $\varGamma \mathbin| V \tto V$ and $\varPhi'$ on $\varGamma \mathbin| V' \tto V'$ be such that $\varPhi \mathbin| U = \varPhi' \mathbin| U$, where $V$ and $V'$ are open and $V \cap V' \supset \overline{U}$. We need to prove that $\varPhi \sim_{-1} \varPhi'$, for which purpose it is not restrictive to assume that $V = V' = M$. We apply Theorem \ref{thm:main} to any connection $H$ on $\varGamma \mathbin| M \times \R \tto M \times \R$ inducing the obvious multiplicative connection along the invariant open set $S = M \times (-\infty,0) \cup M \times (1,+\infty) \cup U \times \R$. Any open neighborhood of the closure of $S$ in $M \times \R$ will contain a ``tube'' of the form $W \times \R$, where $W$ is an open subset of $M$ containing $\overline{U}$. \end{proof}

Inductively, suppose we have constructed $\mathscr{P}_{r-1}$ for some $r \geq 0$. Each $\varPhi \in \mathscr{P}_{r-1}$ determines a corresponding partial equivalence%
\footnote{%
 A symmetric and transitive, but possibly not reflexive, binary relation.
} %
on the set of all multiplicative connections on $\varGamma \mathbin| M_r \tto M_r$: for any two such connections $\varPsi$ and $\varPsi'$, let $\varPsi \sim_\varPhi \varPsi'$ signify that, for some open set $W \supset C_{r-1}$ over which $\varPhi$ is defined, some $C^\infty$ one-parameter family of multiplicative connections $\varPsi_\epsilon$ can be found on $\varGamma \mathbin| M_r \tto M_r$ satisfying $\varPsi_0 = \varPsi$, $\varPsi_1 = \varPsi'$, and $\varPsi_\epsilon \mathbin| M_r \cap W = \varPhi \mathbin| W \cap M_r$ for all $\epsilon$; in particular, $\varPsi \sim_\varPhi \varPsi$ is equivalent to saying that for some $W$ having the properties just specified $\varPsi$ prolongs $\varPhi \mathbin| W \cap M_r$.

\begin{prop}\label{prop:step} There exists a canonical map from the set of\/ $\sim_\varPhi$~classes to\/ $\mathfrak{P}_r$ whose image coincides with the inverse image of the class\/ $[\varPhi]_{r-1}$ under the restriction map\/ $\mathfrak{P}_r \to \mathfrak{P}_{r-1}$. \end{prop}

\begin{proof} Whenever $\varPsi \sim_\varPhi \varPsi$ lies in a $\sim_\varPhi$~class, Theorem \ref{prop:14A.5.2} tells us that there has to be some multiplicative connection $\varPhi'$ defined around $C_r$ which agrees with $\varPhi$ near $C_{r-1}$ and induces $\varPsi$ along $M_r$. Suppose that $\varPsi \sim_\varPhi \varPsi'$, and let $\varPhi''$ bear to $\varPhi$,\ $\varPsi'$ the same relationship that $\varPhi'$ bears to $\varPhi$,\ $\varPsi$. We contend that $[\varPhi']_r = [\varPhi'']_r$: this will give us a well-defined map, $[\varPsi]_\varPhi \mapsto [\varPhi']_r$, from $\sim_\varPhi$~classes to $\sim_r$~classes. In order to prove our contention, we may suppose that $\varPhi'$ and $\varPhi''$ are both defined over $V'$, and that they agree with $\varPhi$ over a smaller $W$ having the properties specified in the definition of $\varPsi \sim_\varPhi \varPsi'$. For any family $\varPsi_\epsilon$ as in that definition, we apply Theorem \ref{prop:14A.5.2} to the groupoid $\varGamma \mathbin| V' \times \R \tto V' \times \R$ and to the obvious partial multiplicative connections defined over $V' \times (-\infty,\frac{1}{3}) \cup V' \times (\frac{2}{3},+\infty) \cup W \times \R$ and along $M_r \times \R$, and finish by taking a suitable open ``tube'' around $C_r \times \R$.

Let $\varPhi'$ extend $\varPhi \mathbin| U$ around $C_r$ and satisfy $[\varPhi']_{r-1} = [\varPhi]_{r-1}$: over some open $W \supset C_{r-1}$, which by \emph{source-properness} we may assume \emph{invariant,} we have $\varPhi_\epsilon$ such that $\varPhi_0 = \varPhi' \mathbin| W$, $\varPhi_1 = \varPhi \mathbin| W$, and $\varPhi_\epsilon \mathbin| U = \varPhi \mathbin| U$ for all $\epsilon$. Our goal is to construct $\varPhi''$ around $C_r$ agreeing with $\varPhi$ near $C_{r-1}$ and satisfying $[\varPhi']_r = [\varPhi'']_r$. Let us pick any \emph{invariant} $C^\infty$ function $c: M \to [0,1]$ with $\supp c \subset W$ and $c = 1$ near $C_{r-1}$; such functions exist by \emph{properness,} and can be constructed by averaging out any non-invariant $C^\infty$ function with the same properties (cf.~section \ref{sec:averaging} below). For all $\epsilon$ in $\R$ and all $g$ in $\varGamma \mathbin| M_r$ let us set
\begin{equation*}
	\eta^{\varPsi_\epsilon}_g =%
\begin{cases}
	\eta^{\varPhi_{c(sg)\epsilon}}_g \mathbin| \T_{sg}M_r &\text{for $g \in \varGamma \mathbin| W$,}
\\	\eta^{\varPhi'|M_r}_g                                 &\text{for $g \notin \varGamma \mathbin| \supp c$.}
\end{cases}
\end{equation*}
Then $\varPsi_1 \sim_\varPhi \varPsi_1$ and, for every $\epsilon$, $\varPsi_\epsilon$ matches $\varPhi_\epsilon$ along the interior of the invariant neighborhood of $C_{r-1}$ where $c = 1$. Now, by Theorem \ref{prop:14A.5.2}, we can as in the previous paragraph construct a $C^\infty$ one-parameter family of multiplicative connections $\varPhi'_\epsilon$ defined around $C_r$ such that $\varPhi'_\epsilon$ agrees with $\varPhi_\epsilon$ near $C_{r-1}$ and induces $\varPsi_\epsilon$ along $M_r$. \end{proof}

The preceding proposition remains valid, exactly as stated, for $\varGamma \tto M$ an arbitrary proper Lie groupoid (not necessarily source proper). Unfortunately, in such generality its proof relies on a homotopy extension theorem for representations of proper Lie groupoids whose discussion would take up much more space than we can possibly afford to spend here. For this and other reasons we think it is preferable to deal with the general case elsewhere.

Determining the set of $\sim_\varPhi$~classes explicitly for each $\varPhi \in \mathscr{P}_{r-1}$ amounts to solving a number of problems of type \ref{prob:ext}--\ref{prob:ext**} for the \emph{regular} groupoid $\varGamma \mathbin| M_r \tto M_r$, a task which we shall without further analysis pretend we can actually carry out.

For each pair $\varPhi$,~$[\varPsi]_\varPhi$ consisting of an element $\varPhi \in \mathscr{P}_{r-1}$ and a $\sim_\varPhi$~class $[\varPsi]_\varPhi$ let us randomly pick one representative $\varPhi'$ of the $\sim_r$~class corresponding to $[\varPsi]_\varPhi$ under the canonical map of Proposition \ref{prop:step} from among those which agree with $\varPhi$ over an open neighborhood of $C_{r-1}$. Let us declare $\mathscr{P}_r$ to be formed by all such arbitrarily chosen representatives $\varPhi'$, one for each pair $\varPhi$,~$[\varPsi]_\varPhi$. The map $\mathscr{P}_r \to \mathfrak{P}_r$, $\varPhi' \mapsto [\varPhi']_r$ is surjective by the inductively postulated surjectivity of $\mathscr{P}_{r-1} \to \mathfrak{P}_{r-1}$ and by the proposition. The $\varPhi$ in $\mathscr{P}_{r-1}$ which $\varPhi' \in \mathscr{P}_r$ came from is supposed to be inductively determined by the induced connections $\varPhi' \mathbin| M_q = \varPhi \mathbin| M_q$, $q < r$. This yields a well-defined map $\mathscr{P}_r \to \mathscr{P}_{r-1}$ covering $\mathfrak{P}_r \to \mathfrak{P}_{r-1}$, as in \eqref{eqn:tower}. Our remaining demands on $\mathscr{P}_r$ are satisfied by construction. This finishes the induction. The description of our algorithm is complete.

A few theorems about prolongation and deformation of multiplicative connections can now be obtained in an obvious way by combining our algorithm with the results described in subsection \ref{sub:regular}. It turns out that the simplest such theorem, asserting that $\mathfrak{P}_r = \{\ast\}$ for $r = 0$, is valid for all proper Lie groupoids, not only for source-proper ones, and that the similar theorem for $r = 1$ is valid for all those proper Lie groupoids whose $s$-fibers are connected; we omit the proofs, which we consider too specialized for this paper (but see our ``supplementary remarks'' at the end of subsection \ref{sub:proof2}). Apart from these theorems, the following (far less obvious) application of both our algorithm and the obstruction theoretic methods of \cite{Tre8} is worth mentioning; recall that the \emph{orbit space} $M/\varGamma$ of $\varGamma \tto M$ is the topological quotient of $M$ which results from declaring two points equal whenever they lie on the same orbit.

\begin{thm*}[\cite{Tre8}, p.~41] Let\/ $\varGamma \tto M$ be a source-proper Lie groupoid which has the property that\/ $\dim\varGamma x \leq 2$ for all\/ $x$ in\/ $M$. Suppose that it is source connected, and that its orbit space\/ $M/\varGamma$ is connected. Then, provided\/ $\dim\varGamma x \leq 1$ for at least one\/ $x$, any two multiplicative connections on\/ $\varGamma \tto M$ can be deformed smoothly into each other through multiplicative connections. \end{thm*}

\section{The averaging operator}\label{sec:averaging}

We begin by reviewing differentiable Haar systems on Lie groupoids. Let $\varGamma \tto M$ be an arbitrary Lie groupoid. For each point $x$ of $M$ let $\mu^x$ be a regular Borel measure on the target fiber $\varGamma^x = t^{-1}(x)$.%
\footnote{%
 There seems to be no universal agreement on the terminology. Some authors, including ourselves, require regular measures to be finite on compact sets in addition to having the usual inner and outer approximation properties. Other authors speak of ``Radon measures'' instead.
} %
If $\varphi \in C_c(\varGamma^x)$ is any continuous function with compact support on $\varGamma^x$ then $\varphi$ is $\mu^x$-integrable and we write $\integral \varphi(h) \der\mu^x(h)$ or $\integral_{th=x} \varphi(h) \der h$ for its integral. We call the family $\{\mu^x\}$ a \emph{differentiable\/ \textup(left\/\textup) Haar system} on $\varGamma \tto M$ if it enjoys the following properties:
\begin{enumerate}
 \item (Positivity.) For each $x$ the support of the measure $\mu^x$ is all of $\varGamma^x$ in other words the empty set is the only open subset of $\varGamma^x$ that is assigned measure zero by $\mu^x$.
 \item (Left invariance.) For every arrow $g$ one has $\mu^{tg}(gA) = \mu^{sg}(A)$ for all Borel subsets $A$ of $\varGamma^{sg}$ and, consequently, $\integral \varphi(gh) \der\mu^{sg}(h) = \integral \varphi(h) \der\mu^{tg}(h)$ for all functions $\varphi \in C_c(\varGamma^{tg})$.
 \item (Differentiability.) For every $C^\infty$-differentiable function $\varphi \in C^\infty_c(\R^n \times \varGamma)$ the function on $\R^n \times M$ given by $(r_1,\dotsc,r_n,x) \mapsto \integral \varphi(r_1,\dotsc,r_n,h) \der\mu^x(h)$ is $C^\infty$-differentiable.
\end{enumerate}
The standard way to endow $\varGamma \tto M$ with a differentiable Haar system is to take any \emph{left invariant} $C^\infty$ vector bundle metric on $\ker\der t$ and then consider the volume densities that this metric induces along the $t$-fibers; by definition any such metric corresponds to an ordinary metric on $1^*\ker\der t$ via the isomorphism $s^*(1^*\ker\der t) \simto \ker\der t$ that makes each vector $w \in \ker\T_{1sg}t$ correspond to its composition, $0_gw \in \ker\T_gt$, with the null vector $0_g \in \T_g\varGamma$ in the tangent groupoid $\T\varGamma \tto \T M$.

A \emph{cut-off function} for $\varGamma \tto M$ is a nonnegative real-valued $C^\infty$ function $c$ on $M$ whose support has the property that for every compact set $K$ in $M$ the intersection $\supp c \cap \varGamma K$ is compact and every orbit $\varGamma x$ has nonempty intersection with the open subset of $M$ on which $c > 0$. For any choice of a differentiable Haar system on $\varGamma \tto M$ a \emph{normalizing function} is a cut-off function such that $\integral_{th=x} c(sh) \der h = 1$ for all $x$ in $M$. The result of dividing any cut-off function by the positive invariant $C^\infty$ function $x \mapsto \integral_{th=x} c(sh) \der h$ is a normalizing function. A straightforward adaptation of the arguments of \cite[section 6.2]{Tu} shows that cut-off functions can only exist when $\varGamma \tto M$ is proper, and that in such case every differentiable Haar system admits normalizing functions.

For the rest of this section $\varGamma \tto M$ is a proper Lie groupoid endowed with a differentiable Haar system and $c$ is a normalizing function.

The key constructions of this paper are all ``Haar integrals depending on parameters'' in a sense which we now explain. Let $f: P \to M$ be any differentiable map from some manifold $P$ of ``parameters'' into the base $M$ of our groupoid. The projection, $\pr$, from the fiber product
\begin{equation*}
	P \ftimes{f}{t} \varGamma = \{(z,h) \in P \times \varGamma: f(z) = th\}
\end{equation*}
to $P$ is evidently surjective. Any smooth vector bundle $E$ over $P$ may be pulled back along $\pr$ thus giving rise to a smooth vector bundle $\pr^*E$ over $P \ftimes{f}{t} \varGamma$. Integration along the target fibers enables us to transform arbitrary $C^\infty$ cross-sections $\vartheta$ of $\pr^*E$ into cross-sections of $E$:%
\begin{subequations}
\label{eqn:12B.10.6}
\begin{equation}
	P \ni z \mapsto \integral_{th=f(z)} c(sh)\vartheta(z,h) \der h \in E_z.
\label{eqn:12B.10.6a}
\end{equation}
In this expression, which is what we refer to as a ``Haar integral'' depending on the ``parameter'' $z$, the integrand is a compactly supported $C^\infty$ function on $\varGamma^{f(z)}$ with values in the finite-dimensional vector space $E_z$. It is easy to check that because of the differentiability of both the Haar system and the normalizing function the cross-section of $E$ defined by \eqref{eqn:12B.10.6a} is always $C^\infty$. Thus, Haar integration with parameters in $P$ and values in $E$ gives a functional
\begin{equation}
	\Gamma^\infty(P \ftimes{f}{t} \varGamma;\pr^*E) \longto \Gamma^\infty(P;E).
\label{eqn:12B.10.6b}
\end{equation}
\end{subequations}
For the purposes of the present section we may think of \eqref{eqn:12B.10.6b} simply as a linear map; it won't be until section \ref{sec:proof} that it will be necessary for us to look in some detail into its continuity properties.

\begin{defn}\label{defn:12B.11.2} A connection $H$ on $\varGamma \tto M$ is \emph{nondegenerate} when its effect $\lambda^H = \der t \circ \eta^H$ is an invertible pseudo-representation $s^*\T M \simto t^*\T M$ in other words every linear map $\lambda^H_g = \T_gt \circ \eta^H_g$ is a bijection of $\T_{sg}M$ onto $\T_{tg}M$. The \emph{mean ratio} of the nondegenerate connection $H$ is the connection $\hat{H}$ with horizontal lift
\begin{equation}
	\eta^{\hat{H}}_g = \integral_{th=sg} c(sh)[\eta^H_{gh} \cdot (\eta^H_h)^{-1}] \circ (\lambda^H_h)^{-1} \der h.
\label{eqn:12B.11.6}
\end{equation}
We shall refer to the correspondence $H \mapsto \hat{H}$, from nondegenerate connections to connections, as the \emph{averaging operator} associated with the given Haar system and normalizing function. \end{defn}

Concerning this definition, a few clarifications are in order. The expression $\eta^H_{gh} \cdot (\eta^H_h)^{-1}$ is one way of writing the composition of the two linear maps
\begin{alignat*}{2} &
	\T_{sh}M \longto \T_{gh}\varGamma \ftimes{\T_{gh}s}{\T_hs} \T_h\varGamma, \quad%
		v \mapsto (\eta^H_{gh}v,\eta^H_hv) & &\llap{and} \\ &
	\T_{gh}\varGamma \ftimes{\T_{gh}s}{\T_hs} \T_h\varGamma \longto \T_{ghh^{-1}=g}\varGamma, \quad%
		(w_1,w_2) \mapsto w_1w_2^{-1},
\end{alignat*}
where $w_1w_2^{-1}$ is the ratio of $w_1$ to $w_2$ as elements of the tangent groupoid $\T\varGamma \tto \T M$. When the bijection $\lambda^H_h: \T_{sh}M \simto \T_{th=sg}M$ is inverted and composed with $\eta^H_{gh} \cdot (\eta^H_h)^{-1}$, the result is a linear map of $\T_{sg}M$ into $\T_g\varGamma$. In equation \eqref{eqn:12B.11.6}, we view the integrand as a $C^\infty$ cross-section of $\pr^*L(s^*\T M,\T\varGamma)$, where $\pr$ is the projection $\varGamma \ftimes{s}{t} \varGamma \to \varGamma$ sending $(g,h) \mapsto g$, and accordingly the whole right-hand member as a Haar integral depending on the parameter $g$. This gives us a global $C^\infty$ cross-section
\begin{equation*}
	\eta^{\hat{H}} \in \Gamma^\infty\bigl(\varGamma;L(s^*\T M,\T\varGamma)\bigr).
\end{equation*}
In order to be able to say that this is the horizontal lift of a connection, we need to make sure it satisfies the condition $\der s \circ \eta^{\hat{H}} = \id$; but that is clear:
\begin{align*}
 \T_gs \circ \eta^{\hat{H}}_g &
	= \integral_{th=sg} c(sh)\T_gs \circ [\eta^H_{gh} \cdot (\eta^H_h)^{-1}] \circ (\lambda^H_h)^{-1} \der h \\ &
	= \integral_{th=sg} c(sh)[\T_ht \circ \eta^H_h] \circ (\lambda^H_h)^{-1} \der h \\ &
	= \integral_{th=sg} c(sh)\id \der h
	= \id.
\end{align*}

The newly defined connection $\hat{H}$ has a number of properties. First, it is always unital:
\begin{align}
 \eta^{\hat{H}}_{1x} &
	= \integral_{th=x} c(sh)[\eta^H_h \cdot (\eta^H_h)^{-1}] \circ (\lambda^H_h)^{-1} \der h \notag\\ &
	= \integral_{th=x} c(sh)[\T_{th}1 \circ \T_ht \circ \eta^H_h] \circ (\lambda^H_h)^{-1} \der h \notag\\ &
	= \integral_{th=x} c(sh)\T_x1 \der h
	= \T_x1.
\label{lem:12B.11.3}
\end{align}
Second, its effect is expressible in terms of the effect of $H$ by means of a simple formula,
\begin{align}
 \lambda^{\hat{H}}_g
	= \T_gt \circ \eta^{\hat{H}}_g &
	= \integral_{th=sg} c(sh)\T_gt \circ [\eta^H_{gh} \cdot (\eta^H_h)^{-1}] \circ (\lambda^H_h)^{-1} \der h \notag\\ &
	= \integral_{th=sg} c(sh)[\T_{gh}t \circ \eta^H_{gh}] \circ (\lambda^H_h)^{-1} \der h \notag\\ &
	= \integral_{th=sg} c(sh)\lambda^H_{gh} \circ (\lambda^H_h)^{-1} \der h,
\label{eqn:12B.11.7}
\end{align}
from which it follows at once that $\lambda^{\hat{H}}$ must be the same as $\lambda^H$ whenever $\lambda^H$ is a representation. Third, if $H$ is multiplicative along an invariant submanifold $Z$ of $M$, then $\hat{H} \mathbin| Z = H \mathbin| Z$. To prove it, let us suppose somewhat more generally that $Z \subset M$ is a locally invariant submanifold satisfying $\supp c \cap \varGamma Z \subset Z$, rather than an invariant submanifold. Let $g \in \varGamma \mathbin| Z$ and $v \in \T_{sg}Z$ be given. For every $h \in s^{-1}(\supp c) \cap t^{-1}(sg)$ we have $h \in \varGamma \mathbin| Z$ and hence $v^h := (\lambda^H_h)^{-1}v = \lambda^H_{h^{-1}}v \in \T_{sh}Z$ by the multiplicativity of $H \mathbin| Z$, and likewise $\eta^H_{gh}v^h = \eta^H_g\lambda^H_hv^h \cdot \eta^H_hv^h = \eta^H_gv \cdot \eta^H_hv^h$. Therefore,
\begin{align}
 \eta^{\hat{H}}_gv &
	= \integral_{th=sg} c(sh)[\eta^H_{gh}v^h \cdot (\eta^H_hv^h)^{-1}] \der h \notag\\ &
	= \integral_{th=sg} c(sh)\eta^H_gv \der h
	= \eta^H_gv.
\label{lem:hat|Z}
\end{align}

\subsection{Proof of Theorem \ref{prop:12B.11.7}}\label{sub:proof1}

Let $H$ be now a connection on $\varGamma \tto M$ whose effect $\lambda^H$ is a \emph{representation.} In virtue of our previous remarks, $\hat{H}$ enjoys all of the properties required of the connection whose existence is stated in Theorem \ref{prop:12B.11.7} except perhaps for multiplicativity. We are going to prove our theorem by showing that under the present hypotheses $\hat{H}$ is actually \emph{always multiplicative.} We begin our proof with a few preliminary considerations of a general character.

Let $\Lie\varGamma = 1^*\ker\der s$ be the \emph{Lie algebroid bundle} of $\varGamma \tto M$, i.e., the smooth vector bundle over $M$ with fibers $\Lie_x\varGamma = \ker\T_{1x}s = \T_{1x}\varGamma_x$. For every $g$ in $\varGamma$, the right-translation map $\varGamma_{sg} \simto \varGamma_{tg}$, $h \mapsto hg^{-1}$ is a diffeomorphism which sends $g \mapsto 1tg$. As one lets $g$ vary, the induced tangent maps $\T_g\varGamma_{sg} \simto \T_{1tg}\varGamma_{tg}$, $w \mapsto w \cdot 0_{g^{-1}}$ (the dot indicates multiplication in the tangent groupoid $\T\varGamma \tto \T M$, as usual) make up an isomorphism of vector bundles over $\varGamma$
\begin{equation}
	\omega: \ker\der s \simto t^*\Lie\varGamma
\label{eqn:12B.9.5}
\end{equation}
which generalizes the familiar (right-invariant) Maurer--Cartan form \cite[p.~96]{Sharpe}. Let $\an: \Lie\varGamma \to \T M$ denote the \emph{infinitesimal anchor map} of $\varGamma \tto M$, that is, the morphism of smooth vector bundles over $M$ given by $1^*\ker\der s \xto{1^*\der t} 1^*t^*\T M \simeq \T M$.

To a large extent, the deviation of an arbitrary connection $H$ from multiplicativity is measured by the so-called ``basic curvature'' of $H$: if we let $s$ and $t$ denote not only the usual source and target but also the two maps of $\varGamma_2 := \varGamma \ftimes{s}{t} \varGamma$ onto $M$ that send $(g,h) \in \varGamma_2$ respectively to $sh$ and $tg$, this is the cross-section%
\begin{subequations}
\label{eqn:basic}
\begin{equation}
	R^H \in \Gamma^\infty\bigl(\varGamma_2;L(s^*\T M,t^*\Lie\varGamma)\bigr)
\label{eqn:basic*}
\end{equation}
defined by the formula
\begin{equation}
	R^H(g,h) = \omega_{gh} \circ (\eta^H_g\lambda^H_h \cdot \eta^H_h - \eta^H_{gh}): \T_{sh}M \longto \Lie_{tg}\varGamma.
\label{eqn:basic**}
\end{equation}
\end{subequations}
According to equation \eqref{prop:12B.9.4}, the vanishing of $R^H$ is necessary for multiplicativity; on account of the remarks following that equation and of the invertibility of $\omega_{gh}$, it is also sufficient provided $H$ is \emph{nondegenerate} (Definition \ref{defn:12B.11.2}).

Besides $R^H$, $H$ determines a pseudo-representation hereafter denoted $\alpha^H: s^*\Lie\varGamma \to t^*\Lie\varGamma$ of $\varGamma \tto M$ on its own Lie algebroid bundle: by definition, for all $g \in \varGamma$, $X \in \Lie_{sg}\varGamma$,
\begin{equation}
	\alpha^H_gX = \omega_g\bigl(\eta^H_g(\an_{sg}X) \cdot X\bigr).
\label{eqn:14A.13.1}
\end{equation}
As in \cite[p.~502]{Wein}, let $\mathfrak{b} = \ker(\an: \Lie\varGamma \to \T M)$ denote the \emph{isotropy subbundle} of $\Lie\varGamma$. This is a subbundle of class $C^\infty$ if, and only if, its fibers have locally constant rank. (They might not!) Clearly $\an_{tg} \circ \alpha^H_g = \lambda^H_g \circ \an_{sg}$ for all $g$ i.e.~$\an$ intertwines the two pseudo-representations $\alpha^H$ and $\lambda^H$, so that $\alpha^H$ carries $\mathfrak{b}$ into itself; in fact $\alpha^H_gX = 0_g \cdot X \cdot 0_{g^{-1}}$ for all $X \in \ker(\an_{sg}) = \T_{1sg}\varGamma_{sg}^{sg}$ in other words $\alpha^H_g: \mathfrak{b}_{sg} \to \mathfrak{b}_{tg}$ coincides with $\Ad(g) = \T_{1sg}c_g$ (the ``adjoint representation''), the differential at $1sg \in \varGamma_{sg}^{sg}$ of the conjugation homomorphism $c_g: \varGamma_{sg}^{sg} \to \varGamma_{tg}^{tg}$, $h \mapsto ghg^{-1}$.

It is possible to express the $H$-vertical component of the tangent multiplication entirely in terms of $\alpha^H$ and $R^H$: if we set
\begin{equation}
	\omega^H = \omega \circ (\id - \eta^H \circ \der s): \T\varGamma \longto t^*\Lie\varGamma,
\label{eqn:H-vertical}
\end{equation}
then, for every pair of tangent vectors $(w_1,w_2)$ in $\T_{g_1}\varGamma \ftimes{\T_{g_1}s}{\T_{g_2}t} \T_{g_2}\varGamma$,%
\footnote{%
 Evidently the formula is valid for any of the following pairs: $(w_1,0_{g_2})$ with $w_1 \in \ker\T_{g_1}s$; $\bigl((\eta^H_{g_1} \circ \T_{g_2}t)w_2,w_2\bigr)$ with $w_2 \in \ker\T_{g_2}s$; $(\eta^H_{g_1}\lambda^H_{g_2}v,\eta^H_{g_2}v)$ with $v \in \T_{sg_2}M$. In view of the linearity of tangent multiplication, this is all we need in order to conclude that it must be valid for the general pair as well.}
\begin{equation}
	\omega^H_{g_1g_2}(w_1 \cdot w_2) = \omega^H_{g_1}w_1 + \alpha^H_{g_1}(\omega^H_{g_2}w_2) + R^H(g_1,g_2)\bigl((\T_{g_2}s)w_2\bigr).
\label{eqn:14A.13.2}
\end{equation}

Let $H$ be now an arbitrary \emph{nondegenerate} connection on $\varGamma \tto M$. The horizontal lift of any other connection $H'$ may be put into the form
\begin{equation}
	\eta^{H'}_g = \eta^H_g + \omega_g^{-1} \circ X(g) \circ \lambda^H_g
\label{eqn:14A.14.3}
\end{equation}
for a unique cross-section $X \in \Gamma^\infty\bigl(\varGamma;t^*L(\T M,\Lie\varGamma)\bigr)$: we simply take $X = \omega^H \circ \eta^{H'} \circ (\lambda^H)^{-1}$. Of course, any $C^\infty$ cross-section $X$ of $t^*L(\T M,\Lie\varGamma)$ gives rise to a corresponding $H'$ via \eqref{eqn:14A.14.3}. We ask the following question\spacefactor3000: For which $X$, if any, is the corresponding $H'$ multiplicative? Let $(g,h)$ belong to $\varGamma_2$. On the basis of the identity $\omega^{H'} = \omega^H \circ (\id - \eta^{H'} \circ \der s)$ and of \eqref{eqn:14A.13.2}, we have
\begin{align}
 R^{H'}(g,h) &
	= \omega^{H'}_{gh} \circ (\eta^{H'}_g\lambda^{H'}_h \cdot \eta^{H'}_h) \notag\\ &
	= \omega^H_{gh} \circ (\eta^{H'}_g\lambda^{H'}_h \cdot \eta^{H'}_h)
	  - \omega^H_{gh} \circ \eta^{H'}_{gh} \circ
	    \T_{gh}s \circ (\eta^{H'}_g\lambda^{H'}_h \cdot \eta^{H'}_h) \notag\\ &
	= (\omega^H_g \circ \eta^{H'}_g) \circ \lambda^{H'}_h
	  + \alpha^H_g \circ (\omega^H_h \circ \eta^{H'}_h)
	  + R^H(g,h) \circ \T_hs \circ \eta^{H'}_h \notag\\* &\justify
	  - (\omega^H_{gh} \circ \eta^{H'}_{gh}) \circ
	    \T_hs \circ \eta^{H'}_h \notag\\ &
	= X(g) \circ \lambda^H_g \circ (\lambda^{H'}_h - \lambda^H_h)
	  + X(g) \circ \lambda^H_g \circ \lambda^H_h
	  + \alpha^H_g \circ X(h) \circ \lambda^H_h \notag\\* &\justify
	  + R^H(g,h) - X(gh) \circ \lambda^H_{gh}.
\label{eqn:defect}
\end{align}
From the latter equation we immediately deduce what follows:

\begin{lem}\label{lem:coboundary} Suppose that the effect\/ $\lambda^H$ of\/ $H$ is a representation and that the effect\/ $\lambda^{H'}$ of the connection\/ $H'$ corresponding via\/ \eqref{eqn:14A.14.3} to the\/ $C^\infty$ cross-section\/ $X$ of\/ $t^*L(\T M,\Lie\varGamma)$ coincides with\/ $\lambda^H$ in other words\/ $X$ takes values in the (possibly singular) subbundle\/ $t^*L(\T M,\mathfrak{b})$ of\/ $t^*L(\T M,\Lie\varGamma)$. Then, $H'$ is multiplicative if, and only if, for all\/ $(g,h) \in \varGamma_2$
\begin{equation}
	\Ad(g) \circ X(h) \circ (\lambda^H_g)^{-1} - X(gh) + X(g) = \varDelta^H(g,h),
\label{eqn:coboundary}
\end{equation}
where, by definition,
\begin{equation}
	\varDelta^H(g,h) = -R^H(g,h) \circ (\lambda^H_h)^{-1} \circ (\lambda^H_g)^{-1}. \qed%
\label{eqn:14A.14.2}
\end{equation} \end{lem}

This lemma can be given a more conceptual interpretation in terms of groupoid cohomology \cite[pp.~11--13]{Ren}. The pseudo-representation of $\varGamma \tto M$ on $L(\T M,\Lie\varGamma)$ given by
\begin{equation}
	L(\T_{sg}M,\Lie_{sg}\varGamma) \ni X \mapsto \alpha^H_g \circ X \circ (\lambda^H_g)^{-1} \in L(\T_{tg}M,\Lie_{tg}\varGamma)
\label{eqn:14A.14.1}
\end{equation}
carries the (possibly singular) vector subbundle $L(\T M,\mathfrak{b})$ of $L(\T M,\Lie\varGamma)$ into itself. By virtue of the obvious identities $\Ad(g_1g_2) = \Ad(g_1)\Ad(g_2)$ and $\Ad(1x) = \id$ and of our hypothesis that $\lambda^H$ is a representation, \eqref{eqn:14A.14.1} yields a ``representation'' of $\varGamma \tto M$ on $L(\T M,\mathfrak{b})$ upon restriction. Since $\an_{tg} \circ R^H(g,h) = \lambda^H_g\lambda^H_h - \lambda^H_{gh} = 0$, the $C^\infty$ cross-section $\varDelta^H \in \Gamma^\infty\bigl(\varGamma_2;t^*L(\T M,\Lie\varGamma)\bigr)$ defined by \eqref{eqn:14A.14.2} takes values in $t^*L(\T M,\mathfrak{b})$ and so may be viewed as a differentiable $2$-cochain on $\varGamma \tto M$ with coefficients in $L(\T M,\mathfrak{b})$.%
\footnote{%
 What we are really dealing with, here, is cohomology with coefficients in a \emph{sheaf,} namely, the sheaf of all $L(\T M,\mathfrak{b})$~valued $C^\infty$ cross-sections of $L(\T M,\Lie\varGamma)$, rather than with coefficients in a vector bundle.
} %
If now for each composable triplet of arrows $g$,~$h$,~$k$ we expand the expression $\omega^H_{ghk} \circ (\eta^H_g\lambda^H_h\lambda^H_k \cdot \eta^H_h\lambda^H_k \cdot \eta^H_k)$ successively by means of \eqref{eqn:14A.13.2} first in one way and then in the other, and compare the results, we get (cf.~\cite[p.~428]{Abad13})%
\begin{subequations}
\label{eqn:cocycle}
\begin{equation}
	\alpha^H_g \circ R^H(h,k) - R^H(gh,k) + R^H(g,hk) - R^H(g,h) \circ \lambda^H_k = 0
\label{eqn:cocycle*}
\end{equation}
or, after dividing by $\lambda^H_g\lambda^H_h\lambda^H_k$ and after using the hypothesis that $\lambda^H$ is a representation,
\begin{equation}
	\Ad(g) \circ \varDelta^H(h,k) \circ (\lambda^H_g)^{-1} - \varDelta^H(gh,k) + \varDelta^H(g,hk) - \varDelta^H(g,h) = 0.
\label{eqn:cocycle**}
\end{equation}
\end{subequations}

\begin{lem}\label{lem:cocycle} When\/ $\lambda^H$ is a representation, the\/ $2$-cochain\/ $\varDelta^H$ given by\/ \eqref{eqn:14A.14.2} is a\/ $2$-cocycle. \qed \end{lem}

Lemma \ref{lem:coboundary} may now be understood as saying that for any $H$ whose effect is a representation the problem of finding a multiplicative connection with the same effect as $H$ is solvable if, and only if, the class $[\varDelta^H]$ of the $2$-cocycle $\varDelta^H$ in the second differentiable cohomology of $\varGamma \tto M$ with coefficients in $L(\T M,\mathfrak{b})$ vanishes, and that in such case the multiplicative connections $H'$ for which $\lambda^{H'}$ is equal to $\lambda^H$ correspond bijectively via \eqref{eqn:14A.14.3} to the differentiable $L(\T M,\mathfrak{b})$~valued $1$-cochains $X$ whose coboundary $\delta X$ is equal to $\varDelta^H$.

What we have said so far is true, of course, regardless of whether $\varGamma \tto M$ is proper. Making use of properness we can now turn every differentiable $L(\T M,\mathfrak{b})$~valued $2$-cocycle $Z$ on $\varGamma \tto M$ into a differentiable $L(\T M,\mathfrak{b})$~valued $1$-cochain $\hat{Z}$ whose coboundary $\delta\hat{Z}$ is equal to $Z$: we can do this e.g.~by taking $\hat{Z}(g) = \integral_{th=sg} c(sh)Z(g,h) \der h$; the verification of the coboundary relation $\delta\hat{Z} = Z$ is standard \cite[pp.~694--695]{Crainic03} and relies on the normalizing function property as well as on the left invariance of the Haar system. Upon taking $Z = \varDelta^H$, we conclude that there must be \emph{some} multiplicative connection with the same effect as $H$, namely, the $H'$ which results from taking $X = \widehat{\varDelta^H}$ in \eqref{eqn:14A.14.3}. However, what we claimed initially was that $\hat{H}$ \emph{itself} was multiplicative. In order to bridge the gap, and thereby finish our proof, suffice it to say that by the linearity of tangent multiplication
\begin{align*}
 \omega_g^{-1} \circ R^H(g,h) &
	= (\eta^H_g\lambda^H_h \cdot \eta^H_h - \eta^H_{gh}) \cdot 0_{h^{-1}} \\ &
	= (\eta^H_g\lambda^H_h \cdot \eta^H_h - \eta^H_{gh}) \cdot [(\eta^H_h)^{-1} - (\eta^H_h)^{-1}] \\ &
	= \eta^H_g\lambda^H_h \cdot \eta^H_h \cdot (\eta^H_h)^{-1} - \eta^H_{gh} \cdot (\eta^H_h)^{-1} \\ &
	= \eta^H_g\lambda^H_h - \eta^H_{gh} \cdot (\eta^H_h)^{-1}
\end{align*}
and therefore
\begin{align*}
 [\eta^H_{gh} \cdot (\eta^H_h)^{-1}] \circ (\lambda^H_h)^{-1} &
	= \eta^H_g - [\eta^H_g\lambda^H_h - \eta^H_{gh} \cdot (\eta^H_h)^{-1}] \circ (\lambda^H_h)^{-1} \\ &
	= \eta^H_g + \omega_g^{-1} \circ \varDelta^H(g,h) \circ \lambda^H_g
\end{align*}
so that by the normalizing function property
\begin{equation}
	\eta^{\hat{H}}_g = \eta^H_g + \omega_g^{-1} \circ \widehat{\varDelta^H}(g) \circ \lambda_g.
\label{eqn:12B.15.5a}
\end{equation}

\begin{rem}\label{rem:abelian} Suppose that $H$ is multiplicative now. Lemma \ref{lem:coboundary} tells us that a connection $H'$ on $\varGamma \tto M$ with the same effect as $H$ is multiplicative if, and only if, the $C^\infty$ cross-section $X$ of $t^*L(\T M,\mathfrak{b})$ which to it corresponds via \eqref{eqn:14A.14.3} satisfies
\begin{equation*}
	\Ad(g) \circ X(h) \circ (\lambda^H_g)^{-1} - X(gh) + X(g) = 0.
\end{equation*}
This equation says that $X$ is a differentiable $L(\T M,\mathfrak{b})$~valued $1$-cocycle on $\varGamma \tto M$. But in that case, because of properness, $X$ must be the coboundary of some differentiable $L(\T M,\mathfrak{b})$~valued $0$-cochain, that is,
\begin{equation*}
	X(g) = \Ad(g) \circ Y(sg) \circ (\lambda^H_g)^{-1} - Y(tg)
\end{equation*}
for some $\mathfrak{b}$-valued $C^\infty$-differentiable vector bundle morphism $Y: \T M \to \mathfrak{b} \subset \Lie\varGamma$, one possible choice being $Y(x) = -\integral_{th=x} c(sh)X(h) \der h$. Let then $\varGamma \tto M$ be an abelian Lie group bundle. For every $g$ we have $\Ad(g) = \id$, $\lambda^H_g = \id$, and, therefore, $X(g) = 0$: we have just recovered, by other means, the same conclusions as in section \ref{sec:preliminaries}, Example (b). \end{rem}

\subsection{Proof of Theorem \ref{thm:regular}}\label{sub:proof2}

Let us now suppose that our proper Lie groupoid $\varGamma \tto M$ is regular, with longitudinal bundle $L \subset \T M$. Let $U$ be an open subset of $M$, and let $H$ be a connection on $\varGamma \tto M$ which is multiplicative along $U$. We are going to prove a version of Theorem \ref{thm:regular} which is more general than that stated in subsection \ref{sub:regular}: instead of requiring our open set $U$ to be \emph{invariant,} we only assume that it has the following property, which is evidently weaker than invariance.
\begin{equation}
	\text{\em There is some open set\/ $B$ in\/ $M$ such that\/ $\varGamma B = M$ and\/ $B \cap \varGamma U \subset U$.}
\label{eqn:domain}
\end{equation}
Of course, in our theorem, the `only if' direction is trivial. As to the nontrivial, `if' direction, let a longitudinal representation $\rho: s^*L \simto t^*L$ of $\varGamma \tto M$ be given such that $\rho_g$ equals the restriction of $\lambda^H_g$ to $L_{sg}$ for all $g \in \varGamma \mathbin| U$.

Our proof starts with the observation that, because of our assumption \eqref{eqn:domain} about $U$, there must be some Riemannian metric on $M$ which is invariant under the effect of $H \mathbin| U$ in the sense that $\lambda^H_g$ is an isometry of $\T_{sg}M$ onto $\T_{tg}M$ for every $g \in \varGamma \mathbin| U$. While conceptually this statement isn't any deeper than the standard result on the existence of invariant inner products in the theory of representations of compact groups, its proof involves certain technical subtleties whose unraveling might not be completely obvious. As the same construction will show up again in a more general context in subsections \ref{sub:deducing2} and \ref{sub:deducing1}, we record it as a stand-alone lemma for future reference.

\begin{lem}\label{lem:metric} Let\/ $\lambda$ be a pseudo-representation of the proper Lie groupoid\/ $\varGamma \tto M$ on the vector bundle\/ $E \to M$. Let\/ $S$ be a subset of\/ $M$ over which\/ $\lambda$ is a representation in the sense that\/ $\lambda_{1x} = \id$ for all\/ $x$ in\/ $S$ and\/ $\lambda_{g_1g_2} = \lambda_{g_1}\lambda_{g_2}$ for any two composable\/ $g_1$,~$g_2$ in\/ $\varGamma \mathbin| S$. Suppose that\/ $B \cap \varGamma S \subset S$ for some open set\/ $B$ satisfying\/ $\varGamma B = M$. There exists on\/ $E$ a metric such that\/ $\lambda_g$ is an orthogonal linear transformation of\/ $E_{sg}$ onto\/ $E_{tg}$ for every\/ $g \in \varGamma \mathbin| S$. \end{lem}

\begin{proof} Let us endow $E$ with an arbitrary $C^\infty$ vector bundle metric, say, $\phi$. For any choice of a differentiable Haar system on $\varGamma \tto M$ we can find some normalizing function $c$ such that $\supp c \subset B$; this follows at once from the obvious fact that, since $\varGamma B = M$, the prolongation by zero of any cut-off function for $\varGamma \mathbin| B \tto B$ is a cut-off function for $\varGamma \tto M$. If for all $x \in M$, $e_1$,~$e_2 \in E_x$ we put (Haar integral depending on parameters)
\begin{equation}
	\hat{\phi}_x(e_1,e_2) = \integral_{th=x} c(sh)\phi_{sh}(\lambda_{h^{-1}}e_1,\lambda_{h^{-1}}e_2) \der h,
\label{eqn:metric}
\end{equation}
we get a $C^\infty$ field of symmetric bilinear forms $\hat{\phi}_x$ on the fibers of $E$. We contend: $\hat{\phi}_x$ is positive definite for every $x$ in the closure $\overline{S}$ of $S$ and consequently in a whole open neighborhood of $\overline{S}$. To see it, notice first of all that $h$ cannot contribute to the above integral unless $h \in s^{-1}(\supp c)$. If $x$ lies in $\overline{S}$, the intersection $s^{-1}(\supp c) \cap t^{-1}(x) \subset s^{-1}(B) \cap t^{-1}(x)$ is contained in $\overline{\varGamma \mathbin| S}$ because $B \cap \varGamma S \subset S$. For every $h \in \overline{\varGamma \mathbin| S}$, the linear map $\lambda_h$ is invertible, its inverse being $\lambda_{h^{-1}}$. Our contention follows. Next, we have $\hat{\phi}_{tg}(\lambda_ge_1,\lambda_ge_2) = \hat{\phi}_{sg}(e_1,e_2)$ for all $g \in \varGamma \mathbin| S$, $e_1$,~$e_2 \in E_{sg}$; this follows as usual from the left invariance of the Haar system once we have observed that, since $B \cap \varGamma S \subset S$, only those $h$ that lie in $\varGamma \mathbin| S$ can contribute to the integral \eqref{eqn:metric} defining $\hat{\phi}_{tg}$. Then, by using a suitable partition of unity, we may change $\hat{\phi}$ outside the closed set $\overline{S}$ into an everywhere positive definite symmetric bilinear form that has all the properties we want. \end{proof}

Let us consider the orthogonal direct sum decomposition of $\T M \simeq L \oplus L^\bot$ determined by any Riemannian metric on $M$ invariant under the effect of $H \mathbin| U$. By invariance, $\lambda^H$ must carry $L^\bot \mathbin| U$ into itself. We contend that the composite morphism of vector bundles over $\varGamma$%
\begin{subequations}
\label{eqn:nu}
\begin{equation}
	s^*L^\bot \to s^*\T M \xto{\lambda^H} t^*\T M \to t^*L^\bot
\label{eqn:nu/a}
\end{equation}
is a representation of $\varGamma \tto M$ on $L^\bot$. The argument is a simple generalization of the argument we used in the proof of Corollary \ref{thm:rank=0}. We start by noting that for each $g$ in $\varGamma$ there is a unique linear map $\nu_g: L^\bot_{sg} \to L^\bot_{tg}$ which for every linear splitting $\eta_g$ of $\T_gs$ makes the composition
\begin{equation}
	\T_{sg}M \to L^\bot_{sg} \xto{\nu_g} L^\bot_{tg}
\text{\quad equal to}\quad
	\T_{sg}M \xto{\eta_g} \T_g\varGamma \xto{\T_gt} \T_{tg}M \to L^\bot_{tg}.
\label{eqn:nu/b}
\end{equation}
\end{subequations}
($\T_gt$ carries $\ker\T_gs$ into $L_{tg}$, so that the composition is the same for all $\eta_g$, and $L_{sg}$ obviously lies in its kernel.) Since for each pair of arrows $g_1$,~$g_2$ with $sg_1 = tg_2$ both $\eta^H_{g_1g_2}$ and $\eta^H_{g_1}\lambda^H_{g_2} \cdot \eta^H_{g_2}$ are splittings of $\T_{g_1g_2}s$, we have $\nu_{g_1g_2} = \nu_{g_1}\nu_{g_2}$. Similarly $\nu_{1x} = \id$.

By definition $\der t: \ker\der s \onto t^*L$ is an epimorphism of vector bundles over $\varGamma$ and therefore splits; let us fix an arbitrary splitting, say, $\xi: t^*L \to \ker\der s$. Let us write $\pr: \T M \to L$ for the orthogonal projection onto $L$ relative to our Riemannian metric on $M$. The morphism
\begin{equation*}
	\eta^H + \xi \circ (\rho \circ s^*\pr - t^*\pr \circ \lambda^H): s^*\T M \longto \T\varGamma
\end{equation*}
is the horizontal lift for a new connection on $\varGamma \tto M$ which agrees with $H$ along $U$ since $\lambda^H$ carries $L^\bot \mathbin| U$ into itself. If we let $\nu$ denote the representation of $\varGamma \tto M$ on $L^\bot$ given by \eqref{eqn:nu}, at the expense of replacing $H$ with this new connection we may assume that the matrix of $\lambda^H$ relative to the direct sum decomposition $\T M \simeq L \oplus L^\bot$ is
\begin{equation*}
	\lambda^H =%
\begin{pmatrix}
	\rho &  0
\\	  0  & \nu
\end{pmatrix}
\end{equation*}
and hence that $\lambda^H$ is a representation. As we observed during the proof of Lemma \ref{lem:metric}, for each $B$ as in \eqref{eqn:domain} we can find some normalizing function $c$ such that $\supp c \subset B$. If we consider the associated averaging operator, then by \eqref{lem:hat|Z} the mean ratio of $H$ is a multiplicative connection $\hat{H}$ such that $\hat{H} \mathbin| U = H \mathbin| U$: the proof of Theorem \ref{thm:regular} is complete. Note that the effect of $\hat{H}$ is given by the above matrix, so the longitudinal effect of $\hat{H}$ coincides with $\rho$.

\subsubsection*{Supplementary remarks}

The version of Theorem \ref{thm:regular} that we have just proved (as opposed to the version enunciated in subsection \ref{sub:regular}) owes most of its relevance to the role it plays in extending the considerations of subsection \ref{sub:algorithm} from the source-proper to the general proper case. Unfortunately, even a barely satisfactory discussion of this role goes beyond what could still be included in the present paper without unreasonably affecting its overall balance. We shall therefore limit ourselves to a brief commentary on how the current enhanced version of Theorem \ref{thm:regular} may be used to vindicate the claims we made at the beginning of the paragraph preceding Corollary \ref{thm:rank=0}.

Let $\varGamma \tto M$ be a proper Lie groupoid, let $C$ be an invariant closed subset of $M$, and let $V$ be an open neighborhood of $C$. While $V$ need not contain any \emph{invariant} neighborhoods of $C$, the following is always true.

\begin{lem}\label{lem:domain} There exists an open neighborhood\/ $U$ of\/ $C$ with\/ $\overline{U} \subset V$ satisfying\/ \textup{\eqref{eqn:domain}}. \end{lem}

\begin{proof} The orbit space $M/\varGamma$ is locally compact, second countable, and, by properness, Hausdorff. It is therefore normal. Hence, there must be an invariant open set $S = \varGamma S$ containing $C$ with $\overline{S} \subset \varGamma V$. Let $c$ be any cut-off function for $\varGamma \tto M$ with $\supp c \subset V \cup (M \smallsetminus \overline{S})$, and let $B$ be the open set on which $c > 0$. Let $W$ be any open neighborhood of $C$ with $\overline{W} \subset S \cap V$. Then $U = (B \cap S) \cup W$ satisfies $\varGamma U = S$ and, thus, $B \cap \varGamma U \subset U$, and $\overline{U} \subset (\supp c \cap \overline{S}) \cup \overline{W} \subset V$. \end{proof}

Let $U$ be as in the lemma. Let $Z$ be an invariant submanifold of $M$. It is completely obvious that $U \cap Z$ satisfies \eqref{eqn:domain} relative to $\varGamma \mathbin| Z \tto Z$ and that the relative closure of $U \cap Z$ within $Z$ is contained in $V \cap Z$. Now, let $Z = M_r$ be the union of all $r$-dimensional orbits---cf.~\eqref{eqn:12B.19.2}---and let $\varPhi$ be a multiplicative connection on $\varGamma \mathbin| V \tto V$. Suppose that among all $U$ as in the lemma we can find one for which we can prolong the longitudinal effect of $\varPhi \mathbin| U \cap Z$ to a longitudinal representation of the whole $\varGamma \mathbin| Z \tto Z$. (When $r = 0$, this is trivially the case for every $U$. When $r = 1$, we can find one such $U$ provided $\varGamma \tto M$ is source connected; we omit the proof.) We can then invoke our improved version of Theorem \ref{thm:regular} in order to conclude that $\varPhi \mathbin| U \cap Z$ can be extended to all of $\varGamma \mathbin| Z \tto Z$ as a multiplicative connection.

\section{Uniform convergence estimates for recursive averaging}\label{sec:estimates}

As in the previous section let $\varGamma \tto M$ be a proper Lie groupoid endowed with a differentiable Haar system and let $c$ be a normalizing function. Among all nondegenerate connections $H$ on $\varGamma \tto M$, the multiplicative ones are invariant under the averaging operator $H \mapsto \hat{H}$ associated as in Definition \ref{defn:12B.11.2} with these data. In line with celebrated fixed-point theorems in functional analysis inspired by Newton's method such as those described in \cite{Ham82}, one may then expect that whenever $H$ is in a suitable sense “close” to being multiplicative the sequence $H$,~$\hat{H}$,~$\hat{\hat{H}}$,~$\dotsc$ of iterated mean ratios of $H$ (exists and) converges towards some fixed point for this operator which is also a multiplicative connection; actually, since as we know from subsection \ref{sub:proof1} $\hat{H}$ is multiplicative as soon as the effect of $H$ is a representation, one might even suspect that such convergence already takes place whenever the effect of $H$ is “close” to being a representation. As we shall see in the first part of this section and in the next section, these expectations turn out to be correct. Even better, it turns out that in order to show they are, one does not need any of the powerful and sophisticated machinery of \cite{Ham82}; instead, by taking advantage of the nice computational properties of our averaging operator, we shall give a self-contained elementary proof which, besides requiring a minimal amount of background, is more concise and direct than any comparable argument known to us in the literature (cf.~Remarks \ref{rems:Karoubi} below).

In order to simplify our notations we shall deal at first with arbitrary pseudo-representations of $\varGamma \tto M$. Let $E$ be an arbitrary smooth vector bundle over $M$.

\begin{defn}\label{npar:12B.12.1} Let $\lambda: s^*E \simto t^*E$ be a pseudo-representation of $\varGamma \tto M$ on $E$ which is invertible in the sense that each $\lambda_g$ is a bijection of $E_{sg}$ onto $E_{tg}$. By the considerations at the beginning of section \ref{sec:averaging} about Haar integrals depending on parameters, the formula
\begin{equation}
	\hat{\lambda}_g = \integral_{th=sg} c(sh)\lambda_{gh} \circ (\lambda_h)^{-1} \der h
\label{eqn:12B.12.4}
\end{equation}
defines a new pseudo-representation $\hat{\lambda}$ of $\varGamma \tto M$ on $E$, obviously unital, hereafter called the \emph{mean ratio} of $\lambda$. \end{defn}

The correspondence $\lambda \mapsto \hat{\lambda}$, from invertible to unital pseudo-representations, is the analog of our averaging operator, $H \mapsto \hat{H}$, from nondegenerate to unital connections (Definition \ref{defn:12B.11.2}). The above expression for $\hat{\lambda}_g$ generalizes, and is motivated by, our formula \eqref{eqn:12B.11.7} for the effect of $\hat{H}$ in terms of that of $H$:
\begin{equation}
	\lambda^{\hat{H}} = \widehat{(\lambda^H)}.
\label{eqn:12B.11.7*}
\end{equation}

Let us endow $E$ with an arbitrary vector bundle metric of class $C^\infty$, say, $\phi$. For each pair of base points $x$,~$y \in M$, an \emph{operator norm} $\lVert\blank\rVert_{x,y}$ is induced by $\phi$ on the vector space $L(E_x,E_y)$ of all linear maps $\lambda: E_x \to E_y$ through the customary formula
\begin{equation}
	\lVert\lambda\rVert_{x,y} = \sup_{\lvert e\rvert_x\leq 1}{}\lvert\lambda e\rvert_y,
\label{eqn:xy-norms}
\end{equation}
where $\lvert e\rvert_x = \sqrt{\phi_x(e,e)}$ denotes the norm on $E_x$ associated with the inner product $\phi_x$. Of course, these operator norms $\lVert\blank\rVert_{x,y}$ satisfy the following inequalities, for all $x$,\ $y$,\ $z$ in $M$;
\begin{equation}
	\lVert\mu \circ \lambda\rVert_{x,z} \leq \lVert\lambda\rVert_{x,y}\lVert\mu\rVert_{y,z}
\label{eqn:12B.12.6}
\end{equation}
in particular, $\lVert\blank\rVert_x = \lVert\blank\rVert_{x,x}$ turns the ring of endomorphisms $\End(E_x) = L(E_x,E_x)$ into a unital Banach algebra.

\begin{lem}\label{lem:12B.12.3} Let\/ $A$ be a unital Banach algebra, with unit element\/ $1$ and norm\/ $\lVert\blank\rVert$. Let\/ $0 \leq r < 1$ be a real constant. For every element\/ $a$ of\/ $A$ such that\/ $\lVert a\rVert \leq r$, the element\/ $1 - a$ is invertible, and
\begin{equation*}
	\lVert(1 - a)^{-1} - 1\rVert \leq r(1 - r)^{-1}.
\end{equation*} \end{lem}

\begin{proof} Since $\lVert a\rVert < 1$, the element $1 - a$ is invertible, with inverse
\begin{equation*}
	(1 - a)^{-1} = 1 + a + a^2 + a^3 + \dotsb.
\end{equation*}
It follows that $\lVert(1 - a)^{-1} - 1\rVert \leq \lVert a\rVert + \lVert a\rVert^2 + \lVert a\rVert^3 + \dotsb = \lVert a\rVert(1 - \lVert a\rVert)^{-1} \leq r(1 - r)^{-1}$. \end{proof}

For every pseudo-representation $\lambda$ of $\varGamma \tto M$ on $E$, or, more generally, continuous cross-section $\lambda$ of the vector bundle $L(s^*E,t^*E) \to \varGamma$, let us set%
\begin{subequations}
\label{eqn:12B.12.8}
\begin{alignat}{2}
	b(\lambda)&= \sup_{g\in\varGamma}{}\lVert\lambda_g\rVert_{sg,tg} & &\llap{and}
\label{eqn:12B.12.8a}\\
	r(\lambda)&= \sup_{x\in M}{}\lVert\id - \lambda_{1x}\rVert_x + \sup_{(g_1,g_2)\in\varGamma\ftimes{s}{t}\varGamma}{}\lVert\lambda_{g_1g_2} - \lambda_{g_1}\lambda_{g_2}\rVert_{sg_2,tg_1}.
\label{eqn:12B.12.8b}
\end{alignat}
\end{subequations}

\begin{prop}\label{npar:12B.12.7} Let\/ $\lambda$ be a pseudo-representation of\/ $\varGamma \tto M$ on\/ $E$ for which\/ $r(\lambda) < 1$. Then\/ $\lambda$ is invertible, so that\/ $\hat{\lambda}$ is defined. Moreover, whenever\/ $b(\lambda) < \infty$, the following estimates hold.%
\begin{subequations}
\label{eqn:12B.12.10}
\begin{gather}
	\lVert\hat{\lambda}_g\rVert_{sg,tg} \leq \frac{b(\lambda)}{1 - r(\lambda)}
\label{eqn:12B.12.10a}\\
	\lVert\hat{\lambda}_{g_1g_2} - \hat{\lambda}_{g_1}\hat{\lambda}_{g_2}\rVert_{sg_2,tg_1} \leq 2\biggl(\frac{b(\lambda)}{1 - r(\lambda)}\biggr)^2r(\lambda)^2
\label{eqn:12B.12.10b}
\end{gather}
\end{subequations} \end{prop}

\begin{proof} Let us set $r = r(\lambda) < 1$. We have $\lVert\id - \lambda_{g^{-1}}\lambda_g\rVert_{sg} \leq \lVert\id - \lambda_{1sg}\rVert_{sg} + \lVert\lambda_{g^{-1}g} - \lambda_{g^{-1}}\lambda_g\rVert_{sg} \leq r < 1$ for every $g$. Since $\End(E_{sg})$ equipped with the norm $\lVert\blank\rVert_{sg}$ is a unital Banach algebra, by the previous lemma applied to the element $a = \id - \lambda_{g^{-1}}\lambda_g$ of $\End(E_{sg})$ we have that $\lambda_{g^{-1}}\lambda_g$ is an invertible element of $\End(E_{sg})$ and so $\lambda_g$ is an injective linear map. Similarly, by considering $\lambda_g\lambda_{g^{-1}}$, we see that $\lambda_g$ is surjective. This proves the invertibility of $\lambda$.

From Lemma \ref{lem:12B.12.3} applied to $a = \id - \lambda_{g^{-1}}\lambda_g \in \End(E_{sg})$ we get (omitting norm subscripts)
\begin{equation*}
 \lVert(\lambda_{g^{-1}}\lambda_g)^{-1} - \id\rVert
	= \bigl\lVert\bigl(\id - (\id - \lambda_{g^{-1}}\lambda_g)\bigr)^{-1} - \id\bigr\rVert
	\leq r(1 - r)^{-1}.
\end{equation*}
Using the inequalities \eqref{eqn:12B.12.6} we then obtain
\begin{equation*}
 \lVert(\lambda_g)^{-1} - \lambda_{g^{-1}}\rVert
	= \bigl\lVert\bigl((\lambda_{g^{-1}}\lambda_g)^{-1} - \id\bigr) \circ \lambda_{g^{-1}}\bigr\rVert
	\leq r(1 - r)^{-1}\lVert\lambda_{g^{-1}}\rVert,
\end{equation*}
whence
\begin{equation}
 \lVert(\lambda_g)^{-1}\rVert
	\leq \lVert\lambda_{g^{-1}}\rVert + \lVert(\lambda_g)^{-1} - \lambda_{g^{-1}}\rVert
	\leq \left(1 + \frac{r}{1 - r}\right)\lVert\lambda_{g^{-1}}\rVert
	\leq \frac{b(\lambda)}{1 - r(\lambda)}.
\label{eqn:12B.12.9a}
\end{equation}
From \eqref{eqn:12B.12.6} and the latter inequality we conclude that for every composable pair of arrows $g$,~$h$
\begin{equation}
 \lVert\lambda_{gh} \circ (\lambda_h)^{-1} - \lambda_g\rVert
	\leq \lVert\lambda_{gh} - \lambda_g\lambda_h\rVert\:\lVert(\lambda_h)^{-1}\rVert
	\leq r(\lambda)\frac{b(\lambda)}{1 - r(\lambda)}.
\label{eqn:12B.12.9b}
\end{equation}

Now, because of the normalizing function property (the equality $\integral_{th=x} c(sh) \der h = 1$), we have%
\begin{subequations}
\label{eqn:12B.12.5}
\begin{align}
 \hat{\lambda}_g &
	= \integral_{th=sg} c(sh)\lambda_g \der h + \integral_{th=sg} c(sh)[\lambda_{gh} \circ (\lambda_h)^{-1} - \lambda_g] \der h \notag\\ &
	= \lambda_g + \integral_{th=sg} c(sh)[\lambda_{gh} \circ (\lambda_h)^{-1} - \lambda_g] \der h
\label{eqn:12B.12.5a}
\end{align}
and we can estimate the last integral by the sup norm of the bracketed term inside its integrand. The first one of the two inequalities \eqref{eqn:12B.12.10} then follows from \eqref{eqn:12B.12.9b}. As to the other inequality, \eqref{eqn:12B.12.10b}, because of the Haar system's left invariance and, again, of the normalizing function property, we have (once more omitting part of the notations for the sake of conciseness and readability)
\begin{align}
 \hat{\lambda}_{g_1g_2} - \hat{\lambda}_{g_1}\hat{\lambda}_{g_2} &
	= \hat{\lambda}_{g_1g_2} - \left(\integral c(sh')\lambda_{g_1h'}\lambda_{h'}^{-1} \der h'\right) \circ \hat{\lambda}_{g_2} \notag\\ &\hskip-2em
	= \hat{\lambda}_{g_1g_2} - \integral c(sh)\lambda_{g_1g_2h}\lambda_{g_2h}^{-1} \circ \hat{\lambda}_{g_2} \der h \notag\\
\begin{split} &\hskip-2em
	= \integral c(sh)\lambda_{g_1g_2h}\lambda_h^{-1} \der h \\ &\hskip-2em\justify
	  - \integral c(sh)\lambda_{g_1g_2h}\lambda_{g_2h}^{-1} \circ \lambda_{g_2} \der h
	  - \integral c(sh)\lambda_{g_1} \circ \lambda_{g_2h}\lambda_h^{-1} \der h
	  + \lambda_{g_1}\lambda_{g_2} \\ &\hskip-2em\justify
	  + \integral c(sh)\lambda_{g_1g_2h}\lambda_{g_2h}^{-1} \circ \lambda_{g_2} \der h
	  + \integral c(sk)\lambda_{g_1} \circ \lambda_{g_2k}\lambda_k^{-1} \der k
	  - \lambda_{g_1}\lambda_{g_2} \\ &\hskip-2em\justify
	  - \iintegral c(sh)c(sk)\lambda_{g_1g_2h}\lambda_{g_2h}^{-1} \circ \lambda_{g_2k}\lambda_k^{-1} \der h \der k
\end{split}
\notag\\ &\hskip-4em\!%
\begin{aligned}[b] &
	= \integral_{th=sg_2} c(sh)[\lambda_{g_1g_2h}\lambda_{g_2h}^{-1} - \lambda_{g_1}] \circ [\lambda_{g_2h}\lambda_h^{-1} - \lambda_{g_2}] \der h \\ &\justify
	  - \iintegral_{\substack{th=sg_2\\ tk=sg_2}} c(sh)c(sk)[\lambda_{g_1g_2h}\lambda_{g_2h}^{-1} - \lambda_{g_1}] \circ [\lambda_{g_2k}\lambda_k^{-1} - \lambda_{g_2}] \der h \der k
\end{aligned}
\label{eqn:12B.12.5b}
\end{align}
\end{subequations}
and we can estimate each integral by the sup norm of the composite bracketed expression inside its integrand, whence on account of \eqref{eqn:12B.12.6} and \eqref{eqn:12B.12.9b} our inequality. \end{proof}

The above proposition suggests that, whenever $\lambda$ is ``close enough'' to being a representation, its mean ratio $\hat{\lambda}$ must be ``even closer'' to being one. Let us make this idea precise.

\begin{defn}\label{defn:12B.14.1} The pseudo-representation $\lambda$ of $\varGamma \tto M$ on $E$ is a \emph{near representation} if, for some choice of a vector bundle metric on $E$, the quantities $b(\lambda)$ and $r(\lambda)$ given by \eqref{eqn:12B.12.8} turn out to be both finite and to be related as follows.
\begin{equation}
	r(\lambda) \leq \min\{\tfrac{1}{4},\tfrac{1}{9}b(\lambda)^{-2}\}
\label{eqn:12B.14.1}
\end{equation} \end{defn}

We observe that a near representation $\lambda$ is necessarily invertible (by Proposition \ref{npar:12B.12.7}), so that $\hat{\lambda}$ is defined. We claim that the (unital) pseudo-representation $\hat{\lambda}$ is itself a near representation. To see this, let us set $b_0 = b(\lambda)$, $r_0 = r(\lambda)$, $b_1 = b(\hat{\lambda})$, and $r_1 = r(\hat{\lambda})$. By \eqref{eqn:12B.12.10a} and \eqref{eqn:12B.14.1}, we have $b_1 \leq b_0/(1 - r_0) \leq \frac{4}{3}b_0$, in particular, $b_1$ is finite. Since $\hat{\lambda}$ is unital, we deduce from \eqref{eqn:12B.12.10b} that $r_1 \leq 2[b_0/(1 - r_0)]^2r_0^2 \leq 2\frac{16}{9}b_0^2r_0^2$ and therefore from \eqref{eqn:12B.14.1} that $r_1 \leq 2\frac{16}{9}\frac{1}{9}r_0 \leq \frac{1}{2}r_0 \leq 1/8$ and
\begin{equation*}
 b_1^2r_1
	\leq 2(\tfrac{16}{9} \cdot b_0^2r_0)^2
	\leq 2(2 \cdot \tfrac{1}{9})^2
	< 1/9.
\end{equation*}
Our claim is proven. For the record: we have shown in addition that besides having $b_1 \leq \frac{4}{3}b_0$ we also have $r_1 \leq \frac{1}{2}r_0$.

Thus, whenever $\lambda$ is a near representation, we can go on forever taking mean ratios: we obtain a whole sequence $\hat{\lambda}^0$,~$\hat{\lambda}^1$,~$\hat{\lambda}^2$,~$\dotsc$ of \emph{averaging iterates} of $\lambda$ which we construct recursively by setting $\hat{\lambda}^0 = \lambda$ and $\hat{\lambda}^{i+1} = \widehat{(\hat{\lambda}^i)}$ for all $i$. We contend that this sequence is \emph{Cauchy} and hence \emph{convergent} in the Banach space of all those cross-sections $\lambda$ of the vector bundle $L(s^*E,t^*E) \to \varGamma$ that are continuous and that are \emph{bounded} i.e.~have \emph{finite\/ $C^0$-norm:} $\lVert\lambda\rVert_{C^0} = b(\lambda) < \infty$. To show this, let us set $b_i = b(\hat{\lambda}^i)$ and $r_i = r(\hat{\lambda}^i)$. By recursively invoking the last sentence of the previous paragraph, we see that $b_i \leq (\frac{4}{3})^ib_0$ and $r_i \leq (\frac{1}{2})^ir_0$. Then, for every $g$ in $\varGamma$, by \eqref{eqn:12B.12.5a} and \eqref{eqn:12B.12.9b},
\begin{align*}
 \lVert\hat{\lambda}^{i+1}_g - \hat{\lambda}^i_g\rVert &
	\leq \integral_{th=sg} c(sh)\lVert\hat{\lambda}^i_{gh} \circ (\hat{\lambda}^i_h)^{-1} - \hat{\lambda}^i_g\rVert \der h \\ &
	\leq \frac{1}{1 - r_i} \cdot b_ir_i
	\leq \frac{1}{1 - r_0} \cdot (\tfrac{4}{3})^i(\tfrac{1}{2})^ib_0r_0
	\leq (\tfrac{2}{3})^ib_0/3.
\end{align*}
It follows at once that the sequence $\hat{\lambda}^0$,~$\hat{\lambda}^1$,~$\hat{\lambda}^2$,~$\dotsc$ is Cauchy, and therefore converges to a unique continuous (bounded) cross-section $\hat{\lambda}^\infty$ of $L(s^*E,t^*E)$. Since $C^0$-convergence implies pointwise convergence, for each $g$ we have $\hat{\lambda}^\infty_g = \lim \hat{\lambda}^i_g$ in the finite-dimensional vector space $L(E_{sg},E_{tg})$. It follows that $\hat{\lambda}^\infty_{1x} = \lim\hat{\lambda}^i_{1x} = \id$ for all $x$ in $M$, because $\hat{\lambda}^i$ is unital for every $i \geq 1$, and also that $\hat{\lambda}^\infty_{g_1g_2} = \hat{\lambda}^\infty_{g_1} \circ \hat{\lambda}^\infty_{g_2}$ for any two composable $g_1$,\ $g_2$, because, by continuity,
\begin{align*}
 \lVert\hat{\lambda}^\infty_{g_1g_2} - \hat{\lambda}^\infty_{g_1} \circ \hat{\lambda}^\infty_{g_2}\rVert
	= \lim{}\lVert\hat{\lambda}^i_{g_1g_2} - \hat{\lambda}^i_{g_1} \circ \hat{\lambda}^i_{g_2}\rVert
	\leq \lim r_i
	\leq \lim 2^{-i}r_0
	= 0.
\end{align*}

We shall see in section \ref{sec:proof} that the sequence of averaging iterates of any near representation converges actually \emph{much faster} than the sequence of partial sums of a geometric series. In fact, the convergence is so fast that it enforces not only the \emph{continuity} of the limiting cross-section but also its \emph{differentiability} to any order:

\begin{thm}\label{thm:12B.14.2} Let\/ $\lambda$ be a near representation of the proper Lie groupoid\/ $\varGamma \tto M$ on the vector bundle\/ $E$. For any choice of a differentiable Haar system and normalizing function, the sequence of iterated mean ratios of\/ $\lambda$ obtained by recursive application of the averaging formula\/ \eqref{eqn:12B.12.4}
\begin{equation*}
	\hat{\lambda}^0 = \lambda, \quad
	\hat{\lambda}^1 = \hat{\lambda}, \quad
	\hat{\lambda}^2 = \hat{\hat{\lambda}},~\dotsc, \quad
	\hat{\lambda}^{i+1} = \widehat{(\hat{\lambda}^i)},~\dotsc
\end{equation*}
converges pointwise as a sequence of global\/ $C^\infty$ cross-sections of the vector bundle\/ $L(s^*E,t^*E) \to \varGamma$ towards a unique ($C^\infty$-differentiable) representation\/ $\hat{\lambda}^\infty$ of\/ $\varGamma \tto M$ on\/ $E$. \end{thm}

\begin{rems}\label{rems:Karoubi} The principal result of \cite{delaHK} states that if one lets $\GL(E)$ denote the group of all invertible bounded linear operators on a Banach space $E$ then any “almost homomorphism” $\lambda$ of a compact group $G$ into $\GL(E)$ which is continuous for the operator norm topology on $\GL(E)$ is a “small” perturbation of an actual homomorphism of topological groups $G \to \GL(E)$, where “small” roughly means “of the same order of magnitude as $r(\lambda)$.” It is clear that the arguments we have given so far in this section do not in any way depend on the finite-dimensionality of the fibers of the vector bundle $E \to M$ which our groupoid $\varGamma \tto M$ is supposed to act on and that from them there follows at once a result for proper groupoids which is a natural generalization of the one we have just mentioned. We emphasize that our proof is not only quite different from, as well as more general than, that of de~la~Harpe and Karoubi, it is also much simpler in terms of both the ideas it involves and the amount of work it takes. While ours essentially reduces to proving Proposition \ref{npar:12B.12.7} and the two inequalities $b_1 \leq \frac{4}{3}b_0$ and $r_1 \leq \frac{1}{2}r_0$, theirs takes up several pages and involves a preliminary study of the properties of idempotents in a Banach algebra, holomorphic functional calculus, and a series of lemmas about the regular representation of a compact group with Banach module coefficients.

A somewhat closer analogy can be drawn between our averaging formula \eqref{eqn:12B.12.4} and the “center of mass construction” of Grove, Karcher and Ruh \cite{GKR}, and the subsequent adaptation thereof by Zung \cite{Zung}. Similar considerations apply to our use of \emph{recursive} averaging, an idea which is not present in \cite{delaHK}. There are also noticeable differences, though. While our formula may be seen as a “global” counterpart of these constructions, at the same time it contains simplifications which are only possible in the “linear algebraic” setting of pseudo-representations and which enhance its performance in computations. Its ultimate justification, in any case, lies in the cohomological considerations of subsection \ref{sub:proof1}, rather than in the constructions of the cited references. \end{rems}

\subsubsection*{The case of connections}

We proceed to discuss an important application of the preceding theory to the situation where $E = \T M$ and $\lambda = \lambda^H$ is the effect of a \emph{nondegenerate} connection $H$ on $\varGamma \tto M$. (Cf.~Definition \ref{defn:12B.11.2}.) We shall make use of the basic identity \eqref{eqn:12B.11.7*} and of the notations of subsection \ref{sub:proof1} typically without express notice.

We begin by spelling out a couple of formulas which, in the context of connections, may be looked upon as the analogs of \eqref{eqn:12B.12.5a} and \eqref{eqn:12B.12.5b}. First, if we define $\varDelta^H$ as in \eqref{eqn:14A.14.2}, then for every $g$ in $\varGamma$%
\begin{subequations}
\begin{equation}
	\omega_g \circ (\eta^{\hat{H}}_g - \eta^H_g) = \integral_{th=sg} c(sh)\varDelta^H(g,h)\lambda_g \der h;
\label{eqn:12B.15.5a*}
\end{equation}
this is merely a rewrite of equation \eqref{eqn:12B.15.5a}, whose proof in subsection \ref{sub:proof1} did obviously not depend on $\lambda = \lambda^H$ being a representation. Second, for every composable pair of arrows $(g,h) \in \varGamma_2$,
\begin{align}
 -R^{\hat{H}}(g,h) &
	= \integral_{tk=sh} c(sk)\varDelta^H(g,hk)\lambda_g \circ [\lambda_{hk}\lambda_k^{-1} - \lambda_h] \der k \notag\\* &\justify
	  - \iintegral_{\substack{tk=sh\\ tk'=sh}} c(sk)c(sk')\varDelta^H(g,hk)\lambda_g \circ [\lambda_{hk'}\lambda_{k'}^{-1} - \lambda_h] \der k \der k'.
\label{eqn:12B.15.5b}
\end{align}
\end{subequations}
The latter formula can be established as follows. If we let $X$ correspond as in \eqref{eqn:14A.14.3} to $H' = \hat{H}$, that is to say, in view of \eqref{eqn:12B.15.5a}, if $X(g) = \integral_{th=sg} c(sh)\varDelta^H(g,h) \der h$ for all $g$, then, by successively using the invariance of the Haar system, equation \eqref{eqn:cocycle*}, and the normalizing function property, we obtain
\begin{align*} &
 \alpha^H_g \circ X(h)\lambda_h - X(gh)\lambda_{gh} + X(g)\lambda_g \circ \lambda_h \\* &\quad
	= \integral_{tk=sh} c(sk)[\alpha^H_g \circ \varDelta^H(h,k)\lambda_h - \varDelta^H(gh,k)\lambda_{gh}] \der k \\* &\quad\justify
	  + \integral_{tk'=sg} c(sk')\varDelta^H(g,k')\lambda_g \circ \lambda_h \der k' \\ &\quad
	= \integral_{tk=sh} c(sk)[\alpha^H_g\varDelta^H(h,k)\lambda_h - \varDelta^H(gh,k)\lambda_{gh} + \varDelta^H(g,hk)\lambda_g\lambda_h] \der k \\ &\quad
	= \integral_{tk=sh} c(sk)[-R^H(g,h) - \varDelta^H(g,hk)\lambda_g\lambda_{hk}\lambda_k^{-1} + \varDelta^H(g,hk)\lambda_g\lambda_h] \der k \\ &\quad
	= -R^H(g,h) - \integral_{tk=sh} c(sk)\varDelta^H(g,hk)\lambda_g \circ [\lambda_{hk}\lambda_k^{-1} - \lambda_h] \der k.
\end{align*}
Substituting into equation \eqref{eqn:defect} and expanding the first term $X(g)\lambda_g \circ (\hat{\lambda}_h - \lambda_h)$ in the right-hand side of that equation by means of \eqref{eqn:12B.12.5a}, we arrive at \eqref{eqn:12B.15.5b}.

Let us suppose now that $\lambda = \lambda^H$ is a \emph{near representation.} In such case we obtain a whole sequence $H$,~$\hat{H}$,~$\hat{\hat{H}}$,~$\dotsc$ of nondegenerate connections $\hat{H}^i$ whose effects $\lambda^{\hat{H}^i}$ coincide with the averaging iterates $\hat{\lambda}^i$ by setting $\hat{H}^0 = H$ and, recursively, $\hat{H}^{i+1} = \widehat{(\hat{H}^i)}$ for all $i \geq 0$. We contend that the horizontal lifts $\eta^{\hat{H}^i}$ converge pointwise over $\varGamma$ as a sequence of cross-sections of the vector bundle $L(s^*\T M,\T\varGamma)$. To see it, let us fix an arbitrary vector bundle metric on $\Lie\varGamma$, as well as a metric on $E = \T M$ of the kind specified in Definition \ref{defn:12B.14.1}. Note that the inequalities \eqref{eqn:12B.12.6} are also valid for $\mu \in L(\T_yM,\Lie_z\varGamma)$ provided we interpret $\lVert\blank\rVert_{y,z}$ as the operator norm on $L(\T_yM,\Lie_z\varGamma)$. Our contention is evidently tantamount to saying that, for any given $g$ in $\varGamma$, the sequence (or series) \[%
	\omega_g \circ (\eta^{\hat{H}^{i+1}}_g - \eta^H_g) = \omega_g \circ (\eta^{\hat{H}^{i+1}}_g - \eta^{\hat{H}^i}_g) + \dotsb + \omega_g \circ (\eta^{\hat{H}}_g - \eta^H_g)
\] is Cauchy within the finite-dimensional normed vector space $L(\T_{sg}M,\Lie_{tg}\varGamma)$. Now, for all $h$ in the compact set $K = s^{-1}(\supp c) \cap t^{-1}(sg)$, in the notations introduced after Definition \ref{defn:12B.14.1}, we have
\begin{alignat*}{2} &
 \lVert\varDelta^{\hat{H}^{i+1}}(g,h)\hat{\lambda}^{i+1}_g\rVert
	\leq \lVert R^{\hat{H}^{i+1}}(g,h)\rVert\:\lVert(\hat{\lambda}^{i+1}_h)^{-1}\rVert &\quad &
		\text{by \eqref{eqn:14A.14.2} and \eqref{eqn:12B.12.6}}
\\ &\quad
	\leq 2\sup_{k\in h^{-1}K}{}\lVert\varDelta^{\hat{H}^i}(g,hk)\hat{\lambda}^i_g\rVert \cdot \frac{b_ir_i}{1 - r_i} \cdot \frac{b_{i+1}}{1 - r_{i+1}} &\quad &
		\text{by \eqref{eqn:12B.15.5b}, \eqref{eqn:12B.12.9b}, and \eqref{eqn:12B.12.9a}}
\\ &\quad
	\leq 2\sup_{hk\in K}{}\lVert\varDelta^{\hat{H}^i}(g,hk)\hat{\lambda}^i_g\rVert \cdot \frac{b_ir_i}{1 - r_i} \cdot \frac{4}{3}\frac{b_i}{1 - r_i} &\quad &
		\text{by \eqref{eqn:12B.12.10a} and \eqref{eqn:12B.14.1} for $\hat{\lambda}^{i+1}$}
\\ &\quad
	\leq \tfrac{2}{3}\sup_{h\in K}{}\lVert\varDelta^{\hat{H}^i}(g,h)\hat{\lambda}^i_g\rVert &\quad &
		\text{by \eqref{eqn:12B.14.1} for $\hat{\lambda}^i$}
\end{alignat*}
and therefore, by \eqref{eqn:12B.15.5a*} and induction on $i \geq 0$,
\begin{equation*}
 \lVert\omega_g \circ (\eta^{\hat{H}^{i+1}}_g - \eta^{\hat{H}^i}_g)\rVert
	\leq \sup_{h\in K}{}\lVert\varDelta^{\hat{H}^i}(g,h)\hat{\lambda}^i_g\rVert
	\leq (\tfrac{2}{3})^i\sup_{h\in K}{}\lVert\varDelta^H(g,h)\lambda^H_g\rVert.
\end{equation*}
It follows that our sequence converges at least as fast as the partial sums of a geometric series. In fact, our argument shows that the convergence is uniform on every compact subset of $\varGamma$. The limiting cross-section, let us call it $\eta^{\hat{H}^\infty}$, is therefore continuous. It is obviously a splitting of $\der s$. It is also unital, because so is every $\hat{H}^i$, $i \geq 1$. We leave it as an exercise for the reader to check that $\eta^{\hat{H}^\infty}$ further satisfies the multiplicativity equations \eqref{prop:12B.9.4}.

\begin{thm}\label{thm:12B.15.1} Let\/ $H$ be a connection on the proper Lie groupoid\/ $\varGamma \tto M$. Suppose that the effect\/ $\lambda^H$ of\/ $H$ is a near representation. Then, for any choice of a differentiable Haar system and normalizing function, the sequence of iterated mean ratios of\/ $H$
\begin{equation*}
	\hat{H}^0 = H, \quad
	\hat{H}^1 = \hat{H}, \quad
	\hat{H}^2 = \hat{\hat{H}},~\dotsc, \quad
	\hat{H}^{i+1} = \widehat{(\hat{H}^i)},~\dotsc
\end{equation*}
(cf.~\textup{Definition \ref{defn:12B.11.2}}) is pointwise convergent, when viewed via the associated horizontal lifts\/ $\eta^{\hat{H}^i}$ as a sequence of global\/ $C^\infty$ cross-sections of the vector bundle\/ $L(s^*\T M,\T\varGamma)$, towards a unique multiplicative ($C^\infty$-differentiable) connection\/ $\hat{H}^\infty$ on\/ $\varGamma \tto M$. \end{thm}

We shall refer to Theorems \ref{thm:12B.14.2} and \ref{thm:12B.15.1} collectively as the ``fast convergence theorem.'' Their proofs will be completed in section \ref{sec:proof}. The remainder of the current section will be devoted to explaining how the principal results of section \ref{sec:results} can be deduced from Theorem \ref{thm:12B.15.1}.

\subsection{Fast convergence implies Theorem \ref{thm:main}}\label{sub:deducing2}

Let us go back to the situation described in the statement of Theorem \ref{thm:main}. The idea of the proof is simple: there must be some open neighborhood $V$ of $S$ such that the effect of $H \mathbin| V$ is a near representation of $\varGamma \mathbin| V \tto V$; once this is proven, Theorem \ref{thm:main} will be a direct consequence of Theorem \ref{thm:12B.15.1} on account of the remark that it is not restrictive to assume $S = \overline{S}$ to be \emph{closed,} in addition to being invariant (the closure $\overline{S}$ of any invariant set $S$ is itself invariant, and if $\lambda^H$ is a representation over $S$, then it is also a representation over $\overline{S}$, by invariance and continuity); it will suffice to pick our normalizing function $c$ on $V$ so that $\supp c \subset B$, and then define $\varPhi$ to be the limit of the iterated mean ratios of $H \mathbin| V$; since obviously apart from multiplicativity the properties enounced in the conclusions of Theorem \ref{thm:main} are valid for every iterated mean ratio of $H \mathbin| V$, as well as stable under passage to the pointwise limit, they will also be valid for $\varPhi$.

In order to simplify the notations, instead of $\lambda^H$ we may as well consider an arbitrary pseudo-representation $\lambda$ of $\varGamma \tto M$ which is a representation over $S$ in the sense that $\lambda_{1x} = \id$ for all $x$ in $S$ and $\lambda_{g_1g_2} = \lambda_{g_1}\lambda_{g_2}$ for any two composable $g_1$,\ $g_2$ in $\varGamma \mathbin| S$; we can do this since for any $V$ the restriction of $\lambda^H$ to $\varGamma \mathbin| V \tto V$ is equal to $\lambda^{H|V}$. Let $E$ be the vector bundle on which $\lambda$ operates. Lemma \ref{lem:metric} gives us a vector bundle metric on $E$ such that $\lambda_g$ is an isometry of $E_{sg}$ onto $E_{tg}$ for every $g$ in $\varGamma \mathbin| S$. Relative to any such metric, let us form the following three open sets.
\begin{gather*}
	\varOmega_0 = \{x \in M: \lVert\id - \lambda_{1x}\rVert < 1/36\} \\
	\varOmega_1 = \{g \in \varGamma: \lVert\lambda_g\rVert < \sqrt{2}\} \\
	\varOmega_2 = \{(g_1,g_2) \in \varGamma \ftimes{s}{t} \varGamma: \lVert\lambda_{g_1g_2} - \lambda_{g_1}\lambda_{g_2}\rVert < 1/36\}
\end{gather*}
Clearly, since $\lambda$ is a representation over $S$, $\varOmega_0$ contains $S$, and $\varOmega_2$ contains $\varGamma \mathbin| S \ftimes{s}{t} \varGamma \mathbin| S$. Also, by the invariance of our metric under the restriction of $\lambda$ over $S$, $\varOmega_1$ contains $\varGamma \mathbin| S$. So, if we can find an open neighborhood $V$ of $S$ such that
\begin{equation}
	\text{$V \subset \varOmega_0$,\quad $\varGamma \mathbin| V \subset \varOmega_1$,\quad and\quad $\varGamma \mathbin| V \ftimes{s}{t} \varGamma \mathbin| V \subset \varOmega_2$,}
\label{eqn:deducing2}
\end{equation}
then $\lambda$ will be a near representation over $V$ because
\begin{equation*}
 \sup_{x\in V}{}\lVert\id - \lambda_{1x}\rVert
		+ \sup_{\substack{g_1,g_2\in\varGamma|V\\ sg_1=tg_2}}{}\lVert\lambda_{g_1g_2} - \lambda_{g_1}\lambda_{g_2}\rVert
	\leq 1/18
	\leq \tfrac{1}{9}(\sup_{g\in\varGamma|V}{}\lVert\lambda_g\rVert)^{-2}.
\end{equation*}

We begin by showing that, for every open subset $U$ of $M$ which is relatively compact in the sense that its closure $\overline{U}$ is compact, each point $x$ in $S \cap U$ admits an open neighborhood $V$ within $U$ which satisfies \eqref{eqn:deducing2} and which is relatively invariant in the sense that $V = U \cap \varGamma V$. In order to do so, let us pick any monotone sequence $B_0 \supset B_1 \supset B_2 \supset \dotsb$ of relatively compact open neighborhoods $B_n$ of $x$ within $U$ such that $\bigcap B_n = \{x\}$. Since $\overline{B}_n$ is compact, so is
\begin{equation*}
	\overline{U} \cap \varGamma\overline{B}_n = t\bigl(s^{-1}(\overline{B}_n) \cap t^{-1}(\overline{U})\bigr).
\end{equation*}
Let us put $A_n = \varGamma \mathbin|(\overline{U} \cap \varGamma \overline{B}_n)$. The $A_n$ form a monotone decreasing sequence of compact sets in $\varGamma$. It is not hard to see that $\bigcap A_n$ is contained in $\varGamma \mathbin| S$ and hence that $\bigcap(A_n \ftimes{s}{t} A_n)$ is contained in $\varGamma \mathbin| S \ftimes{s}{t} \varGamma \mathbin| S$. Thus for $n$ large we must have $A_n \subset \varGamma \mathbin| \varOmega_0$, $A_n \subset \varOmega_1$, and $A_n \ftimes{s}{t} A_n \subset \varOmega_2$, so that if we take $V = U \cap \varGamma B_n$ we are done.

We observe next that for every open subset $U$ of $M$ the union $\bigcup V$ of all relatively invariant open subsets $V = U \cap \varGamma V$ of $U$ which satisfy \eqref{eqn:deducing2} is itself one such $V$. Indeed, by the relative invariance of each $V$, we have $\varGamma \mathbin| \bigcup V = \bigcup \varGamma \mathbin| V$ and $\varGamma \mathbin| \bigcup V \ftimes{s}{t} \varGamma \mathbin| \bigcup V = \bigcup(\varGamma \mathbin| V \ftimes{s}{t} \varGamma \mathbin| V)$. If $U$ is also relatively compact, then by what we have shown in the previous paragraph $\bigcup V$ must contain $S \cap U$.

Let us now fix an arbitrary monotone sequence $U_0 \subset U_1 \subset U_2 \subset \dotsb$ of relatively compact open subsets $U_n$ of $M$ such that $\bigcup U_n = M$. For each $n$ let $V_n$ be the largest relatively invariant open neighborhood $V = U_n \cap \varGamma V$ of $S \cap U_n$ within $U_n$ that satisfies \eqref{eqn:deducing2}. Then $V = \bigcup V_n$ must satisfy \eqref{eqn:deducing2} as well because for $m < n$ we have
\begin{align*}
 V_m \cap \varGamma V_n
	= (V_m \cap U_m) \cap \varGamma V_n
	\subset V_m \cap (U_n \cap \varGamma V_n)
	= V_m \cap V_n
\end{align*}
and this implies at once that $\varGamma \mathbin| V = \bigcup \varGamma \mathbin| V_n$ and that $\varGamma \mathbin| V \ftimes{s}{t} \varGamma \mathbin| V = \bigcup(\varGamma \mathbin| V_n \ftimes{s}{t} \varGamma \mathbin| V_n)$. Of course $V = \bigcup V_n$ is an open neighborhood of $S$. The proof is finished.

\subsection{Deducing Theorem \ref{prop:14A.5.2} from Theorem \ref{thm:main}}\label{sub:deducing1}

Let the situation be now as in the statement of Theorem \ref{prop:14A.5.2}. It is clear that in order to prove our theorem we may always replace $M$ with an open neighborhood $V'$ of $C \cup Z$ and $V$ with an open neighborhood of $C$ contained in $V \cap V'$. We are going to show that, at the expense of doing so, we can always reduce the proof to the situation where there exists a connection $H$ on $\varGamma \tto M$ such that $H \mathbin| V = \varPhi$, $H \mathbin| Z = \varPsi$, and $\lambda^H$ is a representation over $Z$ and hence over the invariant set $C \cup Z$. Once this is done, we shall find ourselves in the condition of applying Theorem \ref{thm:main} to the connection $H$ and to the invariant set $S = C \cup Z$ over which $\lambda^H$ is a representation: over some open neighborhood of $C \cup Z$, which we may suppose to be all of $M$, there will exist, for any choice of an open subset $B$ of $M$ such that $\varGamma B = M$, a multiplicative connection $\varPhi'$ with the properties enunciated for $\varPhi$ in the statement of Theorem \ref{thm:main}. Now, by Lemma \ref{lem:domain}, we know there is some open neighborhood $U$ of $C$ such that $\overline{U} \subset V$ and such that $B \cap \varGamma U \subset U$ for some such $B$. We contend that every $\varPhi'$ associated with this particular choice of $B$ will agree with $\varPhi$ over $V \cap \varGamma U \supset C$ and will induce $\varPsi$ along $Z$. The latter property is clear since $Z$ is an invariant submanifold of $M$ along which $H \mathbin| Z = \varPsi$ is multiplicative and so $\varPhi' \mathbin| Z = H \mathbin| Z = \varPsi$. As to the former, we observe that $V \cap \varGamma U$ is an open subset of $M$ (hence a locally invariant submanifold) such that
\begin{equation*}
 B \cap \varGamma(V \cap \varGamma U)
	\subset B \cap \varGamma(\varGamma U)
	= B \cap \varGamma U
	\subset U
	\subset V \cap \varGamma U
\end{equation*}
along which $H$ induces a multiplicative connection, namely, $\varPhi \mathbin| V \cap \varGamma U$. This proves our contention and, therefore, establishes Theorem \ref{prop:14A.5.2} in the special case when there exists a connection $H$ satisfying our requirements. The rest of this subsection will be devoted to constructing one such connection. The argument, which we shall organize into three steps, will be close in spirit to the proof of Theorem \ref{thm:regular} though a bit more involved.

\paragraph*{Step 1.} To begin with, we explain how to reduce the proof to the situation where on $\varGamma \tto M$ we have a connection $H$ such that both $H \mathbin| V = \varPhi$ and $H \mathbin| Z = \varPsi$. (At this stage, $\lambda^H$ need not yet be a representation over $Z$; in the final step, we shall indicate how to ``correct'' $H$ so as to turn it into a new connection having all of the required properties.)

We notice first of all that there is no loss of generality in assuming the union $C \cup Z$ to be \emph{closed} in $M$. This condition can always be achieved by first replacing $M$ with a suitable open neighborhood of $V \cup Z$ and then making $C$ relatively larger inside that neighborhood, as follows. Since $Z$ is a submanifold of $M$, it is locally closed, hence it admits an open neighborhood $W$ with $W \cap \overline{Z} \subset Z$. The difference between the relative closure of $Z$ in $V \cup W$ and $Z$ itself is contained in $V$, hence at the expense of substituting $M$ with $V \cup W$ we may assume that the two invariant closed sets $\overline{Z} \smallsetminus Z = \overline{Z} \smallsetminus W$ and $C' = C \cup (\overline{Z} \smallsetminus Z)$ are contained in $V$. The union $C' \cup Z = C \cup \overline{Z}$ is closed now.

We shall obtain $H$ by means of a simple extension principle which we shall use again in the course of the proof on two more occasions. Let $E$ be a smooth vector bundle over $\varGamma$. Let $\vartheta$ be a $C^\infty$-differentiable partial cross-section of $E$ defined over the open neighborhood $\varGamma \mathbin| V$ of $s^{-1}(C)$, and let $\zeta$ be a similar cross-section defined over the submanifold $\varGamma \mathbin| Z = s^{-1}(Z)$ of $\varGamma$ and equal to $\vartheta$ on $\varGamma \mathbin| V \cap \varGamma \mathbin| Z = \varGamma \mathbin| V \cap Z$. \em Then, for any choice of an open neighborhood\/ $U$ of\/ $C$ such that\/ $\overline{U} \subset V$, there exists a global\/ $C^\infty$ cross-section of\/ $E$ which agrees with\/ $\vartheta$ on\/ $\varGamma \mathbin| U$ and with\/ $\zeta$ on\/ $\varGamma \mathbin| Z$. \em Proof: we can cover $\varGamma \mathbin| Z$ with open subsets $\varOmega$ of $\varGamma$ such that $\zeta$ extends to $\varOmega$ in a $C^\infty$ fashion; since $C \cup Z$ is closed, $s^{-1}(C) \cup s^{-1}(Z) = s^{-1}(C \cup Z)$ is a closed subset of $\varGamma$, and hence, a fortiori, so is $\varGamma \mathbin| \overline{U} \cup \varGamma \mathbin| Z$; we may then use a partition of unity on $\varGamma$ subordinated to the open cover given by $\varGamma \mathbin| V$, the $\varOmega \smallsetminus \varGamma \mathbin| \overline{U}$, and $\varGamma \smallsetminus (\varGamma \mathbin| \overline{U} \cup \varGamma \mathbin| Z)$.

Now, we know that at the cost of shrinking $V$ around $C$ a bit we can find a connection $H$ on $\varGamma \tto M$ for which $H \mathbin| V = \varPhi$; cf.~the comments following the statement of Problem \ref{prob:ext}. Relative to the choice of a Riemannian metric on $M$, we have the orthogonal direct sum decomposition
\begin{equation}
	\T M \mathbin| Z = \T Z \oplus \T^\bot Z.
\label{eqn:direct}
\end{equation}
For each $g \in \varGamma \mathbin| Z$, let $\zeta_g$ be the linear map of $\T_{sg}M$ into $\ker\T_gs$ that equals $\eta^\varPsi_g - \eta^H_g$ on $\T_{sg}Z$ and $0$ on $\T^\bot_{sg}Z$. We may view $g \mapsto \zeta_g$ as a partial cross-section of $E = L(s^*\T M,\ker\der s)$ defined over $\varGamma \mathbin| Z$ which vanishes over $\varGamma \mathbin| V \cap Z$. For any $U$ as in the previous paragraph, we may add any global cross-section of $E$ extending $g \mapsto \zeta_g$ and vanishing over $\varGamma \mathbin| U$ to $\eta^H$ so as to produce a new $H$ now satisfying both $H \mathbin| U = \varPhi$ and $H \mathbin| Z = \varPsi$.

\paragraph*{Step 2.} Our provisional construction of $H$ in the previous step involves the choice of a Riemannian metric on $M$. In order to be able later to ``correct'' $H$ so that it complies with all of our requirements, however, we cannot just pick any metric: it will be necessary that we pick one which is both invariant under $\lambda^\varPhi$ over (a possibly smaller) $V$, and invariant under $\lambda^\varPsi$ along $Z$ in the sense that $\lambda^\varPsi_g$ is an orthogonal linear transformation of $\T_{sg}Z$ into $\T_{tg}Z$ for all $g$ in $\varGamma \mathbin| Z$. We now explain how to construct such metrics.

Of course, there is some Riemannian metric on $V$ which is invariant under $\lambda^\varPhi$, say, by Lemma \ref{lem:metric}. At the cost of shrinking $V$ around $C$ a bit, we may assume this is the restriction of a global metric on $M$. Let us consider the direct sum decomposition \eqref{eqn:direct} corresponding to this metric. Since $Z$ is invariant, for every $g$ in $\varGamma \mathbin| Z$ the linear map $\lambda^H_g: \T_{sg}M \to \T_{tg}M$ carries $\T_{sg}Z$ into $\T_{tg}Z$. It need not carry $\T^\bot_{sg}Z$ into $\T^\bot_{tg}Z$. However, if we set
\begin{equation*}
	\nu^Z_g = (\T^\bot_{sg}Z \to \T_{sg}M \xto{\lambda^H_g} \T_{tg}M \to \T^\bot_{tg}Z),
\end{equation*}
we get a representation $\nu^Z$ of $\varGamma \mathbin| Z \tto Z$ on $\T^\bot Z$; the argument is the same that we gave after the proof of Lemma \ref{lem:metric} in subsection \ref{sub:proof2} (there the role of $Z$ was played by the single orbits). For all $g$ in $\varGamma \mathbin| V \cap Z$, by the orthogonality of $\lambda^\varPhi$ with respect to our metric, we have
\begin{equation}
 \lambda^\varPhi_g =%
\begin{pmatrix}
	\lambda^\varPsi_g &    0
\\	        0         & \nu^Z_g
\end{pmatrix},
\label{eqn:matrix}
\end{equation}
where the matrix on the right makes sense for every $g$ in $\varGamma \mathbin| Z$ and defines a \emph{representation} of $\varGamma \mathbin| Z \tto Z$ on $\T M \mathbin| Z$ which carries the subbundle $\T Z$ of $\T M \mathbin| Z$ into itself and, on it, agrees with $\lambda^\varPsi$. We may view the latter representation as a partial section of $L(s^*\T M,t^*\T M)$ defined over $\varGamma \mathbin| Z$ that agrees with $\lambda^\varPhi$ on $\varGamma \mathbin| V \cap Z$. Let us fix an open neighborhood $U$ of $C$ of the kind specified in Lemma \ref{lem:domain}. By the extension principle discussed in Step 1, we can find a global section $\lambda$ of $L(s^*\T M,t^*\T M)$ (a pseudo-representation of $\varGamma \tto M$ on $\T M$) which over $U$ agrees with $\lambda^\varPhi$ and over $Z$ agrees with the matrix representation \eqref{eqn:matrix}. Now $\lambda$ and $S = U \cup Z$ satisfy the hypotheses of Lemma \ref{lem:metric}, so the existence of a metric with the desired properties follows from that lemma.

\paragraph*{Step 3.} We explain now how to “correct” $H$ so as to make its effect into a representation over $Z$. In order to simplify the presentation, we shall only do this under the extra hypothesis that $Z$ is \emph{homogeneous} i.e.~consists of orbits all of the same dimension; since the homogeneous case is the only one that is relevant for the applications of Theorem \ref{prop:14A.5.2} discussed in section \ref{sec:results}, there is no reason for delving into the general case here; the latter requires little additional work, and can safely be left to the reader.

Let $L$ be the \emph{longitudinal bundle} of the \emph{regular} groupoid $\varGamma \mathbin| Z \tto Z$; cf.~subsection \ref{sub:regular}. It is a subbundle of $\T Z$ and hence $\T M \mathbin| Z$ of class $C^\infty$. As in subsection \ref{sub:proof2}, let us pick an arbitrary splitting $\xi: t^*L \to \ker\der s$ of the epimorphism of vector bundles $\der t: \ker\der s \onto t^*L$; source and target, here, refer to $\varGamma \mathbin| Z \tto Z$. For each $g$ in $\varGamma \mathbin| Z$, $\xi$ provides a linear map $\xi_g: L_{tg} \to \ker\T_gs$ satisfying $\T_gt \circ \xi_g = \id$. Let us endow $M$ with a Riemannian metric of the sort specified in the previous step. We have a corresponding orthogonal direct sum decomposition
\begin{equation}
	\T M \mathbin| Z = L \oplus L^\bot.
\label{eqn:direct*}
\end{equation}
For every $g$ in $\varGamma \mathbin| Z$, the map $\lambda^H_g$ carries $L_{sg}$ into $L_{tg}$. Thus, letting $\pr$ denote the orthogonal projection from $\T M \mathbin| Z$ onto $L$ determined by our metric, it makes sense to set
\begin{equation*}
	\zeta_g = \xi_g \circ (\lambda^H_g \circ \pr_{sg} - \pr_{tg} \circ \lambda^H_g): \T_{sg}M \longto \ker\T_gs.
\end{equation*}
We regard $g \mapsto \zeta_g$ as a partial cross-section of the vector bundle $L(s^*\T M,\ker\der s)$ with domain of definition $\varGamma \mathbin| Z$. By the invariance of the chosen metric under the effect of $\varPhi = H \mathbin| V$, the linear map $\zeta_g$ must be zero for all $g$ in $\varGamma \mathbin| V \cap Z$. By the extension principle discussed in Step 1, at the expense of shrinking $V$ we can extend $g \mapsto \zeta_g$ to a global cross-section vanishing on $\varGamma \mathbin| V$. We contend that the result of adding this global cross-section to $\eta^H$ is the horizontal lift associated with a new connection, say, $H'$ enjoying all of the properties we want. We obviously have $H' \mathbin| V = H \mathbin| V = \varPhi$. For each $g$ in $\varGamma \mathbin| Z$, the linear map $\lambda^H_g \mathbin| \T_{sg}Z = \lambda^\varPsi_g$ carries $\T_{sg}Z \cap L^\bot_{sg}$ into $\T_{tg}Z \cap L^\bot_{tg}$ (again by our choice of metric). Thus $\T_{sg}Z$ lies in $\ker\zeta_g$, and $H' \mathbin| Z = H \mathbin| Z = \varPsi$. Finally, the matrix representation of the linear map $\lambda^{H'}_g$, $g \in \varGamma \mathbin| Z$, relative to the direct sum decomposition \eqref{eqn:direct*} is readily computed to be
\begin{equation*}
 \lambda^{H'}_g =%
\begin{pmatrix}
	\lambda^\varPsi_g \mathbin| L_{sg} &   0
\\	                0                  & \nu_g
\end{pmatrix},
\end{equation*}
where $\lambda^\varPsi \mathbin| L$ is the longitudinal effect of the \emph{multiplicative} connection $\varPsi$ (hence a \emph{representation}) and $\nu_g$ is defined, as in subsection \ref{sub:proof2}/\eqref{eqn:nu/a}, by setting
\begin{equation*}
	\nu_g = (L^\bot_{sg} \to \T_{sg}M \xto{\lambda^H_g} \T_{tg}M \to L^\bot_{tg}).
\end{equation*}
Now, as in subsection \ref{sub:proof2}, $g \mapsto \nu_g$ can be shown to be a \emph{representation} of $\varGamma \mathbin| Z \tto Z$ on $L^\bot$. This proves our contention and, hence, our theorem.

\section{Proof of the fast convergence theorem}\label{sec:proof}

At this point the only task we have not yet fully carried out is our demonstration of Theorems \ref{thm:12B.14.2} and \ref{thm:12B.15.1}. In order to conclude the proof of either theorem, there remains to be shown that the pointwise limit of the sequence of mean ratios is an infinitely differentiable cross-section of the relevant vector bundle. All of the other claims about the limit, including its existence and continuity, have already been proven. The verification of the infinite differentiability will involve the notion of \emph{$C^\infty$-topology} on the space of infinitely differentiable cross-sections; for the reader's convenience, and also in order to fix the notations, we shall now review the basic definitions and facts underlying this notion that enter into our argument.

Let $P$ be a smooth manifold of dimension, say, $n$. Let $E$ be a smooth vector bundle over $P$. Let $V \subset P$ be an open set which is relatively compact in the sense that its closure $\overline{V}$ is compact. Let $\{\varphi_i\}$ be a finite collection of local coordinate charts $\varphi_i: W_i \simto \R^n$ of class $C^\infty$ for $P$ such that the open ``balls'' $B_i = \varphi_i^{-1}(\{x \in \R^n: \lvert x\rvert < 1\})$ cover $\overline{V}$. For each $i$, let $\tau_i: E \mathbin| W_i \simto W_i \times \mathbb{E}$ be a $C^\infty$ trivialization for $E$ over the domain of $\varphi_i$, where $\mathbb{E}$ is a finite-dimensional vector space, the same for all $i$, which we suppose \emph{normed,} $\lvert\blank\rvert: \mathbb{E} \to \R_{\geq 0}$ being the norm on $\mathbb{E}$. For each $C^\infty$ cross-section $\xi \in \Gamma^\infty(P;E)$, we let $\xi^{\tau_i,\varphi_i} = \pr \circ \tau_i \circ (\xi \mathbin| W_i) \circ \varphi_i^{-1}: \R^n \to \mathbb{E}$ denote the local representation of $\xi$ with respect to $\tau_i$,\ $\varphi_i$, where $\pr$ is the projection on $\mathbb{E}$. For each $n$-tuple of integers~$\geq 0$, $p = (p_1,\dotsc,p_n)$, we write $\Der^p = \Der_1^{p_1} \dotsm \Der_n^{p_n}$, where $\Der_j$ is the partial derivative with respect to the $j$-th variable in $\R^n$, so that $\Der^p\xi^{\tau_i,\varphi_i}$ is a vector-valued function $\R^n \to \mathbb{E}$; we also write $\lvert p\rvert = p_1 + \dotsb + p_n$. Then, for each integer $k \geq 0$, the expression
\begin{equation}
	\lVert\xi\rVert_{C^k\overline{V};\{\tau_i,\varphi_i\}} = \max_i{}\max_{\substack{p\in\N^n\\ \lvert p\rvert\leq k}}{}\sup_{z\in B_i\cap V}{}\bigl\lvert\Der^p\xi^{\tau_i,\varphi_i}\bigl(\varphi_i(z)\bigr)\bigr\rvert
\label{eqn:standard}
\end{equation}
defines what we shall call a \emph{standard\/ $C^k\overline{V}$~seminorm} $\lVert\blank\rVert_{C^k\overline{V};\{\tau_i,\varphi_i\}}$ on the vector space $\Gamma^\infty(P;E)$. The \emph{$C^\infty$-topology} on $\Gamma^\infty(P;E)$ is the locally convex topology generated by all these seminorms as we let $k$ vary over all integers~$\geq 0$ and $V$ over all relatively compact open subsets of $P$. A sequence in $\Gamma^\infty(P;E)$ is Cauchy for the $C^\infty$-topology if, for every $V$ from among those belonging to a predefined open cover by relatively compact subsets of $P$, and for every integer $k \geq 0$, it is Cauchy in the usual sense relative to every standard $C^k\overline{V}$~seminorm. Each Cauchy sequence in $\Gamma^\infty(P;E)$ is convergent for the $C^\infty$-topology---hence, a fortiori, pointwise---towards a unique limiting cross-section of class $C^\infty$. The locally convex space $\Gamma^\infty(P;E)$ is thus \emph{complete,} as well as obviously \emph{metrizable:} it is a \emph{Fré\-chet space.}

It is easy to see that any two standard  $C^k\overline{V}$~seminorms, for the same $k$ and $V$ but for two different choices of $\{\tau_i,\varphi_i\}$, are \emph{equivalent.} This fundamental remark justifies the introduction of the following notational device, which will spare us the nuisance of keeping track of irrelevant scaling factors throughout. Let $\mathcal{S}$ be an arbitrary set. We define a binary relation $\preceq$ on the space of all nonnegative real-valued functions $a$,~$a': \mathcal{S} \to \R_{\geq 0}$ by declaring $a \preceq a'$ to mean: there exists some positive constant $C$ such that $a(s) \leq Ca'(s)$ for all $s$ in $S$. This binary relation is reflexive and transitive; it thus gives rise to an equivalence relation on the set of all $a$, as well as descending to a partial order, which we still denote by $\preceq$, on the set of all equivalence classes. Note that $a \preceq a'$ implies both $a + b \preceq a' + b$ and $ab \preceq a'b$ for all $b: \mathcal{S} \to \R_{\geq 0}$. Moreover, if $f: \mathcal{S}' \to \mathcal{S}$ is any set-theoretic map, then $c \preceq c'$ implies $c \circ f \preceq c' \circ f$. Hence the operations of sum, product, and pullback make sense for classes of functions, and behave as expected. Back to our standard seminorms, we shall write $\lVert\blank\rVert_{C^k\overline{V}}$ for the class of any standard $C^k\overline{V}$~seminorm $\lVert\blank\rVert_{C^k\overline{V};\{\tau_i,\varphi_i\}}$ regarded as a nonnegative real-valued function on $\mathcal{S} = \Gamma^\infty(P;E)$.

We shall need the following elementary properties of our function classes $\lVert\blank\rVert_{C^k\overline{V}}$. First of all, for any morphism $\omega: E \to F$ of smooth vector bundles over $P$, we have
\begin{gather}
	\lVert\omega \circ \xi\rVert_{C^k\overline{V}} \preceq \lVert\xi\rVert_{C^k\overline{V}},
\label{lem:12B.13.9}
\end{gather}
where the variable $\xi$ ranges over $\mathcal{S} = \Gamma^\infty(P;E)$. Second, for all smooth vector bundles $E_1$,\ $E_2$,\ and $F$ over $P$ and all $C^\infty$ bilinear forms $\omega: E_1 \times E_2 \to F$, we have%
\begin{subequations}
\begin{gather}
	\lVert\omega \circ (\xi_1,\xi_2)\rVert_{C^k\overline{V}} \preceq \lVert\xi_1\rVert_{C^k\overline{V}}\lVert\xi_2\rVert_{C^k\overline{V}},
\label{eqn:12B.13.5a}\\
	\lVert\omega \circ (\xi_1,\xi_2)\rVert_{C^{k+1}\overline{V}} \preceq \lVert\xi_1\rVert_{C^k\overline{V}}\lVert\xi_2\rVert_{C^{k+1}\overline{V}} + \lVert\xi_1\rVert_{C^{k+1}\overline{V}}\lVert\xi_2\rVert_{C^k\overline{V}},
\label{eqn:12B.13.5b}
\end{gather}
\end{subequations}
where the variable $\xi_i$ ($i = 1$,~$2$) ranges over $\Gamma^\infty(P;E_i)$; the second inequality implies the first for $k \geq 1$ and essentially arises from the Leibniz rule for the derivative of a product. Third, given any smooth mapping $f: P' \to P$ and any smooth vector bundle $E$ over $P$, if we write $f^*\xi$ for the pullback of a $C^\infty$ cross-section $\xi$ of $E$ (i.e.~the unique cross-section of the pullback vector bundle $f^*E$ satisfying $\pr \circ f^*\xi = \xi \circ f$, where $\pr$ denotes the projection $f^*E = P' \times_P E \to E$) then for all relatively compact open subsets $V'$ of $P'$ such that $f(V') \subset V$
\begin{equation}
	\lVert f^*\xi\rVert_{C^k\overline{V'}} \preceq \lVert\xi\rVert_{C^k\overline{V}}.
\label{eqn:12B.13.15}
\end{equation}
Finally, given any two smooth vector bundles $E$ and $F$ over $P$, if for $\mathcal{S}$ we take the subset of $\Gamma^\infty\bigl(P;L(E,F)\bigr)$ consisting of all vector bundle isomorphisms $\lambda: E \simto F$, then
\begin{equation}
	\lVert\lambda^{-1}\rVert_{C^{k+1}\overline{V}} \preceq \lVert\lambda^{-1}\rVert_{C^k\overline{V}}^2\lVert\lambda\rVert_{C^{k+1}\overline{V}}.
\label{eqn:12B.13.9}
\end{equation}

Besides the above elementary inequalities, we shall need yet another inequality of the same sort concerning Haar integrals depending on parameters. So, let the situation and the notations be as in our discussion of these integrals at the beginning of section \ref{sec:averaging}. For technical reasons we shall suppose that our Haar system $\mu$ is one of those arising in the standard way (i.e.~as specified in loc.~cit.)\ out of a left invariant metric of class $C^\infty$. Let $c$ be the normalizing function for $\mu$. It will be convenient to write
\begin{equation*}
	\der\mu_c: \Gamma^\infty(P \ftimes{f}{t} \varGamma;\pr^*E) \longto \Gamma^\infty(P;E), \quad
	\vartheta \mapsto \langle\vartheta,\der\mu_c\rangle
\end{equation*}
for the Haar integration functional with parameters in $P$ and coefficients in $E$ defined by \eqref{eqn:12B.10.6a}; it will always be clear from the context what the parameters and the coefficients are, so there is no need to include them in the notation. Let $U$ be any relatively compact open subset of $M$ such that $\supp c \cap \varGamma U \subset U$. Let $V$ be a similar subset of $P$ such that $f(V) \subset U$. Because of properness, $\varOmega = s^{-1}(U) \cap t^{-1}(U)$ is a relatively compact open subset of $\varGamma$. One can show that the following inequality of function classes in the variable $\vartheta$ holds for every integer $k \geq 0$.
\begin{equation}
	\lVert\langle\vartheta,\der\mu_c\rangle\rVert_{C^k\overline{V}} \preceq \lVert\vartheta\rVert_{C^k\overline{V\ftimes{f}{t}\varOmega}}
\label{eqn:12B.20.1}
\end{equation}
Modulo technicalities, its proof is a simple differentiation under the integral sign. In practice, we shall always either have $V = \varOmega$ and $V \ftimes{f}{t} \varOmega = \varOmega \ftimes{s}{t} \varOmega$, or else have $V = \varOmega \ftimes{s}{t} \varOmega$ and $V \ftimes{f}{t} \varOmega = \varOmega \ftimes{s}{t} \varOmega \ftimes{s}{t} \varOmega$. As a matter of notation, we shall find it convenient to let $\varOmega_2$ and $\varOmega_3$ denote the latter two sets and to write $\varGamma_2 = \varGamma \ftimes{s}{t} \varGamma$, $\varGamma_3 = \varGamma \ftimes{s}{t} \varGamma \ftimes{s}{t} \varGamma$. We shall also find it convenient to let $d_0$,~$d_2: \varGamma_2 \to \varGamma$ denote the two projections and $d_1$,~$d_3: \varGamma_3 \to \varGamma_2$ the two maps sending each composable triplet $(g,h,k)$ to $(g,hk)$, respectively, $(h,k)$; of course this is just standard notation for the face maps of the nerve of a groupoid.

After these preliminaries, we are in a position to start the actual proof of our fast convergence theorem. We shall deal with Theorem \ref{thm:12B.14.2} first.

Let $E$ be a smooth vector bundle over the base $M$ of our proper Lie groupoid $\varGamma \tto M$. For every pseudo-representation $\lambda: s^*E \to t^*E$ of $\varGamma \tto M$ on $E$ viewed as a $C^\infty$ cross-section of the vector bundle $L(s^*E,t^*E) \to \varGamma$, we shall by analogy with \eqref{eqn:basic} write \[%
	R^\lambda \in \Gamma^\infty\bigl(\varGamma_2;L(s^*E,t^*E)\bigr), \quad
	R^\lambda(g,h) = \lambda_{gh} - \lambda_g\lambda_h.
\] Let $\mathcal{S}$ be the subset of $\Gamma^\infty\bigl(\varGamma;L(s^*E,t^*E)\bigr)$ formed by all invertible pseudo-representations $\lambda: s^*E \simto t^*E$. In the notations introduced after \eqref{eqn:12B.20.1}, we have the following inequalities of function classes in the variable $\lambda \in \mathcal{S}$:
\begin{alignat*}{2}
 \lVert\hat{\lambda} - \lambda\rVert_{C^{k+1}\overline{\varOmega}} &
	= \lVert\langle R^\lambda \circ d_2^*\lambda^{-1},\der\mu_c\rangle\rVert_{C^{k+1}\overline{\varOmega}} &\quad &
		\text{by \eqref{eqn:12B.12.5a}}\\ &
	\preceq \lVert R^\lambda \circ d_2^*\lambda^{-1}\rVert_{C^{k+1}\overline{\varOmega_2}} &\quad &
		\text{by \eqref{eqn:12B.20.1}}\\ &
	\preceq \lVert R^\lambda\rVert_{C^k\overline{\varOmega_2}}\lVert d_2^*\lambda^{-1}\rVert_{C^{k+1}\overline{\varOmega_2}}
	     + \lVert R^\lambda\rVert_{C^{k+1}\overline{\varOmega_2}}\lVert d_2^*\lambda^{-1}\rVert_{C^k\overline{\varOmega_2}} &\quad &
		\text{by \eqref{eqn:12B.13.5b}}\\ &
	\preceq \lVert R^\lambda\rVert_{C^k\overline{\varOmega_2}}\lVert\lambda^{-1}\rVert_{C^{k+1}\overline{\varOmega}}
	     + \lVert R^\lambda\rVert_{C^{k+1}\overline{\varOmega_2}}\lVert\lambda^{-1}\rVert_{C^k\overline{\varOmega}} &\quad &
		\text{by \eqref{eqn:12B.13.15}}
\end{alignat*}
whence on account of \eqref{eqn:12B.13.9}%
\begin{subequations}
\begin{equation}
 \lVert\hat{\lambda} - \lambda\rVert_{C^{k+1}\overline{\varOmega}}
	\preceq \{\lVert\lambda^{-1}\rVert_{C^k\overline{\varOmega}}\lVert\lambda\rVert_{C^{k+1}\overline{\varOmega}}\lVert R^\lambda\rVert_{C^k\overline{\varOmega_2}}
	          + \lVert R^\lambda\rVert_{C^{k+1}\overline{\varOmega_2}}\}\lVert\lambda^{-1}\rVert_{C^k\overline{\varOmega}}.
\label{eqn:12B.14.10}
\end{equation}
On the other hand
\begin{alignat*}{2}
 \lVert R^{\hat{\lambda}}\rVert_{C^{k+1}\overline{\varOmega_2}} &
	\preceq \lVert\langle d_1^*(R^\lambda \circ d_2^*\lambda^{-1}) \circ d_3^*(R^\lambda \circ d_2^*\lambda^{-1}),\der\mu_c\rangle\rVert_{C^{k+1}\overline{\varOmega_2}} \\* &\justify
	        + \lVert d_0^*\langle R^\lambda \circ d_2^*\lambda^{-1},\der\mu_c\rangle \circ d_2^*\langle R^\lambda \circ d_2^*\lambda^{-1},\der\mu_c\rangle\rVert_{C^{k+1}\overline{\varOmega_2}} &\quad &
		\text{by \eqref{eqn:12B.12.5b}} \\ &\hskip-4em
	\preceq \lVert d_1^*(R^\lambda \circ d_2^*\lambda^{-1}) \circ d_3^*(R^\lambda \circ d_2^*\lambda^{-1})\rVert_{C^{k+1}\overline{\varOmega_3}} &\quad &
		\text{by \eqref{eqn:12B.20.1}} \\* &\hskip-4em\justify
	        + \lVert d_0^*\langle\mathellipsis\rangle\rVert_{C^k\overline{\varOmega_2}}\lVert d_2^*\langle\mathellipsis\rangle\rVert_{C^{k+1}\overline{\varOmega_2}} + \lVert d_0^*\langle\mathellipsis\rangle\rVert_{C^{k+1}\overline{\varOmega_2}}\lVert d_2^*\langle\mathellipsis\rangle\rVert_{C^k\overline{\varOmega_2}} &\quad &
		\text{by \eqref{eqn:12B.13.5b}} \\ &\hskip-4em
	\preceq \lVert d_1^*(\mathellipsis)\rVert_{C^k\overline{\varOmega_3}}\lVert d_3^*(\mathellipsis)\rVert_{C^{k+1}\overline{\varOmega_3}} + \lVert d_1^*(\mathellipsis)\rVert_{C^{k+1}\overline{\varOmega_3}}\lVert d_3^*(\mathellipsis)\rVert_{C^k\overline{\varOmega_3}} &\quad &
		\text{by \eqref{eqn:12B.13.5b}} \\* &\hskip-4em\justify
	     + \lVert\langle R^\lambda \circ d_2^*\lambda^{-1},\der\mu_c\rangle\rVert_{C^k\overline{\varOmega}}\lVert\langle R^\lambda \circ d_2^*\lambda^{-1},\der\mu_c\rangle\rVert_{C^{k+1}\overline{\varOmega}} &\quad &
		\text{by \eqref{eqn:12B.13.15}} \\ &\hskip-4em
	\preceq \lVert R^\lambda \circ d_2^*\lambda^{-1}\rVert_{C^k\overline{\varOmega_2}}\lVert R^\lambda \circ d_2^*\lambda^{-1}\rVert_{C^{k+1}\overline{\varOmega_2}}
\end{alignat*}
whence again on account of \eqref{eqn:12B.13.9} we find the following estimate in which the expression within braces is exactly the same as in the above inequality \eqref{eqn:12B.14.10}.
\begin{equation}
 \lVert R^{\hat{\lambda}}\rVert_{C^{k+1}\overline{\varOmega_2}}
	\preceq \{\lVert\lambda^{-1}\rVert_{C^k\overline{\varOmega}}\lVert\lambda\rVert_{C^{k+1}\overline{\varOmega}}\lVert R^\lambda\rVert_{C^k\overline{\varOmega_2}}
	          + \mathellipsis\}(\lVert\lambda^{-1}\rVert_{C^k\overline{\varOmega}})^2\lVert R^\lambda\rVert_{C^k\overline{\varOmega_2}}
\label{eqn:12B.14.15}
\end{equation}
\end{subequations}

Let us be given an arbitrary \emph{near representation} $\lambda$ now, which we shall henceforth regard as fixed, as in the statement of Theorem \ref{thm:12B.14.2}. On the basis of our general estimates \eqref{eqn:12B.14.10} and \eqref{eqn:12B.14.15}, we are going to show that for every integer $k \geq 0$ and for some number $0 \leq \epsilon < 1$ independent of $k$ (say, $\epsilon = 2/3$) the four statements S1–S4 hereafter hold:
\begin{description}
 \item[S1.]\em $\lVert\hat{\lambda}^i\rVert_{C^k\overline{\varOmega}} \preceq 1$, in other words, $i \mapsto \lVert\hat{\lambda}^i\rVert_{C^k\overline{\varOmega}}$ is bounded as a function class on\/ $\N$.\em
 \item[S2.]\em $\lVert(\hat{\lambda}^i)^{-1}\rVert_{C^k\overline{\varOmega}} \preceq 1$, in other words, $i \mapsto \lVert(\hat{\lambda}^i)^{-1}\rVert_{C^k\overline{\varOmega}}$ is bounded as a function class on\/ $\N$.\em
 \item[S3.]\em The inequality\/ $\lVert R^{\hat{\lambda}^{i+1}}\rVert_{C^k\overline{\varOmega_2}} \preceq (\lVert R^{\hat{\lambda}^i}\rVert_{C^k\overline{\varOmega_2}})^2$ of function classes in\/ $i \in \N$ holds.\em
 \item[S4.]\em The inequality\/ $\lVert R^{\hat{\lambda}^i}\rVert_{C^k\overline{\varOmega_2}} \preceq \epsilon^{2^i}$ of function classes in\/ $i \in \N$ holds.\em
\end{description}
The proof will proceed by induction on $k$. Once established, the validity of these four statements will enable us to conclude that for every order of derivation $k$ the sequence $\hat{\lambda}^i$ is Cauchy within $\Gamma^\infty\bigl(\varGamma;L(s^*E,t^*E)\bigr)$ relative to any standard $C^k\overline{\varOmega}$~seminorm, thus proving our theorem: indeed, on the basis of the first two steps in the computation leading to \eqref{eqn:12B.14.10}, and on the basis of \eqref{eqn:12B.13.5a}, S2, and S4, we will then have
\begin{align*}
 \lVert\hat{\lambda}^{i+1} - \hat{\lambda}^i\rVert_{C^k\overline{\varOmega}}
	\preceq \lVert R^{\hat{\lambda}^i}\rVert_{C^k\overline{\varOmega_2}}\lVert(\hat{\lambda}^i)^{-1}\rVert_{C^k\overline{\varOmega}}
	\preceq \epsilon^{2^i},
\end{align*}
so that our sequence will be Cauchy relative to the $C^\infty$-topology, and its pointwise limit $\hat{\lambda}^\infty$ will be $C^\infty$-differentiable, because the open sets $\varOmega$ of the form $s^{-1}(U) \cap t^{-1}(U)$ for $U$ as specified in the text preceding \eqref{eqn:12B.20.1} cover the whole $\varGamma$.

\paragraph*{Base case.} We begin by establishing the validity of our statements S1 to S4 for $k = 0$; in the course of their proof, we shall assign a definite numerical value to the parameter $\epsilon$.

If we endow $E$ with a vector bundle metric and set $b_U(\lambda) = \sup_{g\in \varOmega}{}\lVert\lambda_g\rVert$, we get a seminorm $b_U$ on $\Gamma^\infty\bigl(\varGamma;L(s^*E,t^*E)\bigr)$ which is of course a standard $C^0\overline{\varOmega}$~seminorm. The function $i \mapsto b_U(\hat{\lambda}^i)$ is then a representative for the function class $i \mapsto \lVert\hat{\lambda}^i\rVert_{C^0\overline{\varOmega}}$; ditto for $i \mapsto b_U\bigl((\hat{\lambda}^i)^{-1}\bigr)$ in relation to $i \mapsto \lVert(\hat{\lambda}^i)^{-1}\rVert_{C^0\overline{\varOmega}}$. Similarly, if we define $r_U(\lambda)$ by restricting the supremums in \eqref{eqn:12B.12.8b} to $x \in U$ and to $(g_1,g_2) \in \varOmega_2$, then $i \mapsto r_U(\hat{\lambda}^i)$ is a representative for the function class $i \mapsto \lVert R^{\hat{\lambda}^i}\rVert_{C^0\overline{\varOmega_2}}$.

Now, for $k = 0$ statement S3 follows from the obvious remark that when $(g_1,g_2) \in \varOmega_2$ the inequality \eqref{eqn:12B.12.10b} is still satisfied after replacing $b(\lambda)$, $r(\lambda)$ respectively with $b_U(\lambda)$, $r_U(\lambda)$. As to the other statements, we note that the quantities $b(\lambda)$, $r(\lambda)$ given by \eqref{eqn:12B.12.8} satisfy $b(\lambda) \geq b_U(\lambda)$, $r(\lambda) \geq r_U(\lambda)$; S1 and S4 are then immediate consequences of our next lemma, which, on account of \eqref{eqn:12B.12.9a}, also implies S2.

\begin{lem}\label{lem:12B.12.8} Let\/ $\lambda$ be a near representation (\textup{Definition \ref{defn:12B.14.1}}). Let the quantities\/ $b_i = b(\hat{\lambda}^i)$ and\/ $r_i = r(\hat{\lambda}^i)$ be defined as in\/ \eqref{eqn:12B.12.8} with respect to a vector bundle metric such that\/ \eqref{eqn:12B.14.1} is satisfied. Put\/ $\epsilon = 6b_0^2r_0 \leq 2/3$. Then, the two inequalities below are valid for all\/ $i$.%
\begin{subequations}
\label{eqn:12B.12.12}
\begin{gather}
	r_i \leq \frac{1}{6b_0^2}\epsilon^{2^i}
\label{eqn:12B.12.12a}\\
	\frac{b_i}{1 - r_i} \leq \sqrt{3}b_0
\label{eqn:12B.12.12b}
\end{gather}
\end{subequations} \end{lem}

\begin{proof} It is clear that both inequalities hold for $i = 0$, the first one by definition of $\epsilon$, the second one because $1/(1 - r_0) \leq 1/(1 - \frac{1}{4}) = \frac{4}{3} \leq \sqrt{3}$. It will then suffice to show that for every $n \geq 0$ the second inequality must be satisfied for $i = n$ whenever the first is satisfied for all $i$ between zero and $n$: indeed if we can prove this and if we know that \eqref{eqn:12B.12.12a} is satisfied for all $i = 0$,~$\dotsc,$~$n$, then by the unitality of $\hat{\lambda}^{n+1}$ and by \eqref{eqn:12B.12.10b}
\begin{equation*}
 r_{n+1}
	\leq 2\biggl(\frac{b_n}{1 - r_n}\biggr)^2r_n^2
	\leq 2 \cdot 3b_0^2 \cdot \biggl(\frac{1}{6b_0^2}\biggr)^2(\epsilon^{2^n})^2
	= \frac{1}{6b_0^2}\epsilon^{2^{n+1}}.
\end{equation*}

In virtue of the remarks following Definition \ref{defn:12B.14.1} we know $r_i < 1$ for all $i$. In consequence of \eqref{eqn:12B.12.10a} we then have $b_{i+1} \leq b_i/(1 - r_i)$ for all $i$ and, therefore, upon combining these inequalities recursively as $i$ runs from zero to $n - 1$,
\begin{equation}
	b_n/(1 - r_n) \leq b_0/[(1 - r_0) \dotsm (1 - r_n)].
\label{eqn:12B.12.13}
\end{equation}
We proceed to estimate the quantity
\begin{equation*}
\textstyle%
 1\big/\prod_{i=0}^n (1 - r_i)
	= 1\big/\bigl[\exp\log{}\bigl(\prod_{i=0}^n (1 - r_i)\bigr)\bigr]
	= \exp\bigl(-\sum_{i=0}^n \log(1 - r_i)\bigr).
\end{equation*}
For any real number $r$ such that $\lvert r\rvert < 1$,
\begin{align*}
 -\lvert r\rvert + \lvert\log(1 + r)\rvert &
	\leq \lvert r - \log(1 + r)\rvert \\ &
	= \biggl\lvert\frac{r^2}{2} - \frac{r^3}{3} + \frac{r^4}{4} - \dotsb\biggr\rvert
	\leq \frac{\lvert r\rvert^2}{2} + \frac{\lvert r\rvert^3}{2} + \frac{\lvert r\rvert^4}{2} + \dotsb
	= \frac{\lvert r\rvert^2}{2}\frac{1}{1 - \lvert r\rvert};
\end{align*}
this quantity is smaller than $\lvert r\rvert^2$ whenever $\lvert r\rvert$ is~$\leq 1/2$, hence upon substituting $r$ with $-r$ we conclude that
\begin{equation*}
	0 \leq r \leq 1/2
\text{\quad implies}\mathopen{\quad}
	-\log(1 - r) \leq r + r^2.
\end{equation*}
Then, noting that the inequalities $r_i \leq r_0 \leq \frac{1}{4} < \frac{1}{2}$ and $2^i \geq 2i$ hold for all $i$,
\begin{align*}
 \exp\left(\textstyle\sum\limits_{i=0}^n {-\log(1 - r_i)}\right) &
	\leq \exp\left(\textstyle\sum\limits_{i=0}^n r_i + \sum\limits_{i=0}^n r_i^2\right) \\ &
	\leq \exp\left(\frac{1}{6}\textstyle\sum\limits_{i=0}^n \epsilon^{2i}\right)\exp\left(\frac{1}{6^2}\textstyle\sum\limits_{i=0}^n \epsilon^{4i}\right) \\ &
	\leq \exp\left(\frac{1}{6}\frac{1}{1 - \epsilon^2}\right)\exp\left(\frac{1}{6^2}\frac{1}{1 - \epsilon^4}\right) \\ &
	\leq \exp\left(\frac{1}{6}\frac{9}{5} + \frac{1}{6^2}\frac{81}{65}\right)
	\leq \exp\left(1/2\right) \leq \sqrt{3}.
\end{align*}
Now, if we combine this with \eqref{eqn:12B.12.13} we obtain the desired upper bound: $b_n/(1 - r_n) \leq \sqrt{3}b_0$. \end{proof}

\paragraph*{Inductive step.} We proceed to demonstrate that the validity of our statements S1–S4 for a given value of $k \geq 0$ implies their validity for the next higher value $k + 1$ as well.

In view of our inductive hypothesis S2, the inequalities of function classes \eqref{eqn:12B.14.10} and \eqref{eqn:12B.14.15} lead to the estimates below.
\begin{subequations}
\label{eqn:12B.14.20}
\begin{gather}
\begin{aligned}[b]
 \lVert\hat{\lambda}^{i+1}\rVert_{C^{k+1}\overline{\varOmega}} &
	\preceq \lVert\hat{\lambda}^i\rVert_{C^{k+1}\overline{\varOmega}} + \lVert\hat{\lambda}^{i+1} - \hat{\lambda}^i\rVert_{C^{k+1}\overline{\varOmega}} \\ &
	\preceq \lVert\hat{\lambda}^i\rVert_{C^{k+1}\overline{\varOmega}} + \{\lVert\hat{\lambda}^i\rVert_{C^{k+1}\overline{\varOmega}}\lVert R^{\hat{\lambda}^i}\rVert_{C^k\overline{\varOmega_2}} + \lVert R^{\hat{\lambda}^i}\rVert_{C^{k+1}\overline{\varOmega_2}}\}
\end{aligned}
\label{eqn:12B.14.20a}\\
 \lVert R^{\hat{\lambda}^{i+1}}\rVert_{C^{k+1}\overline{\varOmega_2}}
	\preceq \{\lVert\hat{\lambda}^i\rVert_{C^{k+1}\overline{\varOmega}}\lVert R^{\hat{\lambda}^i}\rVert_{C^k\overline{\varOmega_2}}
	           + \lVert R^{\hat{\lambda}^i}\rVert_{C^{k+1}\overline{\varOmega_2}}\}\lVert R^{\hat{\lambda}^i}\rVert_{C^k\overline{\varOmega_2}}
\label{eqn:12B.14.20b}
\end{gather}
\end{subequations}
Let $i \mapsto a_i$ be the function class on $\N$ given by the expression within brackets. Then
\begin{alignat*}{2}
 a_{i+1} &
	\preceq (\lVert\hat{\lambda}^i\rVert_{C^{k+1}\overline{\varOmega}} + a_i)\lVert R^{\hat{\lambda}^{i+1}}\rVert_{C^k\overline{\varOmega_2}}
	        + a_i\lVert R^{\hat{\lambda}^i}\rVert_{C^k\overline{\varOmega_2}} &\quad &
		\text{by \eqref{eqn:12B.14.20a} and \eqref{eqn:12B.14.20b}} \\ &
	\preceq (\lVert\hat{\lambda}^i\rVert_{C^{k+1}\overline{\varOmega}}^{\vphantom{2}} + a_i)\lVert R^{\hat{\lambda}^i}\rVert_{C^k\overline{\varOmega_2}}^2
	        + a_i\lVert R^{\hat{\lambda}^i}\rVert_{C^k\overline{\varOmega_2}}^{\vphantom{2}} &\quad &
		\text{by our inductive hypothesis S3} \\ &
	\preceq \lVert\hat{\lambda}^i\rVert_{C^{k+1}\overline{\varOmega}}^{\vphantom{2}}\lVert R^{\hat{\lambda}^i}\rVert_{C^k\overline{\varOmega_2}}^2
	        + a_i\lVert R^{\hat{\lambda}^i}\rVert_{C^k\overline{\varOmega_2}}^{\vphantom{2}}(\epsilon^{2^i} + 1) &\quad &
		\text{by our inductive hypothesis S4} \\ &
	\preceq a_i\lVert R^{\hat{\lambda}^i}\rVert_{C^k\overline{\varOmega_2}}
	        + a_i\lVert R^{\hat{\lambda}^i}\rVert_{C^k\overline{\varOmega_2}} &\quad &
		\text{a fortiori.}
\end{alignat*}
It follows from S4 that $a_{i+1} \preceq a_i\epsilon^{2^i}$ as function classes on $\N$. Since $C\epsilon^{2^i}$ tends to zero for each constant $C > 0$ as $i$ goes to infinity, we see that $i \mapsto a_i$ is eventually decreasing, hence bounded, and thus that $a_{i+1} \preceq \epsilon^{2^i}$, equivalently, $a_i \preceq \epsilon^{2^{i-1}}$ as function classes on $\N$.

We notice next that if $f: \N \to \R_{\geq 0}$ is any actual function representing a function class $F$, the class of the “partial sums” function $i \mapsto \sum_{j=0}^i f(j)$ only depends on $F$, not on the choice of representative $f$. We write $\sum F$ for this “partial sums” function class. Let us now take $F$ to be the function class $i \mapsto \lVert\hat{\lambda}^{i+1} - \hat{\lambda}^i\rVert_{C^{k+1}\overline{\varOmega}}$. Let $F'$ be the function class $i \mapsto \lVert\hat{\lambda}^{i+1} - \lambda\rVert_{C^{k+1}\overline{\varOmega}}$. We obviously have $F' \preceq \sum F \preceq \sum \epsilon^{2^{i-1}} \preceq 1$ because by \eqref{eqn:12B.14.10} and S2 $F_i \preceq a_i$. The boundedness of $i \mapsto \lVert\hat{\lambda}^{i+1}\rVert_{C^{k+1}\overline{\varOmega}}$ then follows. Our first inductive claim S1 is proven.

Our second inductive claim S2 follows from the already proven S1 on account of \eqref{eqn:12B.13.9} and of our inductive hypothesis S2.

As to our next inductive claim S3, by \eqref{eqn:12B.14.20b} and the already proven S1, we have
\begin{align*}
 \lVert R^{\hat{\lambda}^{i+1}}\rVert_{C^{k+1}\overline{\varOmega_2}} &
	\preceq \{\lVert R^{\hat{\lambda}^i}\rVert_{C^k\overline{\varOmega_2}}
	           + \lVert R^{\hat{\lambda}^i}\rVert_{C^{k+1}\overline{\varOmega_2}}\}\lVert R^{\hat{\lambda}^i}\rVert_{C^k\overline{\varOmega_2}} \\ &
	\preceq \{\lVert R^{\hat{\lambda}^i}\rVert_{C^{k+1}\overline{\varOmega_2}}
	           + \lVert R^{\hat{\lambda}^i}\rVert_{C^{k+1}\overline{\varOmega_2}}\}\lVert R^{\hat{\lambda}^i}\rVert_{C^{k+1}\overline{\varOmega_2}}
	= \lVert R^{\hat{\lambda}^i}\rVert_{C^{k+1}\overline{\varOmega_2}}^2.
\end{align*}

Finally, from \eqref{eqn:12B.14.20b} and our inductive hypothesis S4, we deduce that since $a_i \preceq \epsilon^{2^{i-1}} \preceq 1$
\begin{align*}
 \lVert R^{\hat{\lambda}^{i+1}}\rVert_{C^{k+1}\overline{\varOmega_2}} &
	\preceq a_i\lVert R^{\hat{\lambda}^i}\rVert_{C^k\overline{\varOmega_2}}
	\preceq \lVert R^{\hat{\lambda}^i}\rVert_{C^k\overline{\varOmega_2}}
	\preceq \epsilon^{2^i}.
\end{align*}

\subsubsection*{The case of connections}

The proof of Theorem \ref{thm:12B.15.1}, which we now turn our attention to, is essentially a byproduct of the proof of Theorem \ref{thm:12B.14.2} we have just concluded. Our notations, in what follows, will be as in the text preceding the statement of Theorem \ref{thm:12B.15.1}.

Let $\mathcal{S}$ be the set of all nondegenerate connections on the proper Lie groupoid $\varGamma \tto M$. For each $H$ in $\mathcal{S}$, we regard $R^H$ as an element of $\Gamma^\infty\bigl(\varGamma_2;L(s^*\T M,t^*\Lie\varGamma)\bigr)$. From the identity \eqref{eqn:12B.15.5b}, by making repeated use of the elementary estimates \eqref{lem:12B.13.9}–\eqref{eqn:12B.20.1}, we obtain the following inequality of function classes in the variable $H \in \mathcal{S}$, where for the sake of readability we write $\lambda$ instead of $\lambda^H$:
\begin{align*}
 \lVert R^{\hat{H}}\rVert_{C^k\overline{\varOmega_2}} &
	\preceq \lVert\langle d_1^*(R^H \circ d_2^*\lambda^{-1}) \circ d_3^*(R^\lambda \circ d_2^*\lambda^{-1}),\der\mu_c\rangle\rVert_{C^k\overline{\varOmega_2}} \\* &\justify
	        + \lVert d_0^*\langle R^H \circ d_2^*\lambda^{-1},\der\mu_c\rangle \circ d_2^*\langle R^\lambda \circ d_2^*\lambda^{-1},\der\mu_c\rangle\rVert_{C^k\overline{\varOmega_2}} \\ &
	\preceq \lVert d_1^*(R^H \circ d_2^*\lambda^{-1})\rVert_{C^k\overline{\varOmega_3}}\lVert d_3^*(R^\lambda \circ d_2^*\lambda^{-1})\rVert_{C^k\overline{\varOmega_3}} \\* &\justify
	        + \lVert\langle R^H \circ d_2^*\lambda^{-1},\der\mu_c\rangle\rVert_{C^k\overline{\varOmega}}\lVert\langle R^\lambda \circ d_2^*\lambda^{-1},\der\mu_c\rangle\rVert_{C^k\overline{\varOmega}} \\ &
	\preceq \lVert R^H \circ d_2^*\lambda^{-1}\rVert_{C^k\overline{\varOmega_2}}\lVert R^\lambda \circ d_2^*\lambda^{-1}\rVert_{C^k\overline{\varOmega_2}} \\ &
	\preceq \lVert R^H\rVert_{C^k\overline{\varOmega_2}}\lVert\lambda^{-1}\rVert_{C^k\overline{\varOmega}}\lVert R^\lambda\rVert_{C^k\overline{\varOmega_2}}\lVert\lambda^{-1}\rVert_{C^k\overline{\varOmega}}.
\end{align*}

Suppose that for a given $H$, which we shall keep fixed, $\lambda = \lambda^H$ is a near representation. Our last inequality in conjunction with S2 and S4 above entails that $\lVert R^{\hat{H}^{i+1}}\rVert_{C^k\overline{\varOmega_2}} \preceq \epsilon^{2^i}\lVert R^{\hat{H}^i}\rVert_{C^k\overline{\varOmega_2}}$ as function classes in $i \in \N$. By the argument we used earlier in our “inductive step” to derive the estimate $a_i \preceq \epsilon^{2^{i-1}}$, this entails in turn that $\lVert R^{\hat{H}^i}\rVert_{C^k\overline{\varOmega_2}} \preceq \epsilon^{2^{i-1}}$ and, therefore,
\begin{alignat*}{2}
 \lVert\eta^{\hat{H}^{i+1}} - \eta^{\hat{H}^i}\rVert_{C^k\overline{\varOmega}} &
	= \lVert\omega \circ (\eta^{\hat{H}^{i+1}} - \eta^{\hat{H}^i})\rVert_{C^k\overline{\varOmega}} &\quad &
		\text{by \eqref{lem:12B.13.9}} \\ &
	= \lVert\langle R^{\hat{H}^i} \circ d_2^*(\hat{\lambda}^i)^{-1},\der\mu_c\rangle\rVert_{C^k\overline{\varOmega}} &\quad &
		\text{by \eqref{eqn:12B.15.5a*}} \\ &
	\preceq \lVert R^{\hat{H}^i}\rVert_{C^k\overline{\varOmega_2}}\lVert(\hat{\lambda}^i)^{-1}\rVert_{C^k\overline{\varOmega}} &\quad &
		\text{by \eqref{eqn:12B.20.1} and \eqref{eqn:12B.13.5a}} \\ &
	\preceq \epsilon^{2^{i-1}} &\quad &
		\text{again by S2.}
\end{alignat*}
It follows immediately that the sequence $\eta^{\hat{H}^i}$ is Cauchy for the $C^\infty$-topology and, hence, that its pointwise limit $\eta^{\hat{H}^\infty}$ is $C^\infty$-differentiable, as was to be shown.

{\footnotesize
\newcommand{\mdash}{-}
\bibliographystyle{abbrv}
\bibliography{revision,/home/gtrentin/preprints/bib/gtrentin,/home/gtrentin/preprints/bib/2016a}
}%
\end{document}